\documentclass[11pt]{article}

\usepackage{amsmath, amsthm, amssymb}
\usepackage{booktabs}
\usepackage{url,enumerate}
\usepackage{graphicx}
\usepackage{color}
\usepackage{setspace}
\usepackage{url}
\usepackage[multiple]{footmisc}

\usepackage{bbm}

\usepackage{authblk}

\newtheorem{lemma}{Lemma}
\newtheorem{theorem}{Theorem}
\newtheorem{corollary}{Corollary}

\theoremstyle{definition}
\newtheorem*{conditions}{Basic conditions}

\theoremstyle{definition}
\newtheorem*{condition}{An additional condition}

\theoremstyle{remark}
\newtheorem{remark}{Remark}

\usepackage[]{algorithm2e}

\usepackage[margin=1in]{geometry}

\usepackage[compact]{titlesec}
\titleformat{\subsection}[runin]{\bfseries}{\thesubsection. }{0pt}{}[]
\titleformat{\subsubsection}[runin]{\bfseries}{\thesubsubsection. }{0pt}{}[]
\usepackage{setspace}
\AtBeginDocument{%
  \addtolength\abovedisplayskip{-0.5\baselineskip}%
  \addtolength\belowdisplayskip{-0.5\baselineskip}%
}

\newcommand{\ignore}[1]{}

\newcommand{\jqw}[1]{#1}
\usepackage{soul}

\newcommand{\tp}{\textnormal{\tiny T}} 
\newcommand{\ct}{\textnormal{\tiny H}} 

\makeatletter
\newcommand{\opnorm}{\@ifstar\@opnorms\@opnorm}
\newcommand{\@opnorms}[1]{%
  \left|\mkern-1.5mu\left|\mkern-1.5mu\left|
   #1
  \right|\mkern-1.5mu\right|\mkern-1.5mu\right|
}
\newcommand{\@opnorm}[2][]{%
  \mathopen{#1|\mkern-1.5mu#1|\mkern-1.5mu#1|}
  #2
  \mathclose{#1|\mkern-1.5mu#1|\mkern-1.5mu#1|}
}
\makeatother

\title{Fast randomized iteration: diffusion Monte Carlo through the lens of numerical linear algebra}
\author[1]{Lek-Heng Lim}
\author[1,2]{Jonathan Weare}
\affil[1]{Department of Statistics, University of Chicago}
\affil[2]{James Franck Institute, University of Chicago}
\date{}

\begin{document}

\maketitle

\begin{abstract} 

\jqw{We review the basic outline of the highly successful diffusion Monte Carlo technique commonly used in contexts ranging from electronic structure calculations to rare event simulation and data assimilation, and propose a new class of randomized iterative algorithms based on similar principles to address  a variety of common tasks in numerical linear algebra.  From the point of view of numerical linear algebra, the main novelty of the Fast Randomized Iteration schemes described in this article is  that they work in either linear or constant  cost per iteration (and in total, under appropriate conditions) and are rather versatile: we will show how they apply to
solution of linear systems,  eigenvalue problems, and matrix exponentiation, in dimensions far beyond the present limits of numerical linear algebra. 
While 
traditional iterative methods in numerical linear algebra were created in part to deal with instances where a matrix (of size $\mathcal{O}(n^2)$) is too big to store, the algorithms that we propose are effective even in instances where the solution vector itself (of size $\mathcal{O}(n)$) may be too big to store or manipulate. In fact, our work is motivated by recent DMC based quantum Monte Carlo schemes that have been applied to matrices as large as $10^{108} \times 10^{108}$.
We provide basic convergence results, discuss the dependence of these results on the dimension of the system, and demonstrate dramatic cost savings on a range of test problems.}
\end{abstract}


\section{Introduction}\label{sec:intro}

 Numerical linear algebra has been the cornerstone of scientific computing from its earliest days and randomized approaches to solving problems in linear algebra   have a history almost as long as numerical linear algebra itself (see e.g.\ \cite{AlexandrovLakka:1996:RandInv,DimovDimov:1998:RNLAinv, 
DimovKaraivanova:1998:RNLA,Forsythe:1950:rnla,Halton:1962:smc,Halton:1970:mc,Halton:1994:rnla,Hammersley:1960:mc,HammersleyHandscomb:1964:mc,Wasow:1952:rnla}).\footnote{As pointed out in \cite{GoodmanMadras:1995:RandomWalksandLA} many classical iterative techniques in numerical linear algebra are intimately related to Markov chain Monte Carlo (MCMC) schemes.}  As the size of matrices encountered in typical applications has increased (e.g.\ as we sought greater and greater accuracy in numerical solution of partial differential equations), so has the attention paid to the performance of linear algebra routines on very large matrices both in terms of memory usage and operations count.  Today, in applications ranging from numerical solution of partial differential equations (PDE) to data analysis, we are frequently  faced with the need to solve linear algebraic problems at and beyond the boundary of applicability of classical techniques.  In response,
randomized numerical linear algebra algorithms are receiving renewed attention and, over the last decade, have become an immensely popular subject of study  within the applied mathematics and computer science communities (see e.g.\ \cite{CoakleyRokhlin:2011:OrthogProj,DrineasKannan:2006:MatMult, DrineasKannan:2006:LowRank, DrineasKannan:2006:CompMat, ErikssonBiqueSolbrig:2011:ISforMatMult, Frieze:2004:FMA:1039488.1039494,  LibertyWoolfe:2007:LowRank,MartinssonRokhlin:2011:MatDecomp, RokhlinSzlam:2009:PCA, RokhlinTygert:2008:LeastSquares, Strohmer2008, WoolfeLiberty:2008:MatApprox}).  


The goal of this article is to, after providing a brief introduction to the highly successful \textit{diffusion Monte Carlo} (DMC) algorithm,   suggest a new class of algorithms inspired by DMC for problems in numerical linear algebra.  DMC is used in applications including electronic structure calculations, rare event simulation, and data assimilation, to efficiently approximate 
expectations of the type appearing in Feynman--Kac formulae, i.e., for weighted expectations of Markov processes typically associated with parabolic partial differential equations (see e.g.\ \cite{DelMoral:2004:FK}).  While based on principles underlying DMC, the Fast Randomized Iteration (FRI) schemes that we study in this article are designed to address arguably the most classical and common tasks in matrix computations:
linear systems, eigenvector problems, and matrix exponentiation, i.e., solving for $v$ in
\begin{equation}\label{eq:list}
Av = b, \qquad  Av = \lambda v, \qquad v = \exp(A)b
\end{equation}
for matrices $A$ that might not have any natural association with a Markov process.

FRI schemes rely on basic principles similar to those at the core of other randomized methods that have appeared recently in the numerical linear algebra literature but they differ substantially in detail and in the problems they address. These differences will be remarked on again later, but roughly, while many recent randomized linear algebra techniques  rely on a single sampling step, FRI methods randomize repeatedly and, as a consequence, are more sensitive to errors in the constructed randomizations.
FRI schemes are not, however, the first to employ randomization within iterative schemes (see in particular  \cite{AlexandrovLakka:1996:RandInv} and \cite{Halton:1962:smc}).  In fact the strategy of replacing expensive integrals or sums (without immediate stochastic interpretation) appearing in iterative protocols has a long history in a diverse array of fields.  
For example, it was used in schemes for the numerical solution of hyperbolic systems of partial differential equations in \cite{Chorin:1976:Glimm}.  That strategy is represented today in applications ranging from density functional calculations in physics and chemistry (see e.g.\ \cite{BaerNeuhauserRabani:2013:SDFT}) to maximum likelihood estimation in statistics and machine learning (see e.g.\ \cite{Bottou:2004:SGD}).   Though related in that they rely on repeated randomization within an iterative procedure, these schemes differ from the methods we consider in that they do not use a stochastic representation of the solution vector itself.  
In contrast to these and other randomized methods that have been used in linear algebra applications, our focus is on problems for which the solution vector is extremely large so they can only be treated by linear or constant cost algorithms. In fact, the  scheme that is our primary focus is ideally suited to problems so large that the solution vector itself is too large to store so that no traditional iterative method (even for sparse matrices) is appropriate.  This is possible because our scheme computes only low dimensional projections of the solution vector and not the solution vector itself.  The full solution is replaced by a sequence of sparse random vectors whose expectations are  close to the true solution and whose variances are  small.

Diffusion Monte Carlo (see e.g.\ \cite{Anderson1975,BoothThom:2009:CIQMC,CeperleyAlder1980, 
FoulkesMitas:2001:QMC, GrimmStorer1971,HairerWeare:2014:TDMC, Kalos1962,KolorencMitas:2011:QMC,LopezMa:2006:QMC})     is a central component in the quantum Monte Carlo (QMC) approach to computing the electronic ground state  energy  of the Schr\"odinger--Hamiltonian operator\footnote{The symbol $\Delta$ is used here and below to denote the usual Laplacian operator on functions of  $\mathbb{R}^d,$ $\Delta u = \sum_{i=1}^d \partial^2_{x_i} u.$  $U$ is a potential function that acts on $v$ by pointwise multiplication.  Though this operator is symmetric, the FRI schemes we introduce below are not restricted to symmetric eigenproblems.}
\begin{equation}\label{H}
\mathcal{H} v = -\frac{1}{2}\Delta v + U\,v.
\end{equation}
We are motivated in particular by the work in \cite{BoothThom:2009:CIQMC} in which the authors apply a version of the  DMC procedure to a finite (but huge) dimensional projection of $\mathcal{H}$ onto a discrete basis respecting an anti-symmetry property of the desired eigenfunction.  The approach in \cite{BoothThom:2009:CIQMC} and subsequent papers have yielded remarkable results in situations where the projected Hamiltonian is an  extremely large matrix (e.g.\ $10^{108}\times 10^{108}$, see \cite{ShepherdBooth:2012:FCIQMC6}) and standard approaches to finite dimensional eigenproblems are far from reasonable (see \cite{BoothAlavi:2010:FCIQMC2,BoothCleland:2011:FCIQMC5,BoothGruneis:2013:FCIQMC7,ClelandBooth:2010:FCIQMC3, ClelandBooth:2011:FCIQMC4,ShepherdBooth:2012:FCIQMC6}).

The basic DMC approach is also at the core of schemes developed for a number of applications beyond electronic structure.  In fact, early incarnations of DMC were used  in the simulation of small probability events \cite{HammersleyMorton:1954:SIS, RosenbluthRosenbluth:1955:SIS} for statistical physics models.   Also in statistical physics, the transfer matrix Monte Carlo (TMMC) method was developed to compute the partition functions of certain lattice models by exploiting the observation that the partition function can be represented as the dominant eigenvalue of the so-called transfer matrix, a real, positive, and sparse matrix (see \cite{NightingaleBlote:1986:TMMC}).  TMMC may be regarded as an application of DMC to discrete eigenproblems.
DMC has also become a popular method for many data assimilation problems and the notion of a ``compression'' operation introduced below is very closely related to the ``resampling'' strategies developed for those problems (see e.g.\ \cite{GordonSalmondSmith:1993:IEEProcF,Kitagawa:1996:JCompGraphStat,deFreitasDoucetGordon:2005:book}).


One can view the basic DMC (or MCMC for that matter) procedure as a combination of two steps:  In one step an integral operator (a Green's function) is applied to an approximate solution consisting of a weighted finite sum of delta functions, in another step the resulting function (which is no longer a finite mixture of delta functions) is again approximated by a finite mixture of delta functions.  The more delta functions allowed in the mixture, the higher the accuracy and cost of the method.  
A key to understanding the success of these methods is the observation that not \emph{all} delta functions (i.e., at all positions in space) need appear in the mixture.   A similar observation holds for the methods we introduce:
{FRI schemes need not access all entries in the matrix of interest to yield an accurate solution}.
In fact, we prove that the cost to achieve a fixed level of accuracy with our methods can be bounded independently of the size of the matrix, though in many applications one should expect some dependence on dimension.   As with other randomized schemes, when an effective deterministic method is available it will very likely outperform the methods we propose; our focus is on problems for which satisfactory deterministic alternatives are not available (e.g.\ when the size of the intermediate iterates or final result are so large as to prohibit any conceivable deterministic methods). Moreover, the schemes that we propose are a supplement and not a replacement  for traditional dimensional reduction strategies (e.g.\ intelligent choice of basis).  Indeed, successful application of DMC within practical QMC applications relies heavily on a change of variables based on approximations extracted by other methods (see the discussion of importance sampling in \cite{FoulkesMitas:2001:QMC}). 


The theoretical results that we provide are somewhat atypical of results commonly presented in the numerical linear algebra literature.  In the context of linear algebra applications, both DMC and FRI schemes are most naturally viewed as randomizations of standard iterative procedures and their performance is largely determined by the structure of the particular deterministic iteration being randomized.  For this reason, as well as to avoid obscuring the essential issues with the details of individual cases,  we choose to frame our results in terms of 
the difference between the iterates $v_t$ produced by a general iterative scheme and the iterates generated by the corresponding randomized scheme $V_t$ (rather than considering the difference between $V_t$ and $\lim_{t\rightarrow \infty} v_t$).   In ideal situations \jqw{(see Corollary \ref{nonneg})} our bounds are of the form
\begin{equation}\label{opnormintro}
\opnorm{ V_t - v_t} := \sup_{f\in\mathbb{C}^n,\; \lVert f\rVert_\infty\leq 1}
\sqrt{\mathbf{E}\left[ \lvert f \cdot V_t - f\cdot v_t\rvert^2\right]} \leq \frac{C}{\sqrt{m}}
\end{equation}
where the (in general $t$-dependent) constant $C$ is independent of the dimension $n$ of the problem and $m\leq n$ controls the cost per iteration of the randomized scheme (one iteration of the randomized scheme is roughly a factor of $n/m$ less costly than its deterministic counterpart and the two schemes are identical when $m=n$).  The norm in \eqref{opnormintro} measures the root mean squared deviation in low dimensional projections of the iterates.  This choice  is important as described in more detail in Sections \ref{sec:gen} and \ref{sec:conv}.  For more general applications, one can expect the constant $C,$ which incorporates the stability properties of the randomized iteration, to depend on dimension.  In the worst case scenario, the randomized scheme is no more efficient than its deterministic counterpart (in other words, reasonable performance may require $m\sim n$).  Our numerical simulations, in which $n/m$ ranges roughly between $10^{7}$ and  $10^{14},$ strongly suggest that this scenario may be rare.

We will  begin our development in Section~\ref{sec:DMC} with a description of the basic diffusion Monte Carlo procedure.  Next, in Section~\ref{sec:gen} we describe how ideas borrowed from DMC  can be applied to general  iterative procedures in numerical linear algebra.  
As a prototypical example, we describe how randomization can be used to dramatically decrease the cost of finding the dominant eigenvalue and (projections of) the dominant eigenvector.  Also in that section, we consider the specific case in which the iteration mapping is an $\varepsilon$-perturbation of the identity, relevant to a wide range of applications involving evolutionary differential equations.  In this case a poorly chosen randomization scheme can result in an unstable algorithm while a well chosen randomization can  result in an error that decreases with $\varepsilon$ (over $\varepsilon^{-1}$ iterations).    Next, in Sections~\ref{sec:conv} and \ref{sec:rr}, we establish several simple bounds regarding the stability and  error of our schemes.  
Finally, in Section~\ref{sec:ex} we provide three computational examples to demonstrate the performance of our approach.  A simple, educational implementation of Fast Randomized Iteration applied to our first computational example is available online  (see \cite{FRIforIsingCode}).

\begin{remark}
In several places we have included remarks that clarify or emphasize concepts  that may otherwise be unclear  to readers more familiar with classical, deterministic, numerical linear algebra methods.   We anticipate that some of these remarks will be useful to this article's broader audience as well.
\end{remark}

\section{Diffusion Monte Carlo within quantum Monte Carlo}\label{sec:DMC}

The ground state energy, $\lambda_*,$ of a quantum mechanical system governed by the Hamiltonian in \eqref{H} is the smallest eigenvalue (with corresponding eigenfunction $v_*$) of the Hermitian operator $\mathcal{H}.$  The starting point for a DMC calculation  is the imaginary-time Schr\"odinger equation\footnote{The reader familiar with quantum mechanics but unfamiliar with QMC may wonder why we begin with the imaginary time Schr\"odinger equation and not the usual Schr\"odinger equation $i \partial_t v = -  \mathcal{H}v$.  The reason is that while   the solutions to both equations can be expanded in terms of the eigenfunctions of $\mathcal{H},$ for the usual Schr\"odinger equation the contributions to the solution from eigenfunctions with larger eigenvalues do not decay relative to the ground state.  By approximating the solution to the imaginary time equation for large times we can approximate the ground state eigenfunction of $\mathcal{H}.$}\footnote{ In practical QMC applications one solves for $\rho = v_* \tilde v$ where $\tilde v$ is an approximate solution found in advance by other methods.  The new function $\rho$ is the ground state eigenfunction of a Hamiltonian of the form $\tilde {\mathcal{H}}v = -\frac{1}{2}\Delta v + \text{div} (b v) + \tilde U v.$  The implications for the discussion in this section are minor. }
\begin{equation}\label{iSE}
\partial_t v = - \mathcal{H} v
\end{equation}
(for a review of QMC see \cite{FoulkesMitas:2001:QMC}).
One can, in principle, use power iteration to find $\lambda_*$:  beginning from an initial guess $v_0$ (and assuming a gap between $\lambda_*$ and the rest of the spectrum of $\mathcal{H}$), the iteration
\begin{equation}\label{powerforH}
\lambda_t =  - \frac{1}{\varepsilon} \log \int  e^{-\varepsilon \mathcal{H}} v_{t-1} (x) \, dx
\quad
\text{and}
\quad
v_t = \frac{e^{-\varepsilon \mathcal{H}} v_{t-1}}{\int e^{-\varepsilon \mathcal{H}} v_{t-1}(x)\, dx}
\end{equation}
will converge to the pair $(\lambda_*, v_*)$ where $v_*$ is the eigenfunction corresponding to $\lambda_*.$ 
Here the integral is over $x\in \mathbb{R}^d.$

\begin{remark}
For readers who are more familiar with the power method in numerical linear algebra, this may seem a bit odd but a discrete analogue of \eqref{powerforH} is just $v_{t} = A v_{t-1}/\lVert A v_{t-1} \rVert_1$ applied to a positive definite matrix $A = \exp(-\varepsilon H) \in \mathbb{C}^{n \times n}$ where $H$ is Hermitian.\footnote{Note that the matrix  $H$ is not the same as the operator $\mathcal{H}$ on a function space,  and is only introduced for the purposes of relating the expression in \eqref{powerforH} to the usual power method for matrices} The slight departure from the usual power method is only in (i) normalizing by a $1$-norm (or rather, by the sum of entries $\mathbbm{1}^\tp A v_{t-1}$ since both $v$ and $A$ are non-negative), and  (ii) iterating on  $\exp(-\varepsilon H) $ instead of on $H$ directly (the goal in this context is to find the smallest eigenvalue of $H$ not the magnitude-dominant eigenvalue of $H$). The iteration on the eigenvalue is then $\lambda_t = -\epsilon^{-1} \log \lVert A v_{t-1} \rVert_1$ since $\lambda_*(H) = -\epsilon^{-1} \log \lambda_*(A)$. 
\end{remark}

The first step in any (deterministic or stochastic) practical implementation of \eqref{powerforH} is discretization of the operator $e^{-\varepsilon \mathcal{H}}.$  Diffusion Monte Carlo often uses the second order time discretization
\[
e^{-\varepsilon \mathcal{H}} \approx K_\varepsilon = e^{-\frac{\varepsilon}{2} U} e^{\frac{\varepsilon}{2} \Delta }
e^{-\frac{\varepsilon}{2} U}.
\]
A standard deterministic approach would then proceed by discretizing the operator $K_\varepsilon$ in space and replacing $e^{-\varepsilon \mathcal{H}}$ in \eqref{powerforH} with the space and time discretized approximate operator.  The number of spatial discretization points required by such a scheme to achieve a fixed accuracy will, in general, grow exponentially in the dimension $d$ of $x.$

Diffusion Monte Carlo uses two randomization steps to avoid this explosion in cost as $d$ increases.  These randomizations have the effect of ensuring that the random approximations $V_t^m$ of the iterates $v_t$ are always of the form
\[
V_t^m(x) =  \sum\nolimits_{j=1}^{N_t}  W_t^{(j)} \delta_{X_t^{(j)}}(x)
\]
where $\delta_y(x)$ is the Dirac delta function centered at $y\in \mathbb{R}^d,$  the $W_t^{(j)}$ are real, non-negative numbers with $\mathbf{E}\bigl[\sum_{j=1}^{N_t} W_t^{(j)}\bigr]=1,$ and, for each $j\leq N_t,$ $X_t^{(j)}\in \mathbb{R}^d.$   As will be made clear in a moment, the integer $m$ superscripted  in our notation controls  the number of delta functions, $N_t,$ included in the above expression for $V_t^m.$ 

 The fact that the function $V_t^m$ is non-zero at only $N_t$ values is crucial to the efficiency of diffusion Monte Carlo.  Starting with $N_0 =m$ and from an initial condition of the form
\[
V_0^m  = \frac{1}{m} \sum\nolimits_{j=1}^{m} \delta_{X_0^{(j)}}
\]
 the first factor of $e^{-\frac{\varepsilon}{2} U}$ applied to $V_0^m$ results in
\[
\frac{1}{m} \sum\nolimits_{j=1}^{m} e^{-\frac{\varepsilon}{2} U(X_0^{(j)})}\delta_{X_0^{(j)}}
\]
which can be assembled in $\mathcal{O}(m)$ operations.  The first of the randomization steps used in DMC relies on the well known relationship
\begin{equation}\label{BMheat}
\int f(x) [e^{\frac{\varepsilon}{2} \Delta} \delta_y](x) \, dx = \mathbf{E}_{y} \left[ f(B_\varepsilon)\right]
\end{equation}
where $f$ is a test function, $B_s$ is a standard Brownian motion evaluated at time $s\geq 0,$ and the subscript on the expectation is meant to indicate that $B_0=y$ (i.e., in the expectation in \eqref{BMheat} $B_\varepsilon$ is a Gaussian random variable with mean $y$ and variance $\varepsilon$). In fact, this representation is a special case of the Feynman--Kac formula 
$
  \int f(x) [e^{-\varepsilon \mathcal{H}} \delta_y ](x) \, dx = 
  \mathbf{E}_{y} \bigl[ f(B_\varepsilon)e^{-\int_0^\varepsilon U(B_s) \,ds} \bigr]
$.
Representation \eqref{BMheat} suggests the approximation 
\[
K_\varepsilon V_0^m \approx \tilde V_{1}^m = \frac{1}{m} \sum\nolimits_{j=1}^{m} e^{-\frac{\varepsilon}{2}\bigl(U(\xi_{1}^{(j)})+  U(X_0^{(j)})\bigr)}\delta_{\xi_{1}^{(j)}}
\]
where, conditioned on the $X_0^{(j)},$ the $\xi_{1}^{(j)}$ are independent and $\xi_{1}^{(j)}$ is normally distributed with  mean $X_0^{(j)}$ and covariance $\varepsilon I$   (here $I$ is the $d\times d$ identity matrix).  This first randomization has allowed an approximation of $K_\varepsilon V_0^m$ by a distribution, $\tilde V_{1}^m,$  that is again supported on only $m$ points in $\mathbb{R}^d.$   One might, therefore, attempt to define a sequence   $V_t^m$ iteratively by the recursion
\[
V_{t+1}^m = \frac{ \tilde V_{t+1}^m }{ \int  \tilde V_{t+1}^m(x) \, dx}
=  \sum\nolimits_{j=1}^m W^{(j)}_{t+1}\, \delta_{\xi_{t+1}^{(j)}}
\]
where we have recursively defined the weights
\[
W_{t+1}^{(j)} = \frac{ e^{-\frac{\varepsilon}{2}\bigl(U(\xi_{t+1}^{(j)})+  U(X_t^{(j)})\bigr)}W_t^{(j)}}
{ \sum\nolimits_{\ell =1}^m e^{-\frac{\varepsilon}{2}\bigl(U(\xi_{t+1}^{(\ell)})+  U(X_t^{(\ell)})\bigr)}W_t^{(\ell)} }
\]
with $W_0^{(j)} = 1/m$ for each $j.$
The cost of this randomization procedure is  $\mathcal{O}(dm)$ so that the total cost of a single iteration is $\mathcal{O}(dm).$  

At each step the weights in the expression for the iterates $V_t^m$ are multiplied by additional random factors.  These factors are determined by the potential $U$ and the positions of the $\xi^{(j)}_t.$  On the other hand, 
the $\xi_t^{(j)},$ evolve without reference to the potential function $U$ (they are discretely sampled points from $m$ independent Brownian motions).  As a consequence, over many iterations one can expect extreme relative variations in the $W_t^{(j)}$ and, therefore, poor accuracy in $V_t^m$ as an approximation of the functions $v_t$ produced by \eqref{powerforH}.

The second randomization employed by the DMC algorithm is the key to controlling the growth in variance and generalizations of the idea will be key to designing fast randomized iteration schemes in the next section.
In order to control the variation in weights, at step $t,$ DMC randomly removes points $\xi^{(j)}_t$ corresponding to small weights $W^{(j)}_t$ and duplicates points corresponding to large weights.   The resulting number of points stored at each iteration, $N_t,$ is close to, or exactly, $m.$
At step $t,$ a new distribution $Y_t^m$ is generated from $V_t^m$ so that
\[
\mathbf{E}\left[ Y_t^m \mid V_t^m\right] = V_t^m
\]
by ``resampling'' a new collection of $N_{t+1}$ points from the $N_t$ points $\xi_t^{(j)}$ with associated probabilities $W_t^{(j)}.$   
The resulting points are labeled $ X_t^{(j)}$ and the new distribution $Y_t^m$ takes the form
\[
Y_t^m   = \frac{1}{N_{t}} \sum\nolimits_{j=1}^{N_{t+1}}   \delta_{X^{(j)}_t}.
\]
The next iterate $V_{t+1}^m$ is then built exactly as before but with $V_t^m$ replaced by $Y_t^m.$
All methods to select the $X_t^{(j)}$ generate, for each $j,$ a non-negative integer $N^{(j)}_t$ with 
\[
\mathbf{E}\bigl[ N^{(j)}_t \bigm| \bigl\{W_t^{(\ell)}\bigr\}_{\ell=1}^m\bigr] = m W_t^{(j)}
\]
and then sets $N^{(j)}_t$ of the elements in the collection $\{ X_t^{(j)}\}_{j=1}^{N_{t+1}}$
equal to $\xi^{(j)}_t$ so that $N_{t+1} = \sum_{j=1}^{N_t} N^{(j)}_t.$  
For example, one popular strategy in DMC generates the $N^{(j)}_t$ independently with 
\begin{equation}\label{Ndef}
\begin{aligned}
\mathbf{P}\bigl[ N^{(j)}_t&=  \bigl\lfloor m W_t^{(j)} \bigr\rfloor \bigr] =\bigl\lceil m W_t^{(j)}\bigr\rceil - m W_t^{(j)},\\
\mathbf{P}\bigl[ N^{(j)}_t&=  \bigl\lceil m W_t^{(j)} \bigr\rceil \bigr] = m W_t^{(j)} - \bigl\lfloor  m W_t^{(j)}\bigr\rfloor .
\end{aligned}
\end{equation}



 The above steps define a randomized iterative algorithm to generate approximations $V_t^m$ of $v_t.$  The second randomization procedure (generating $Y_t^m$ from $V_t^m$) will typically require $\mathcal{O}(m)$ operations, preserving the overall $\mathcal{O}(dm)$ per iteration cost of DMC (as we have described it).  The memory requirements of the scheme are also $\mathcal{O}(dm).$
The eigenvalue $\lambda_*$ can be approximated, for example, by
\[
-\frac{1}{\varepsilon} \log \left( \frac{1}{T} \sum\nolimits_{t=1}^T \frac{1}{N_t}\sum\nolimits_{j=1}^{N_t} 
e^{- \frac{\varepsilon}{2}\left(U(\xi_{t+1}^{(j)}) +  U(X ^{(j)}_t)\right) }
\right)
\]
for $T$ large.  


Before moving on to more general problems notice that the scheme just outlined applies just as easily to space-time discretizations of $e^{-\varepsilon \mathcal{H}}.$  For example, if we set $h =\sqrt{(1+2d)\varepsilon}$ and denote by $\mathcal{E}_h^d\subset \mathbb{R}^d$ the  uniform rectangular grid with resolution $h$ in each direction,  the operator $e^{-\varepsilon \mathcal{H}}$ can be discretized using 
\[
e^{-\varepsilon \mathcal{H}} \approx K_{\varepsilon,h} =  e^{-\frac{\varepsilon}{2}U} e^{\frac{\varepsilon}{2} \Delta_h} e^{-\frac{\varepsilon}{2} U}
\]
where, for any vector $g\in \mathcal{E}_h^d$,
\[
\Delta_h g(x) = \frac{1}{\varepsilon}\left(-g(x) + \frac{1}{1+2d} \sum\nolimits_{y\in \mathcal{E}_h^d, \; \lVert y-x\rVert_2 \leq h } g(y)  \right)
\]
(here we find it convenient to identify functions $g:\mathcal{E}_h^d\rightarrow \mathbb{R}$ and vectors in $\mathbb{R}^{\mathcal{E}_h^d}$).  The operator $e^{-\varepsilon \Delta_h}$ again has a stochastic representation; now the representation is in terms of a jump Markov process with jumps from a point in $\mathcal{E}_h^d$ to one of its nearest neighbors on the grid (for an interesting approach to discretizing stochastic differential equations taking advantage of a similar observation see \cite{BouRabeeVandenEijnden:2015:SDEbySSA}). 
 
\begin{remark}
The reader should notice that not only will we be unable to store the matrix $K_{\varepsilon,h}$ (which is exponentially large in $d$) or afford to compute $K_{\varepsilon,h} v$ for a general vector $v,$ but we will not even be able to store the iterates $v_t$ generated by the power method.  Even the sparse matrix routines developed in numerical linear algebra to deal with large   matrices are not reasonable for this problem. 
\end{remark}  

In this discrete context, a direct application of the DMC approach 
(as in \cite{NightingaleBlote:1986:TMMC}) would represent the solution vector as a superposition of standard basis elements\footnote{The delta functions represent the indices of $v_t$ that we are keeping track of; $\delta_\xi$ is more commonly denoted $\mathbf{e}_\xi$ in numerical linear algebra --- the standard basis vector with $1$ in the $\xi$th coordinate and zero elsewhere.} and replace calculation of $K_{\varepsilon,h}v$ by a random approximation whose cost is (for this particular problem) free of any direct dependence on the size of $K_{\varepsilon,h}$ (though its cost can depend on structural properties of $K_{\varepsilon,h}$ which may be related to its size), whose expectation is exactly $K_{\varepsilon,h} v,$ 
and whose variance is small.   The approach in \cite{BoothThom:2009:CIQMC} is also an application of these same basic DMC steps to a discrete problem, though in that case the desired eigenvector has entries of {\it a priori}
 unkown sign, requiring that the solution vector be represented by a superposition of signed standard basis elements.

In this article we take a different approach to adapting DMC to discrete problems.  Instead of reproducing in the discrete setting exactly the steps comprising DMC, consider a slightly modified scheme that omits direct randomization of an approximation to $e^{\varepsilon \Delta_h},$ and instead relies solely on a general random mapping $\Phi^m_t$ very similar to the map from $V_t^m$ to $Y_t^m$ but which takes a vector $V_t^m \in \mathbb{R}^{\mathcal{E}_h^d}$ with non-negative entries and  $\lVert V_t^m\rVert_1 = 1 $ (i.e., a probability measure on $\{1,2,\dots, \lvert \mathcal{E}_h^d\rvert\}$) and produces a new random vector $Y_t^m$ with  $m,$ or nearly $m,$  non-zero components and satisfying $\mathbf{E}\left[ Y_t^m \mid V_t^m\right] = V_t^m$ as above.  
Starting from a non-negative initial vector $V_0^m\in \mathcal{E}_h^d$  with $\lVert V_0^m\rVert_1 =1$ and with at most $\mathcal{O}(md)$ non-zero entries,   $V_{t+1}^m$ is generated from $V_t^m$ as follows:
\begin{enumerate}[\textsc{Step} 1.]
\item  Generate $Y_t^m = \Phi^m_{t}\left(V_{t}^m\right)$ with approximately or exactly $m$ non-zero entries.
\item  Set
$V_{t+1}^m = \dfrac{  K_{\varepsilon,h} Y_t^m}{\lVert K_{\varepsilon,h} Y_t^m \rVert_1}$.
\end{enumerate}
Just as at iteration $t,$ DMC produces a random approximation of the result of $t$ iterations of  power iteration for the infinite dimensional integral operator $e^{-\varepsilon \mathcal{H}},$ the above steps produce a random approximation of the result of $t$ iterations of the power iteration for the matrix $K_{\varepsilon,h}.$
The improved efficiency of DMC is due to the application of the integral operator to  a finite sum of delta functions in place of a more general function.  Similarly, the efficiency of the finite dimensional method in the above two step procedure is a consequence of replacement of a general vector $v$ in the product $K_{\varepsilon,h} v$ by a sparse approximation, $\Phi^m_t(v).$   

For the random mapping $\Phi^m_t(v)$ we might, for example,  adapt the popular resampling choice mentioned above and choose the entries of $Y_t^m$ independently with 
\begin{equation}\label{Phi1}
\begin{aligned}
\mathbf{P}\bigl[ (Y_t^m)_j &=  \bigl\lfloor  (V_t^m)_j \,m \bigr\rfloor/m \bigr] =\bigl\lceil  (V_t^m)_j\, m\bigr\rceil - (V_t^m)_j\,m, \\
\mathbf{P}\bigl[ (Y_t^m)_j &=  \bigl\lceil (V_t^m)_j\, m\bigr\rceil/m \bigr] =  (V_t^m)_j\, m - \bigl\lfloor  (V_t^m)_j \,m\bigr\rfloor .
\end{aligned}
\end{equation}
\noindent \jqw{Note that this rule results in a vector $Y_t^m$ with expectation exactly equal to $V_t^m.$  On the other hand, when the number of non-zero entries in $V_t^m$ is large, many of those entries must be less than $1/m$ (because $\lVert V_t^m\rVert_1 = 1$) and will have some probability of being set equal to zero in $Y_t^m.$ 
 In fact, the number of non-zero entries in $Y_t^m$ has expectation and variance bounded by $m.$}  
 The details of the mappings $\Phi_t^m,$ which we call compression mappings, will be described later in Section~\ref{sec:rr} where, for example, we will find that the cost of applying the mapping $\Phi^m_t$ will typically be $\mathcal{O}(n)$ when its argument has $n$ non-zero entries (in this setting $n=\mathcal{O}(md)$). And while the cost of applying $K_{\varepsilon,h}$ to an arbitrary vector in $\mathbb{R}^{\mathcal{E}_h^d}$ is $\lvert \mathcal{E}_h^d\rvert,$ the cost of applying $K_{\varepsilon,h}$ to a vector with $m$ non-zero entries is only $\mathcal{O}(md).$  The total cost of the scheme per iteration is therefore $\mathcal{O}(md)$ in storage and operations.   These cost requirements are dependent on the particular spatial discretization of $e^{\varepsilon \Delta};$ if we had chosen a discretization corresponding to a dense matrix $K_{\varepsilon,h}$ then the cost reduction would be much less extreme.  Nonetheless, as we will see in Section \ref{sec:gen}, the scheme just described can be easily generalized and, as we will see in Sections \ref{sec:conv} and \ref{sec:ex}, will often result in methods that are significantly faster than their deterministic counterparts.    

\begin{remark}
Rather than focusing on sampling indices of entries in a vector $v,$ as is typical of some of the literature on randomized numerical linear algebra, we focus on constructing an accurate sparse representation of $v.$  This is primarily a difference of perspective, but has consequences for the accuracy of our randomizations.  For example, the techniques in \cite{DimovDimov:1998:RNLAinv, 
DimovKaraivanova:1998:RNLA} and in \cite{DrineasKannan:2006:MatMult, DrineasKannan:2006:LowRank, DrineasKannan:2006:CompMat} would correspond, in our notation and context, to setting for $v\in \mathbb{R}^n$
\begin{equation}\label{multiPhi}
\Phi_t^m(v) = \frac{\lVert v \rVert_1}{m} \sum_{j=1}^n \frac{v_j}{\lvert v_j\rvert}N_t^{(j)} \mathbf{e}_j
\end{equation}
where $\mathbf{e}_j$ is the $j$th standard basis element and the random vector
\[
\bigl(N_t^{(1)},N_t^{(2)},\dots,N_t^{(n)}\bigr)\sim \operatorname{\textsc{Multinomial}}(m, p_1,\dots,p_n)
\]
with 
$p_j = \lvert v_j\rvert / \lVert v\rVert_1.$  As for all Monte Carlo methods, the sparse characteristic of this representation is responsible for gains in efficiency.  And, when error is measured by the norm in \eqref{opnormintro}, only a random sparse representation can be accurate for general $v.$  But random index selection yields only one of many possible random sparse representations of $v$ and not one that is particularly accurate.  In fact, \emph{effective fast randomized iteration schemes of the type introduced in this paper cannot be based solely on random index selection as in \eqref{multiPhi}.}  In the setting of this section, if we were to use \eqref{multiPhi} in place of the rule in \eqref{Phi1} the result would be an unstable scheme (the error would become uncontrollable as $\varepsilon$ is decreased). As we will see in Section \ref{sec:rr}, much more accurate sparse representations are possible.
\end{remark}

Even restricting ourselves to the quantum Monte Carlo context, there is ample motivation to generalize the DMC scheme.    Often one wishes to approximate not the smallest eigenvalue of $\mathcal{H}$ but instead the smallest eigenvalue corresponding to an antisymmetric (in exchange of particle positions) eigenfunction.  DMC as described in this section, cannot be applied directly to computing this value, a difficulty commonly referred to as the Fermion sign problem.  Several authors have attempted to address this issue with various modifications of DMC.  In particular, Alavi and coworkers recently developed a version of DMC for a particular spatial disretization of the Hamiltonian (in the configuration interaction basis) that exactly preserves antisymmetry (unlike the finite difference discretization we just described).  Run to convergence, their method provides the same approximation as the so called full CI method but can be applied with a much larger basis (e.g.\ in experiments by Alavi and coworkers reported in \cite{ShepherdBooth:2012:FCIQMC6} up to $10^{108}$ total functions in the expansion of the solution).   Though the generalizations of DMC represented by the two step procedure in the last paragraph and by the developments in the next section are motivated by the scheme proposed in \cite{BoothThom:2009:CIQMC}, they differ substantially in their details and can be applied to a wider range of problems (including different discretizations of $\mathcal{H}$).  Finally we remark that, while we have considered DMC in the particular context of computing the ground state energy of a Hamiltonian, the method is used for a much wider variety of tasks with only minor modification to its basic structure.  For example, particle filters (see e.g.\ \cite{deFreitasDoucetGordon:2005:book}) are an application DMC to on-line data assimilation and substantive differences are mostly in the interpretation of the operator to which DMC is applied (and the fact that one is typically interested in the solution after finitely many iterations).

\section{A  general framework}\label{sec:gen}

Consider the general iterative procedure,   
\begin{equation}\label{Mit}
v_{t+1} = \mathcal{M}(v_t)
\end{equation}
for $v_t \in \mathbb{C}^n$.  Eigenproblems, linear systems, and matrix exponentiation can each be accomplished by versions of this iteration.  In each of those settings the cost of evaluating $\mathcal{M}(v)$ is dominated by a matrix-vector multiplication.  We assume that the cost (in terms of floating point operations and storage) of performing the required matrix-vector multiplication makes it impossible to carry out recursion \eqref{Mit} to the desired precision.  As in the steps described at the end of the last section, we will consider the error resulting from replacement of \eqref{Mit} by 
\begin{equation}\label{Rit}
V^m_{t+1} = \mathcal{M}\left(\Phi^m_t(V^m_t)\right).
\end{equation}
where the compression maps 
$\Phi^m_t: \mathbb{C}^n \rightarrow \mathbb{C}^n$ are independent,  inexpensive to evaluate, and enforce sparsity in the $V_t^m$ iterates (the number of non-zero entries in $V_t^m$ will be $\mathcal{O}(m)$) so that $\mathcal{M}$ can be evaluated at much less expense.  
When $\mathcal{M}$ is a perturbation of identity \emph{and} an $\mathcal{O}(n)$ scheme is appropriate (see Sections~\ref{sec:pert} and \ref{sec:conv} below) we will also consider the scheme, 
\begin{equation}\label{Ritpert}
V^m_{t+1} = V_t^m + \mathcal{M}\left(\Phi^m_t(V^m_t)\right) - \Phi^m_t(V^m_t).
\end{equation}
  The compressions $\Phi^m_t$ will satisfy (or very nearly satisfy) the statistical consistency criterion
\begin{equation*}
\mathbf{E}\left[  \Phi^m_t(v) \right] = v
\end{equation*}
and will have to be carefully constructed to avoid instabilities and yield effective methods. For definiteness one can imagine that $\Phi^m_t$ is defined by a natural extension of   \eqref{Phi1} 
\begin{equation}\label{Phi2}
\begin{aligned}
\mathbf{P}\left[ (\Phi_t^m(v))_j = \frac{v_j}{m \lvert v_j\rvert} \left\lfloor  \frac{m \lvert v_j\rvert}{ \lVert v\rVert_1} \right\rfloor \right] &= \left\lceil  \frac{m \lvert v_j\rvert}{ \lVert v\rVert_1} \right\rceil - \frac{m \lvert v_j\rvert}{\lVert v\rVert_1},\\
\mathbf{P}\left[ (\Phi_t^m(v))_j = \frac{v_j}{m \lvert v_j\rvert} \left\lceil \frac{m \lvert v_j\rvert}{ \lVert v\rVert_1} \right\rceil \right] &= \frac{m \lvert v_j\rvert}{\lVert v\rVert_1}-\left\lfloor  \frac{m \lvert v_j\rvert}{ \lVert v\rVert_1} \right\rfloor 
\end{aligned}
\end{equation}
to accept arguments $v\in \mathbb{C}^n$ ($V_t^m$ is no longer a non-negative real number). This choice has several drawbacks, not least of which is its cost, and \emph{we do not use it in our numerical experiments}.
Alternative compression schemes, including the one used in our numerical simulations, are considered in detail in Section~\ref{sec:rr}. There we will learn that one can expect that, for any pair $f,v \in \mathbb{C}^n,$
\begin{equation}\label{Cerror1}
\sqrt{\mathbf{E}\left[ \lvert f^\ct \Phi^m_t(v) - f^\ct v\rvert^2\right]} \leq \frac{2}{\sqrt{m}}\lVert f\rVert_\infty \lVert v\rVert_1
\end{equation}
(the superscript ${}^\ct$ is used   throughout this article to denote the conjugate transpose of a vector with complex entries).
These errors are introduced at each iteration and need to be removed to obtain an accurate estimate.  Depending on the setting, we may rely on averaging over long trajectories, averaging over parallel simulations (replicas), or dynamical self-averaging (see Sections~\ref{sec:pert} and \ref{sec:conv}), to remove the noise introduced by our randomization procedure.  Because the specific choice of $\mathcal{M}$ and the form of averaging used to remove noise can differ substantially by setting, we will describe the schemes within the context of specific (and common) iterative procedures.  


\subsection{The eigenproblem revisited}\label{sec:eig}
Consider, for example, a more general eigenproblem than the one we considered in Section~\ref{sec:DMC}.
Given $K \in \mathbb{C}^{n \times n}$ the goal is to determine $\lambda_* \in \mathbb{C}$ and $v_* \in \mathbb{C}^n$ such that
\begin{equation}\label{eigeq1}
K v_* = \lambda_* v_* 
\end{equation}
and such that, for any other solution pair $(\lambda, v)$, $|\lambda| < |\lambda_*|$.  In what  follows in this section we will assume that this problem has a unique solution.
The standard methods of approximate solution of \eqref{eigeq1} are variants of the power method, a simple version of which performs
\begin{equation}\label{pow1}
v_{t+1} = \frac{K v_{t}}{\lVert K v_{t}\rVert_1},\qquad
\lambda_{t+1} = \frac{u^\ct K v_{t}}{u^\ct v_{t}}
\end{equation}
where $u\in \mathbb{C}^n$ is chosen by the user.
Under generic conditions, these converge to the correct $(\lambda_*,v_*)$ starting from an appropriate initial vector $v_0$
 (see e.g.\ \cite{Demmel:1997:NLA}). The scheme in \eqref{pow1} requires $\mathcal{O}(n^2)$ work per iteration and at least $\mathcal{O}\left(n\right)$ storage. 
In this article, we are interested in situations in which  these cost and storage requirements are unreasonable.  

\paragraph{From $\mathcal{O}\left(n^2\right)$ to $\mathcal{O}\left(n m\right)$}

For the iteration in \eqref{pow1} the randomized scheme \eqref{Rit} (along with an approximation of $\lambda_*$) becomes
\begin{equation}\label{random2}
\begin{aligned}
V_{t+1}^m &= \frac{K \Phi^m_t(V^m_t)}{\lVert K \Phi^m_t(V^m_t)\rVert_1}, & \Lambda_{t+1}^m &= \frac{u^\ct   K \Phi^m_t(V^m_t)}{u^\ct V_t^m},\\
\overline{V}_{t}^m &= \frac{1}{t} \sum_{s=1}^t   V_{s}^m, \qquad &\overline{\Lambda}_{t}^m , & = \frac{1}{t}\sum_{s=1}^t \Lambda_{t}^m
\end{aligned}
\end{equation}
where $\overline{V}_t^m$ and $\overline{\Lambda}_t^m$ are trajectory averages estimating $v_*$ and $\lambda_*.$
 In \eqref{random2}, the compressions $\Phi^m_t$ are independent of one another.  Using the rules defining $\Phi^m_t$ in Section~\ref{sec:rr}, construction of $\Phi^m_t(V^m_t)$ at each step will require $\mathcal{O}(n)$ operations.  Since multiplication of the vector $\Phi^m_t(V^m_t)$ by a dense matrix $K$ requires $\mathcal{O}(nm)$ operations, this scheme has  $\mathcal{O}(n m)$ cost and $\mathcal{O}(n)$ storage per iteration requirement.  

\jqw{ 
Iteration \eqref{Ritpert} on the other hand replaces \eqref{pow1} with  
 \begin{equation}\label{random3}
V_{t+1}^m =V_t^m + \left( \frac{K \Phi^m_t(V^m_t)}{\lVert K \Phi^m_t(V^m_t)\rVert_1} - \Phi^m_t(V^m_t)\right), \quad \Lambda_{t+1}^m =
 \frac{u^\ct  K  \Phi^m_t(V^m_t)}{u^\ct V_t^m}.
\end{equation}
By the same arguments as above, this iteration will also have cost and storage requirements of   $\mathcal{O}(n m)$ and $\mathcal{O}(n)$ respectively.  
When $K = I + \varepsilon A$ for some matrix $A$ and small parameter $\varepsilon>0,$ the iteration in \eqref{random3} bears strong resemblance to the Robbins--Monro stochastic approximation algorithm \cite{Robbins:1951:SA, KushnerYin:2003:SA}.  In fact, as we will see in Section~\ref{sec:conv}, when the mapping $\mathcal{M}$ is of the form $ v + \varepsilon b(v)$ the convergence of methods of the form in \eqref{Rit} and \eqref{Ritpert} is reliant on the self-averaging phenomenon also  at the heart of stochastic approximation.    We will also learn that for $\mathcal{M}$ of this form one can expect the error corresponding to \eqref{Ritpert} to be smaller than the error  corresponding to \eqref{Rit}. }

\paragraph{From $\mathcal{O}(n m)$ to $\mathcal{O}(m)$}
For many problems even $\mathcal{O}(n m)$  cost and storage requirements are unacceptable.  This is the case, for example, when $n$ is so large that a vector of length $n$ cannot be stored.  But now suppose that $K$ is sparse with at most $q$ non-zero entries per column.  Because $\Phi^m_{t-1}(V^m_{t-1})$ has $\mathcal{O}(m)$ non-zero entries, the product $K\Phi^m_{t-1}(V^m_{t-1})$ (and hence also $V_t^m$) has at most $\mathcal{O}(qm)$ entries and requires $\mathcal{O}(qm)$ operations to assemble.  On the other hand, if $V^m_t$ has at most $\mathcal{O}(qm)$ non-zero entries, then application of $\Phi^m_t$ to $V^m_t$ requires only $\mathcal{O}(qm)$ operations.  Consequently, as long as $V^m_0$ has at most $\mathcal{O}(qm)$ non-zero entries, the total number of floating point operations required by \eqref{random2} reduces to $\mathcal{O}\left(qm\right)$ per iteration.
 This observation does not hold for methods of the form  \eqref{Ritpert} which will typically result in dense iterates $V_t^m$ and a cost of $\mathcal{O}(n)$ even when $K$ is sparse.

As we have mentioned (and as was true in Section~\ref{sec:DMC}), in many settings even storing the full solution vector is impossible.  Overcoming this impediment  requires a departure from the usual  perspective of numerical linear algebra.   Instead of trying to approximate all entries of $v_*,$ our goal becomes to compute
\[
f_* = f^\ct v_*
\]
for some some vector (or small number of vectors) $f.$  This change in perspective is reflected in the form of our compression rule error estimate in \eqref{Cerror1} and in the form  of our convergence results in Section~\ref{sec:conv} that measure error  in terms of dot products with test vectors as in   \eqref{opnormintro} above. 
As discussed in more detail in Section \ref{sec:conv}, the choice of error norm in \eqref{opnormintro} is essential to our eventual error estimates.   Indeed, were we to estimate a more standard quantity such as 
\[
\mathbf{E}\left[ \lVert V^m_t - v_t \rVert_1\right]
\]
we would find that the error decreased proportional to $(n-m)/n$ requiring that $m$ increase with $n$ to achieve fixed accuracy.
  The algorithmic consequence of our focus on computing low dimensional projections of $v_*$  is simply  the removal in \eqref{random2} of the equation defining $\overline V_t^m$ and insertion of 
\begin{equation}\label{Fitr}
F_{t}^m = f^\ct V^m_{t}  \quad\text{and}\quad \overline{F}_{t}^m = \frac{1}{t}\sum_{s=1}^t F_s^m
\end{equation}
which produces an estimate $F_t^m$ of $f_*.$  

\begin{remark}
While estimation of $f_*$ may seem an unusual goal in the context of classical iterative schemes it is completely in line with the goals of any Markov chain Monte Carlo scheme which typically seek only to compute averages with respect to the invariant measure of a Markov chain and not to completely characterize that measure.
\end{remark}

\jqw{Schemes with $\mathcal{O}\left(m\right)$ storage and operations requirements per-iteration can easily be designed for any general matrix.  Accomplishing this for a dense matrix  requires an additional randomization in which columns of $K$ (or of some factor of $K$) are randomly set to zero independently at each iteration, e.g.\ again in the context of power iteration, assuming that $V_{t-1}^m$ has at most $\mathcal{O}(m)$ non-zero entries, one can use
\begin{equation}\label{altrand}
V_{t+1}^m = \frac{Y_{t+1}^m}{\lVert Y_{t+1}^m\rVert_1} \quad\text{with}\quad Y_{t+1}^m = {\sum_{j=1}^n  \left(\Phi_t^m(V_t^m)\right)_j \Phi_t^{m^j_t,j}\left(  K_j \right)}
\end{equation}
in place of \eqref{random2}, 
where here $K_j$ is used to denote the $j$th column of $K$ and each $\Phi_{t}^{m^j_t,j}$ is an independent copy of $\Phi_{t}^{m^j_t}$ which are assumed independent of $\Phi_t^m.$  The number of entries retained in each column is controlled by $m^j_t$ which can, for example,  be set to
\[
m^j_t = \left\lceil \frac{\lVert K_j \rVert_1 \lvert \left(V_{t}^m\right)_j\rvert}{\sum_{\ell=1}^n \lVert K_\ell \rVert_1 \lvert \left(V_{t}^m\right)_\ell\rvert}\, m\right\rceil \quad\text{or}\quad m^j_t = \left\lceil\lvert \left(V_{t}^m\right)_j \rvert \, m\right\rceil
\]   at each iteration and the resulting vector can then be compressed so that it has exactly or approximately $m$ non-zero entries.    Use of \eqref{altrand} in place of \eqref{random2}  will result in a scheme whose cost per iteration is independent of $n$   if the compressions of the columns have cost independent of $n.$  This may be possible without introducing significant error, for example, when  the entries in the columns of $K$ can take only a small number of distinct values.  Notice that one obtains the update in \eqref{random2} from \eqref{altrand} by removing the compression of the columns.  Consequently, given $V_t^m,$ the conditional variance of $V_{t+1}^m$ generated by \eqref{altrand} will typically exceed the conditional variance resulting from \eqref{random2}.  
}

\subsection{Peturbations of identity}\label{sec:pert}
We now consider the case that $\mathcal{M}$ is a perturbation of the identity, i.e., that
\begin{equation}\label{Mpert1}
\mathcal{M}(v) -  v= \varepsilon\, b(v) + o(\varepsilon)
\end{equation}
where $\varepsilon$ is a small positive parameter.  This case is of particular importance because, when the goal is to solve a differential equation initial value problem
\begin{equation}\label{ode1}
\frac{d}{dt} y = b(y),\qquad y(0) = y_0,
\end{equation}
discrete-in-time approximations take the form \eqref{Mit} with $\mathcal{M}$ of the form in \eqref{Mpert1}.
As is the case in several of our numerical examples, the solution to \eqref{ode1} may represent, for example, a semi-discretization (a discretization in space) of a partial differential equation (PDE).

Several common tasks in numerical linear algebra, not necessarily directly related to ODE or PDE can also be addressed by considering \eqref{ode1}.  For example, suppose that we solve  the ordinary differential equation (ODE) \eqref{ode1} with $b(y) = Ay-r$ for some $r\in \mathbb{C}^n$ and any $n\times n$ complex valued matrix $A.$  The solution to \eqref{ode1} in this case is
\[
y(t) = e^{t A} y_0 + A^{-1}\left( I - e^{tA}\right)r.
\]
Setting $r=0$ in the last display we find that 
any method to approximate ODE \eqref{ode1} for $t=1$ can  be used to approximate the product of a given vector and the exponential of the matrix $A.$
On the other hand, if $r\neq 0$ and all eigenvalues of $A$ have negative real part then, for very large $t,$ the solution to \eqref{ode1} converges to $A^{-1} r.$  In fact, in this case  we obtain the continuous time variant of Jacobi iteration for the equation $Ax = r.$  Like Jacobi iteration, it can be extended to somewhat more general matrices.  Discretizing  \eqref{ode1} in time with $b(y) = Ay-r$ and a small time step allows treatment of matrices with a wider range of eigenvalues than would be possible with standard Jacobi iteration.  

Some important eigenproblems are also solved using an $\mathcal{M}$ satisfying \eqref{Mpert1}.  For example, this is the case when the goal is to compute the eigenvalue/vector pair corresponding to the  eigenvalue of largest real part (rather than of largest magnitude) of a differential operator, e.g.\ the Schr\"odinger operator discussed in Section~\ref{sec:DMC}.
While the power method applied directly to a matrix $A$ converges to the eigenvector of $A$ corresponding to the eigenvalue of largest magnitude,
 the angle between the vector $\exp(tA)y_0$  and the eigenvector corresponding to the eigenvalue of $A$ with largest real part converges to zero (assuming $y_0$ is not orthogonal to that eigenvector).  If we discretize \eqref{ode1} in time with $b(v) = Av$ and renormalize the solution at each time step (to have unit norm) then the iteration will converge to the desired eigenvector (or a $\varepsilon$ dependent approximation of that eigenvector).

As we will learn in the next section, designing effective fast randomized iteration schemes for these problems requires that the error in the stochastic representation $\mathcal{M}\circ \Phi^m_t$ of $\mathcal{M}$ decrease sufficiently rapidly with $\varepsilon.$  
In particular, in order for our schemes to accurately approximate solutions to \eqref{ode1} over  intervals of $\mathcal{O}(1)$ units of time  (i.e., over $\mathcal{O}\left(1/\varepsilon\right)$ iterations of the discrete scheme),  
 we will need, and will verify in certain cases, that 
\[
\sqrt{\mathbf{E}\left[ \lvert f^\ct \mathcal{M}\left(\Phi^m_t(v)\right) - f^\ct  \mathcal{M}(v)\rvert^2 \right]} \sim \sqrt{\varepsilon}.
\]
Obtaining a bound of this type will require that we use a carefully constructed random compression $\Phi^m_t$ such as those described in Section~\ref{sec:rr}.  
In fact, when a scheme with $\mathcal{O}(n)$ cost per iteration is  acceptable, iteration \eqref{Rit} can be replaced by \eqref{Ritpert}, i.e., by 
\begin{equation}\label{Rpert1}
V_{t+1}^m = V_t^m +  \varepsilon\, b\left(\Phi^m_t(V_t^m)\right)+ o(\varepsilon)
\end{equation}
in which case we can expect errors over $\mathcal{O}\left(\varepsilon^{-1}\right)$ iterations that vanish with $\varepsilon$ (rather than merely remaining stable).
As we will see in more detail in the next section, the principle of dynamic self-averaging is essential to the convergence of either \eqref{Rit} or \eqref{Ritpert} when $\mathcal{M}$ is a perturbation of identity.
 The same principle is invoked in the contexts of, for example, multi-scale simulation (see e.g.\ \cite{PavliotisStuart:2008:Averaging} and \cite{EEngquist:2007:HMM} and the many references therein) and stochastic approximation (see e.g.\ \cite{KushnerYin:2003:SA} and the many references therein).

\section{Convergence}\label{sec:conv}


Many randomized  linear algebra schemes referenced in the opening paragraph of this article
 rely at their core on an approximation of a product such as $AB,$ where, for example, $A$ and $B$ are $n \times n$ matrices, by a product of the form $A \Theta B$ where $\Theta$ is an $n\times n$ random matrix with $\mathbf{E}\left[ \Theta \right] = I$ and so that $A\Theta B$ can be assembled at much less expense than $AB.$  For example, one might choose $\Theta$ to be a diagonal matrix with only $m\ll n$ non-zero entries on the diagonal so that $\Theta B$ has only $m$ non-zero rows and $A \Theta B$ can be assembled in $\mathcal{O}(n^2 m)$ operations instead of $\mathcal{O}(n^3)$ operations.  Alternatively one might choose $\Theta = \xi \xi^\tp $ where $\xi$ is a $n\times m$ random matrix with independent entries, each having mean 0 and variance $1/m$.  With this choice one can again construct $A\Theta B$ in $\mathcal{O}(n^2 m)$ operations.  Typically, this randomization is carried out once in the course of the algorithm.  The error made in such an approximation can be expected to be of size $\mathcal{O}(1/\sqrt{m})$ where the prefactor depends (very roughly) on the size of the matrices (and other structural properties) but does not depend directly on $n$ (see e.g.\ \cite[Equation~30]{Hammersley:1960:mc} or \cite[Theorem 1]{DrineasKannan:2006:MatMult}).  

In the schemes that we consider, we apply a similar randomization to speed matrix vector multiplication at each iteration of the algorithm (though our compression rules vary in distribution from iteration to iteration). 
As explored below, the consequence is that any stability property of the original, deterministic iteration responsible for its convergence, will be weakened by the randomization and that effect may introduce an additional $n$ dependence in the cost of the algorithm to achieve a fixed accuracy.  The compression rule must therefore be carefully constructed to minimize error.  Compression rules are discussed in detail in Section~\ref{sec:rr}.  In this section, we consider the error resulting from \eqref{Rit} and \eqref{Ritpert} for an unspecified compression rule satisfying the generic error properties established (with caveats) in Section~\ref{sec:rr}.  Both because it provides a dramatic illustration of the need to construct accurate compression rules and because of its importance in practical applications, we pay particular attention to the case in which $\mathcal{M}$ is an $\varepsilon$-perturbation of the identity. Our results rely on  classical techniques in the numerical analysis of deterministic and stochastic dynamical systems and, in particular, are typical of basic results concerning the convergence of stochastic approximation (see e.g.\ \cite{KushnerYin:2003:SA} for a general introduction and \cite{MoulinesBach:2011:SGD} for results in the context of machine learning) and interacting particle  methods (see e.g.\ \cite{DelMoral:2004:FK} for a general introduction and  \cite{Rousset:2006:QMC} for results in the context of QMC).  They concern the mean squared size of the difference between the output, $V_t^m,$ of the randomized scheme and the output, $v_t,$ of its deterministic counterpart and are chosen to efficiently  highlight important issues such as the role of stability properties of the deterministic iteration \eqref{Mit}, the dependence of the error on the size of the solution vector, $n,$ and the role of dynamic self-averaging.  More sophisticated results (such as Central Limit Theorems and asymptotic and non-asymptotic exponential bounds on deviation probabilities) are possible following developments in, for example,   \cite{KushnerYin:2003:SA} and \cite{DelMoral:2004:FK, Rousset:2006:QMC}.  In the interest of reducing the length of this article we list the proofs of all of our results separately in a supplemental document.

Our notion of error will be important.  It will not be possible to prove, for example that 
$
\mathbf{E}\left[ \lVert V_t^m - v_t\rVert_1\right]
$
remains small without a strong dependence on $n.$  It is not even the case that
$
\mathbf{E}\left[ \lVert \Phi^m_t\left(v\right) - v \rVert_1 \right]
$
is small when $n$ is large and $\lVert v\rVert_1 = 1.$  Take, for example, the case that $v_i = 1/n.$  In this case any scheme that sets $n-m$ entries to zero will result in an error $\lVert \Phi^m_t(v)- v\rVert_1\geq (n-m)/n.$
On the other hand, we need to choose a measure of error sufficiently stringent so that our eventual error bounds imply that our methods accurately approximate observables of the form $f^\ct v_t$.  For example, analogues of all of the results below using the error metric
$
(\mathbf{E}[ \lVert V_t^m - v_t \rVert_2^2])^{1/2}
$
could be established.  However, error bounds of this form are not, by themselves, enough to imply \jqw{dimension independent} bounds  on the error in $f^\ct v_t$ because they ignore correlations between the components of $V_t^m.$  Indeed, in general one can only expect that $(\mathbf{E}[ \lvert f^\ct V^m_t - f^\ct v_t \rvert])^{1/2}\leq \sqrt{n}(\mathbf{E} \lVert V_t^m - v_t \rVert_2^2])^{1/2}$ when $\lVert f\rVert_\infty\leq 1$

\begin{remark}
It is perhaps more typical in numerical linear algebra to state error bounds in terms of the quantity one ultimately hopes to approximate and not in terms of the distance to another approximation of that quantity.  For example, one might wonder why our results are not stated in terms of the total work required to achieve (say with high probability) an error of a specified size in an approximation of the dominate eigenvalue of a matrix.  Our choice to consider the difference between $V_t^m$ and $v_t$ is motivated by the fact that the essential characteristics contributing to errors due to randomization are most naturally described in terms of the map defining the deterministic iteration.
 More traditional characterizations of the accuracy of the schemes can be inferred from the bounds provided below and error bounds for the corresponding deterministic iterative schemes.
 \end{remark}

Motivated by  our stated goal, as described in Section~\ref{sec:gen}, of estimating quantities of the form $f^\ct v_t$  we measure the size of the (random) errors produced by our scheme using the norm
\begin{equation}\label{opnorm}
\opnorm{X} = \sup_{\lVert f\rVert_\infty \leq 1}
\sqrt{\mathbf{E}\left[ \lvert f^\ct X \rvert^2 \right]}
\end{equation}
where $X$ is a random variable with values in  $\mathbb{C}^n$  (all random variables referenced are assumed to be functions on a single probability space which will be left unspecified). 
This norm is the $(\infty,2)$-norm \cite[Section~7]{FL} of the square root of the second moment matrix of $X,$ i.e.,
\[
\opnorm{X} =\lVert B\rVert_{\infty,2} =  \sup_{\lVert f\rVert_\infty \leq 1} \lVert B f \rVert_2
\]
where \[
B^\ct B = \mathbf{E}\left[ X X^\ct \right].
\]
It is not difficult to see that the particular square root chosen does not affect the value of the norm. It will become apparent that our choice of the norm in \eqref{opnorm} is a natural one for our convergence results in this and the next section.

The following alternate characterization of $\opnorm{\,\cdot\,}$ will be useful later.
\begin{lemma}\label{opnormrep2}
The norm in \eqref{opnorm} may also be expressed as
  \begin{equation}\label{opnorm2}
  \opnorm{X} = \sup_{\lVert G \rVert_{\infty, *} \le 1}
\sqrt{\mathbf{E}\left[ \lVert G X \rVert_1^2 \right]},
\end{equation}
where
\begin{equation}\label{eq:inftydual}
\lVert G \rVert_{\infty, *} := \sum\nolimits_{i=1}^n \max_{j=1,\dots,n} \lvert G_{ij}\rvert 
\end{equation}
is the dual norm\footnote{See \cite[Proposition~7.2]{FL}.} of the $\infty$-norm of $G \in \mathbb{C}^{n \times n}$,
\[
\lVert G \rVert_{\infty} = \max_{\lVert f \rVert_\infty \leq 1} \lVert G f \rVert_\infty = \max_{i=1,\dots,n}  \sum\nolimits_{j=1}^n  \lvert G_{ij}\rvert .
\]
\end{lemma}

Note that if the variable $X$ is not random  then one can choose $f_i = X_i/\lvert X_i\rvert$ in \eqref{opnorm} and find that  $\opnorm{ X} = \lVert X\rVert_1.$  When $X$ is random we have the upper bound $\opnorm{X}^2 \leq \mathbf{E}\left[ \lVert X\rVert_1^2\right]$.  
  If, on the other hand,   $X$ is random but has mean zero and independent components then  $\opnorm{X}^2 =\mathbf{E}\left[ \lVert X\rVert_2^2\right].$   Concerning the relationship between these two norms more generally, we rely on the following lemma. 
  \begin{lemma}\label{opnormlowerbound}
Let $A$ be any  $n\times n$ Hermitian matrix with entries in $\mathbb{C}.$  
Then \[
 \sup_{\lVert f\rVert_\infty\leq 1}
 f^\ct  A f \geq  {\rm trace}\, A.
 \]
 \end{lemma}
  \noindent Lemma \ref{opnormlowerbound}, applied to the second moment matrix of $X,$    implies that $\opnorm{X}^2 \geq \mathbf{E}\left[ \lVert X\rVert_2^2\right]$.  Summarizing these relationships we have
 \begin{equation}\label{normorder}
\mathbf{E}\left[ \lVert X\rVert_2^2\right]\leq \opnorm{X}^2 \leq 
\mathbf{E}\left[ \lVert X\rVert_1^2\right].
\end{equation}
\jqw{The norms appearing in \eqref{normorder} are all equivalent.  What is important about the inequalities in \eqref{normorder} for our purposes is that they are independent of dimension.}

\begin{conditions} Consistent with results in the next section we will assume that  the typical error from our compression rule is
\begin{equation}\label{genPhierr}
\opnorm{ \Phi^m_t(v) - v} \leq  \frac{\gamma}{\sqrt{m}}\,\lVert v\rVert_1
\end{equation}
for $v\in \mathbb{C}^n,$ where $\gamma$ is a constant  that is independent of $m$ and $n.$
We will also assume that 
\begin{equation}\label{phibnd}
\mathbf{E}\left[\lVert \Phi_t^m(v)\rVert_1^2\right] \leq C_b \lVert v\rVert_1^2
\end{equation}
for some constant $C_b$ independent of $m$ and $n$ (for the compression scheme used in Section \ref{sec:ex}, \eqref{phibnd} is an equality with $C_b=1$).
For all of the compression methods detailed in Section \ref{sec:rr} (including the one used in our numerical experiments in Section \ref{sec:ex}), the statistical consistency  condition
\begin{equation}\label{Cunbiased}
\mathbf{E}\left[\Phi^m_t(v)\right] = v
\end{equation} 
 is satisfied exactly and we will assume that it holds exactly in this section.  Modification of the results of this section to accommodate a bias $\lVert \mathbf{E}\left[ \Phi^m_t(v)\right] - v\rVert_1 \neq 0$ is straightforward.
\end{conditions}

As a result of the appearance of $\lVert v\rVert_1$ in \eqref{genPhierr}, in our eventual error bounds it will be impossible to avoid dependence on $ \lVert V_t^m\rVert_1 .$   The growth of these quantities is controllable by increasing $m,$ but the value of $m$ required will often depend on the $n.$  The next theorem concerns the size of $\mathbf{E}\left[\lVert V_t^m\rVert_1^2\right].$  After the statement and proof of the theorem we discuss how the various quantities appearing there can be expected to depend on $n.$
%
In this theorem and in the rest of our results it will be convenient to recognize that, in many applications, the iterates $V_t^m$ and $Y_t^m = \Phi^m_t(V_t^m)$ are confined within some subset of $\mathbb{C}^n.$  For example, the iterates may all have a fixed norm or may have all non-negative entries.  We use the symbol $\mathcal{X}$ to identify this subset (which differs depending on the problem).  Until Theorem \ref{ft3} at the end of this section, our focus will be on iteration \eqref{Rit} though all of our results have analogues when \eqref{Rit} is replaced by iteration \eqref{Ritpert}.
\begin{theorem}\label{stable1}  Assume that $V_t^m$ is generated by either \eqref{Rit}  with a compression rule satisfying \eqref{genPhierr} and \eqref{Cunbiased}.
Suppose that $\mathcal{U}$ is a twice continuously differentiable function from $\mathcal{X}$ to $\mathbb{R},$ satisfying  
\[
\mathcal{U}(\mathcal{M}(v)) \leq \alpha\,  \mathcal{U}(v) + R\quad\text{for all}\quad v\in \mathcal{X}
\]
for some constants $\alpha$ and $R,$ and
\[
\lVert v\rVert_1^2 \leq \beta\, \mathcal{U}(v)\quad\text{for all}\quad v\in \mathcal{X}
\]
for some constant $\beta.$  Assume further that there is a constant $\sigma $ and a matrix $G\in \mathbb{C}^{n\times n}$ satisfying $\lVert G\rVert_{\infty,*}\leq 1$, so that, for   $z \in \mathbb{C}^n,$ 
\[
\sup_{ v \in \mathbb{C}^n} z^\ct  \left(D^2 \mathcal{U}(v)\right)  z \leq \sigma \lVert G z \rVert_1^2 
\]
where $D^2 \mathcal{U}$ is the matrix of second derivatives of $\mathcal{U}.$
Then
\[
\mathbf{E}\left[ \lVert V_t^m\rVert_1^2\right] \leq
 \beta R\Biggl[ \frac{ 1- \alpha^t \bigl( 1 + \frac{ \beta \gamma^2 \sigma}{2 m}\bigr)^t}
{ 1- \alpha \bigl( 1 + \frac{ \beta \gamma^2 \sigma}{2 m}\bigr)}\Biggr]
+ \beta \alpha^t \biggl( 1 + \frac{ \beta \gamma^2 \sigma}{2 m}\biggr)^t \mathcal{U}(V_0^m).
\]
\end{theorem}

First, the reader should notice that setting $\gamma=0$ in  Theorem~\ref{stable1} shows that the deterministic iteration \eqref{Mit} is stable whenever $\alpha <1.$  However, even for an $\mathcal{M}$ corresponding to a  stable iteration, the randomized iteration \eqref{Rit} may not be stable (and will, in general, be less stable).  If the goal is to estimate, e.g.\ a fixed point of $\mathcal{M},$ the user will first have to choose $m$ large enough that the randomized iteration is stable.

Though it is not explicit in the statement of Theorem~\ref{stable1}, in general the requirements for stability will depend on $n.$  Consider, for example, the case of a linear iteration, $\mathcal{M}(v) = Kv.$  This iteration is stable if the largest eigenvalue (in magnitude) is less than 1.  If we choose $\mathcal{U}(v) = \lVert v\rVert_2^2$ then
we can take $\alpha$ to be the largest eigenvalue of $K^\ct  K$ and $R=0$ in the statement of Theorem~\ref{stable1}.
 The   bound $\lVert\, \cdot\, \rVert_1 \leq \sqrt{n} \lVert \, \cdot\, \rVert_2$ (and the fact that it is sharp) suggests that we will have to take  $\beta = n$ in Theorem~\ref{stable1} (note that we can take $\sigma=1$ in the eventual bound).  This scaling suggests that, to guarantee stability we need to choose $m\sim n$.  

Fortunately this prediction is often (but not always) pessimistic.  For example, if $K$ is a matrix with non-negative real entries and $V_0^m$ has non-negative real entries then the iterates $V_t^m$ will have real, non-negative  entries (i.e., $v\in \mathcal{X}$ implies $v_i\geq 0$).  We can therefore use $\mathcal{U}(v) = (\mathbbm{1}^\tp v)^2 = \lVert v\rVert_1^2$ for $v\in \mathcal{X}$ and find that
 we can take $\alpha = \lVert K\rVert_1^2,$ $R=0,$ and $\beta=1,$ in the statement of Theorem~\ref{stable1}. 
With this choice of $\mathcal{U}$ we can again choose $\sigma = 1$  so that $n$ does not appear directly in the stability bound.  We anticipate that most applications will fall somewhere between these two extremes; maintaining stability will require increasing $m$ as $n$ is increased but not in proportion to the increase in $n.$

Having characterized the stability our schemes we now move on to bounding their error.  We have crafted the theorem below to address both situations in which one is interested in the error after a finite number of iterations and situations that require error bounds independent of the number of iterations.  In general, achieving error bounds independent of the number of iterations requires that $\mathcal{M}$ satisfy  stronger stability properties than those implied by the conditions in Theorem~\ref{stable1}.  While the requirements in Theorem~\ref{stable1} could be modified to imply the appropriate properties for most applications, we opt instead for a more direct approach and modify our stability assumptions on  $\mathcal{M}$  to \eqref{cntr1} and \eqref{cntr2} below.  
In this theorem and below we will make use the notation $\mathcal{M}_s^{t}$ to denote $\mathcal{M}$ composed with itself $t-s$ times.   In our proof of the  bound in Theorem~\ref{ft1} below we  divide the error into two terms, one of which is a sum of individual terms with  vanishing conditional expectations.  \jqw{Much like sums of independent, mean zero, random variables with finite variance, the size (measured by the square root of the second moment) of their sum  can be expected to grow less than linearly with the number of iterations (see the proof of Theorem~\ref{ft1}).}  This general observation is called dynamic self-averaging and results in an improved error bound.  The improvement is essential in the context of perturbations of the identity and we will mention it again below when we focus on that case.
\begin{theorem}\label{ft1}
Suppose that the iterates $V^m_t$ of \eqref{Rit} remain in $\mathcal{X}\subset \mathbb{C}^n.$
Fix a positive integer $T.$  Assume that there are constants $\alpha\geq 0,$  $L_1,$ and $L_2,$  so that for every pair of integers $s\leq r \leq T$ and for every vector $f\in \mathbb{C}^n$ with $\lVert f\rVert_\infty \leq 1$ there are matrices $G$ and $G'$ in $\mathbb{C}^{n\times n}$ satisfying $\lVert G\rVert_{\infty,*}\leq 1$ and a 
bounded, measurable $\mathbb{C}^{n\times n}$ valued function, $A,$ 
such that
 \begin{equation}\label{cntr1}
\sup_{v,\tilde v\in \mathcal{X}}\frac{\lvert f^\ct \mathcal{M}_s^{r}(v) - f^\ct \mathcal{M}_s^{r}(\tilde v)\rvert}{ \lVert Gv - G \tilde v\rVert_1
}
 \leq L_1 \alpha^{r-s}
 \end{equation}
 and 
  \begin{equation}\label{cntr2}
\sup_{v,\tilde v\in \mathcal{X}} \frac{\lvert f^\ct \mathcal{M}_s^{r}(v) - f^\ct \mathcal{M}_s^{r}(\tilde v) - f^\ct A(\tilde v) (v -\tilde v) \rvert}{ \lVert G' v - G' \tilde v\rVert_1^2
}
 \leq L_2 \alpha^{r-s}.
 \end{equation}
 Then the error at step $t\leq T$ satisfies the bound
  \[
  \opnorm{V_t^m - v_t} \leq \alpha  \biggl[\gamma \frac{1}{\sqrt{m}} (L_1+L_2)  \biggl( \frac{1-\alpha^{2t}}{1-\alpha^2}\biggr)^{1/2}  M_t+ \gamma^2 \frac{1}{m} L_2  \frac{1-\alpha^t}{1-\alpha} M_t^2 \biggr]  
 \]
 where $M_t^2 = \sup_{r<t}  \mathbf{E}\left[ \lVert V_r^m\rVert_1^2\right].$
\end{theorem}

Conditions \eqref{cntr1} and \eqref{cntr2} is easily verified for general linear maps $\mathcal{M} = K$ with $\alpha = \lVert K\rVert_1.$  The conditions are more difficult to verify for the power iteration map
\[
\mathcal{M}(v) = \frac{K v}{\lVert Kv \rVert_1}.
\]
A condition close to \eqref{cntr1} holds for all $v,\tilde v$ with $\lVert v\rVert_1 = \lVert \tilde v\rVert_1 =1$ but the parameter $\alpha$ in general depends on the proximity of $v$ and $\tilde v$ to the space spanned by all of the non-dominant eigenvectors (see e.g.\ \cite[Theorem~1.1]{Stewart}). 
For large enough $m$ we can ensure that all iterates $V_t^m$ remain at least some fixed distance from the space spanned by the non-dominant eigenvectors but this will often require that   $m$  grow  with $n.$  

As the following corollary establishes, when the matrix $K$ is real and non-negative we expect both matrix multiplication and power iteration to have errors independent of dimension.
\begin{corollary}\label{nonneg}
Suppose that $K$ is a real, entry-wise non-negative, irreducible, $n\times n$ matrix and that $V^m_0$ is real and non-negative with at most $m$ non-zero entries.  Then for both $\mathcal{M}(v) = Kv$ and $\mathcal{M}(v) = Kv / \lVert Kv\rVert_1,$ the bound on $ \opnorm{V_t^m - v_t}$ in Theorem \ref{ft1} is independent of dimension $n.$  Let $v_L$ and $v_R$ be the unique dominant left and right eigenvectors of $K$ with corresponding  eigenvalue  $\lambda_*.$  If $S$ is the stochastic matrix with entries $S_{ij} = \lambda_*^{-1}\, (v_L)_i\, K_{ij} / (v_L)_j$ and 
\[
\alpha =  \sup_{\substack{\lVert v\rVert_1 = 1\\ \mathbbm{1}^\tp v=0}} \lVert S v\rVert_1 
\]
then $\alpha<1$ and  the total error  for the randomized iteration \eqref{Rit} with $\mathcal{M}(v) = Kv / \lVert Kv\rVert_1$ (i.e., for randomized power iteration) as an approximation of $v_R$ is bounded by
\begin{equation}\label{nonnegtoterr}
\opnorm{V^m_t - v_R} \leq C_1\, \frac{\alpha}{1-\alpha} \frac{1}{\sqrt{m}} + C_2\, \alpha^t \,\opnorm{V^m_0 - v_R}
\end{equation}
for some constants $C_1$ and $C_2$ that depend on $ \max_j \{ (v_L)_j \}/  \min_j  \{(v_L)_j \}$, but do not otherwise depend on the iteration index, dimension, $\alpha,$ or $K$. 
\end{corollary}
The total error bound on randomized power iteration in Corollary \ref{nonneg} is a finite dimensional analogue of similar results concerning convergence of DMC and related schemes (see \cite{DelMoral:2004:FK} and the references therein).  In fact, the transformation from $K$ to $S$ in Corollary \ref{nonneg} is a finite dimensional analogue of a transformation that is essential to the efficiency of QMC in practical applications (see the discussion of importance sampling in \cite{FoulkesMitas:2001:QMC}) and that was used in \cite{Rousset:2006:QMC} to establish error bounds for a QMC scheme by an argument similar to the proof of Corollary \ref{nonneg}.

As discussed in Section \ref{sec:gen}, when $K$ is a dense matrix, the cost per iteration (measured in terms of floating point operations) of computing $\Phi^m_t(V^m_t)$ is $\mathcal{O}(n),$ while the cost of assembling the product $K\Phi^m_t(V^m_t)$ is $\mathcal{O}(mn).$  On the other hand, when $K$ has at most $q$ non-zero entries per column, the number of non-zero entries in $V_t^m$ will be at most $km$ so that the cost of computing $\Phi^m_t(V^m_t)$ is only $\mathcal{O}(km)$ and the cost of assembling $K\Phi^m_t(V^m_t)$ is only $\mathcal{O}(km).$  As a consequence of these observations and the bound in \ref{nonnegtoterr} we see that within any family of sparse (with a uniformly bounded number of non-zero entries per column) entry-wise non-negative matrices among which the parameter $\alpha$ is uniformly bounded below 1 and the ratio $ \max_j \{ (v_L)_j \}/  \min_j  \{(v_L)_j \}$ is uniformly bounded, the total cost to achieve a fixed accuracy is completely independent of dimension.


 For more general problems  one can expect the speedup over the standard deterministic power method to be roughly between a factor of $n$ and no speedup at all (it is clear that the randomized scheme can be worse than its deterministic counterpart when that method is a reasonable alternative).  Identification of more general conditions under which one should expect sublinear scaling for FRI in the particular context of  power iteration seems a  very interesting problem, but is not pursued here. 


\subsection{Bias} Even for a very general iteration the effect of randomization is evident when one considers the size of the expected error (rather than the expected size of the error).   When $\mathcal{M}(v) = K v$ and $V_t^m$ is generated by \eqref{Rit},  one can easily check that $z_t = \mathbf{E}\left[ V_t^m\right]$ satisfies the iteration $z_{t+1} = K z_t,$ i.e., $z_t = v_t.$  Even when the mapping $\mathcal{M}$ is non-linear, the expected error, $\mathbf{E}\left[ V_t^m\right] - v_t,$ is often much smaller than the error, $V_t^m - v_t,$ itself.  The following is just one simple result in this direction and demonstrates that one can often expect the bias to be $\mathcal{O}(m^{-1})$ (which should be contrasted to an expected error of $\mathcal{O}(m^{-1/2})$).  The proof is very similar to the proof of Theorem~\ref{ft1} and is omitted.
\begin{theorem}\label{bias}
Under the same assumptions  as in Theorem~\ref{ft1} and using the same notation, 
the bias at step $t\leq T$ satisfies the bound
  \[
\lVert \mathbf{E}\left[V_t^m\right] -   v_t\rVert_1  \leq  \frac{\gamma^2 L_2}{m}  \frac{\alpha  (1-\alpha^t)}{1-\alpha}M_t^2.
 \]
 \end{theorem}

\subsection{Perturbations of identity.}  When the goal is to solve ordinary or partial differential equations, stronger assumptions on the structure of $\mathcal{M}$ are appropriate.  We now consider the case in which $\mathcal{M}$ is a perturbation of the identity.  More precisely we will assume that
\begin{equation}\label{Mpert}
\mathcal{M}(v) = v + \varepsilon\, b(v).
\end{equation}
Though we will not write it explicitly, further dependence of $b$ on $\varepsilon$ is allowed as long as the assumptions on $b$ below hold uniformly in $\varepsilon.$  
In the differential equations setting $\varepsilon$ can be thought of as a time discretization parameter as in Section~\ref{sec:DMC}. 

\begin{condition}
When $\mathcal{M}$ is a perturbation of identity, it is reasonable to   strengthen our assumptions on the error made at each compression step.  The improvement stems from the fact that the mapping $\mathcal{M}$ nearly preserves the sparsity of its argument.  As we will explain in detail in the next section, if $v\in \mathcal{S}_m$ where
\[
 \mathcal{S}_m = \left\{z\in \mathbb{C}^n:  \lvert\{j:\,z_j \neq 0\}\rvert \leq m\right\}
\] 
and $w\in \mathbb{C}^n,$
then it is reasonable to assume that, for example,
\begin{equation}\label{Cpert}
\opnorm{\Phi^m_t(v+ w)  - v - w} \leq \frac{\gamma_p}{\sqrt{m}} \lVert w\rVert_1^\frac{1}{2} \lVert v+w \rVert_1^\frac{1}{2}
\end{equation}
for some constant $\gamma_p$ independent of $m$ and $n.$
\end{condition}

The following Lemma illustrates how such a bound on the compression rule can translate into 
small compression errors when $\mathcal{M}$ is a perturbation of the identity.
\begin{lemma}\label{pertlemma}  Suppose that the iterates $V^m_t$ of \eqref{Rit} remain in $\mathcal{X}\subset \mathbb{C}^n$ and that the compression rule satisfies   \eqref{Cpert} and \eqref{phibnd}.  Suppose that $\mathcal{M}(v) = v+ \varepsilon b(v)$ with $\lVert b(v)\rVert_1 \leq L (1+\lVert v\rVert_1)$ for all $v\in \mathcal{X}.$
Then for some constant $\tilde \gamma,$
\begin{equation}\label{Cpert2}
\opnorm{ \Phi^m_t(V_t^m) - V_t^m}^2 \leq    \tilde \gamma ^2  \frac{ \varepsilon }{ m } \sqrt{\mathbf{E}\left[ \lVert V_t^m \rVert_1^2\right]} \sqrt{ 1 + \mathbf{E}\left[ \lVert V_{t-1}^m\rVert_1^2\right]}
    \end{equation}
\end{lemma}

We now provide versions of Theorems~\ref{stable1} and \ref{ft1} appropriate when $\mathcal{M}$ is a perturbation of identity.  The proofs of both of these  theorems are very similar to the proofs of 
Theorems~\ref{stable1} and \ref{ft1} and are, at least in part, omitted.
First we address stability in the perturbation of identity case.
\begin{theorem}\label{stable2} Suppose that the iterates $V^m_t$ of \eqref{Rit} remain in $\mathcal{X}\subset \mathbb{C}^n$ and that the compression rule satisfies   \eqref{Cpert}, \eqref{phibnd}, and \eqref{Cunbiased}.  Suppose that $\mathcal{M}(v) = v+ \varepsilon b(v)$ with $\lVert b(v)\rVert_1 \leq L (1+\lVert v\rVert_1)$ for all $v\in \mathcal{X}.$
 Suppose further that $\mathcal{U}$ satisfies the conditions in the statement of Theorem~\ref{stable1} with the exception of the following: Now
\[
\mathcal{U}(\mathcal{M}(v)) \leq  e^{\varepsilon \alpha}\, \mathcal{U}(v) + \varepsilon R.
\]
Then
\[
\sup_{t < T/\varepsilon}\mathbf{E}\left[ \lVert V_t^m\rVert_1^2\right] \leq 
\beta R \Biggl[ \varepsilon + \frac{ \exp\Bigl[T\Bigl(\alpha + \frac{ \beta \tilde \gamma^2 \sigma}{2 m}\Bigr)\Bigr]-1}{ \alpha + \frac{ \beta \tilde \gamma^2 \sigma}{2 m}}\Biggr] 
+ \beta \exp\biggl[T\biggl(\alpha + \frac{ \beta \tilde \gamma^2 \sigma}{2 m}\biggr)\biggr]  \mathcal{U}(V_0^m)
\]
where $\tilde \gamma$ is the constant appearing in \eqref{Cpert2} and $\beta$ and $\sigma,$ 
are defined in  the statement of Theorem~\ref{stable1}.
\end{theorem}
What is important about the statement of Theorem~\ref{stable2} is that the bound remains stable as $\varepsilon$ decreases despite the fact that the set being supremized over is increasing.  Under the assumptions in the theorem (which are only reasonable when $\mathcal{M}$ is a perturbation of the identity) one can expect that the iterates can be bounded over $\mathcal{O}\left( \varepsilon^{-1}\right)$ iterations uniformly in $\varepsilon.$

The following theorem   interprets the result of Theorem~\ref{ft1} when $\mathcal{M}$ is a perturbation of identity.
One might expect that, over $\mathcal{O}(\varepsilon^{-1})$ iterations, $\mathcal{O}\left( \sqrt{\varepsilon}\right)$ 
errors made during the compression step would accumulate and lead to an error of $\mathcal{O}(\varepsilon^{-1/2}).$  Indeed, this is exactly what would happen if the errors made in the compression step were systematic (i.e., if the compression bias was $\mathcal{O}\left( \sqrt{\varepsilon}\right)$).  Fortunately, when the compression rule satisfies the consistency criterion \eqref{Cunbiased} the errors self average and their effect on the overall error of the scheme is reduced.  As mentioned above, this phenomenon played a role in the structure of the result in Theorem~\ref{ft1} and its proof, but its role is more crucial in Theorem~\ref{ft2} which provides uniform in $\varepsilon$ bounds on the error of \eqref{Rit} over $\mathcal{O}(\varepsilon^{-1})$ iterations.   Without the reduction in the growth of the error with $t$ provided by self-averaging it would not be possible to achieve an error bound over $\mathcal{O}\left(\varepsilon^{-1}\right)$ iterations that is stable as $\varepsilon$ decreases.
\begin{theorem}\label{ft2}  Suppose that the iterates $V^m_t$ of \eqref{Rit} remain in $\mathcal{X}\subset \mathbb{C}^n$ and that the compression rule satisfies   \eqref{Cpert}, \eqref{phibnd}, and \eqref{Cunbiased}.  Suppose that $\mathcal{M}(v) = v+ \varepsilon b(v)$ with $\lVert b(v)\rVert_1 \leq L (1+\lVert v\rVert_1)$ for all $v\in \mathcal{X}.$ Fix a real number $T>0$ and assume that,  for some real number $\alpha$ and some constants $L_1$ and $L_2$ and for every pair of integers $s\leq r \leq T/\varepsilon,$ for every vector $f\in \mathbb{C}^n$ with $\lVert f\rVert_\infty \leq 1$ there are matrices $G$ and $G'$ in $ \mathbb{C}^{n\times n}$ satisfying $\lVert G\rVert_{\infty,*}\leq 1$ and a bounded, measurable $\mathbb{C}^{n\times n}$ valued function $A$ so that
 \begin{equation}\label{cntr1eps}
\sup_{v,\tilde v\in \mathcal{X}}\frac{\lvert f^\ct \mathcal{M}_s^{r}(v) - f^\ct \mathcal{M}_s^{r}(\tilde v)\rvert}{ \lVert Gv - G \tilde v\rVert_1
}
 \leq L_1 e^{-\varepsilon \alpha (r-s)}
 \end{equation}
 and 
  \begin{equation}\label{cntr2eps}
\sup_{v,\tilde v\in \mathcal{X}} \frac{\lvert f^\ct \mathcal{M}_s^{r}(v) - f^\ct \mathcal{M}_s^{r}(\tilde v) - f^\ct A(\tilde v) (v -\tilde v) \rvert}{ \lVert Gv - G \tilde v\rVert_1^2
}
 \leq L_2 e^{-\varepsilon \alpha (r-s)}.
 \end{equation}
 Then the error at step $t\leq T/\varepsilon$ satisfies the bound
  \begin{multline*}
 \opnorm{V_t^m - v_t} \leq   \frac{\tilde \gamma(L_1+ L_2)}{\sqrt{m}}\left(e^{-2\alpha T}\mathbf{E}\left[ \lVert V_0^m\rVert_1^2\right]
  + \frac{1-e^{-2\alpha T}}{2 \alpha} M_T \sqrt{1+M_T^2}\right)^{1/2} \\     +   \frac{\tilde \gamma ^2L_2 }{m} \left( e^{-2\alpha T}\mathbf{E}\left[ \lVert V_0^m\rVert_1^2\right]
  +\frac{1-e^{-\alpha T}}{\alpha}M_T \sqrt{1+M_T^2}\right).
 \end{multline*}
where $M_T^2 = \sup_{r< T/\varepsilon}  \mathbf{E}\left[ \lVert V_r^m\rVert_1^2\right].$
\end{theorem}

%
%

Though the error established in the last claim is stable as $\varepsilon$ decreases, we have mentioned in Section~\ref{sec:pert} that when $\mathcal{M}$ is a perturbation of identity, by using iteration \eqref{Ritpert} instead of \eqref{Rit}, one might be able to obtain errors that vanish as $\varepsilon$ decreases (keeping $m$ fixed).
This is the subject of Theorem~\ref{ft3} below which, like Theorem~\ref{ft2} relies crucially on self-averaging of the compression errors.
Note that iteration \eqref{Ritpert} typically requires $\mathcal{O}\left(n\right)$ operations per iteration and storage of length $n$ vectors.
We have the following theorem demonstrating the decrease in error with $\varepsilon$ in this setting.
\begin{theorem}\label{ft3}
Suppose that the iterates $V^m_t$ of \eqref{Ritpert} remain in $\mathcal{X}\subset \mathbb{C}^n$ and that   \eqref{genPhierr} holds.  Under the same assumptions on $\mathcal{M}$ as in Theorem~\ref{ft2},  the error at step $t\leq T/\varepsilon$ satisfies the bound
  \begin{equation*}
 \sup_{t\leq T/\varepsilon} \opnorm{V_t^m - v_t} \leq \frac{ \sqrt{\varepsilon} \gamma }{\sqrt{m}}(L_1+ L_2)L_1 \left(\frac{1-e^{-2\alpha T}}{2 \alpha}\right)^\frac{1}{2} e^{\alpha \varepsilon} M_T      +  \frac{  \varepsilon \gamma^2 L_2 L_1^2}{m} \left( \frac{1-e^{-\alpha T}}{\alpha}\right)
   M_T^2 
 \end{equation*}
where $M_T^2 = \sup_{r< T/\varepsilon}  \mathbf{E}\left[ \lVert V_r^m\rVert_1^2\right].$
\end{theorem}

\section{Compression rules}\label{sec:rr}

In this section we give a detailed description of the compression rule used in our numerical simulations as well as several others, and  an analysis of the accuracy of those schemes.  
Programmed efficiently, and assuming that $v$ has exactly $n$ nonzero entries, all of the schemes we discuss in this section will require at most $\mathcal{O}(n)$ floating point operations  including the generation of as few as one uniform random variate  and $\mathcal{O}(n+m\log n)$ floating point comparisons.
It is likely that better compression schemes are possible, for example by incorporation of ideas from \cite{HairerWeare:2014:TDMC}.  The reader should note that in this section $n$ represents the number of non-zero entries in the input vector $v$ of the compression rule and not the dimension associated with a particular problem (which may be much larger).  In our implementation of \eqref{Rit}, when the underlying matrix is sparse (so that an $\mathcal{O}(m)$ work/storage per iteration method is possible) we store only the indices and values of the non-zero entries in any vector (including matrix columns).

We begin by discussing the simple choice
\begin{equation}\label{Phi3}
\left(\Phi^m_t(v)\right)_j = \begin{cases}
N_j \frac{\lVert v\rVert_1}{m} \frac{v_j}{\lvert v_j\rvert} & \text{if}\; \lvert v_j \rvert >0,\\
 0 & \text{if}\; \lvert v_j \rvert = 0,
 \end{cases}
\end{equation}
where each $N_j$ is a random, non-negative, integer with  
\begin{equation}\label{Nconsist}
\mathbf{E}\left[ N_j\mid v\right] = \frac{m \lvert v_j\rvert}{\lVert v\rVert_1}
\end{equation}
so that $\mathbf{E}\left[\Phi^m_t(v)\right] = v$ and the consistency condition \eqref{Cunbiased} is satisfied.   Notice that if we define a collection of  $N = \sum\nolimits_{j=1}^n N_j$ integers $\{X^{(j)}\}$ so that exactly $N_j$ elements of the collection are equal to $j,$ then the output of a compression scheme of this type can be written
\[
\Phi^m_t(v) = \frac{\lVert v\rVert_1}{m} \sum\nolimits_{j=1}^N  \frac{v_{X^{(j)}}}{\lvert v_{X^{(j)}}\rvert}\,\mathbf{e}_{X^{(j)}}
\]
where $\mathbf{e}_j$ is the $j$th standard basis vector in $\mathbb{R}^n.$  When the input vector $v$ is real  $\Phi^m_t(v)$ is a  finite dimensional analogue of the DMC resampling step described in Section \ref{sec:DMC}.  In the infinite dimensional setting the efficiency of DMC is due to the application of an  integral operator, $e^{-\varepsilon \mathcal{H}},$ to a finite sum of delta functions in place of a more general function.  Likewise, the gain in efficiency of an FRI scheme over deterministic methods is a consequence of replacement of a general vector $v$ in the product $Kv$ by a sparse approximation, $\Phi^m_t(v).$   Though we will deviate somewhat from the form in \eqref{Phi2} to arrive at the compression scheme used in the numerical simulations reported on in Section \ref{sec:ex}, essential elements of \eqref{Phi2} will be retained.

%

Notice that the consistency condition \eqref{Nconsist} leaves substantial freedom in the specification of the joint distribution of the $N_j.$  For example, one simple choice might be to select the vector of $N_j$ from the multinomial distribution with parameters $m$ and $( \lvert v_1\rvert, \lvert v_2\rvert, \dots, \lvert v_n\rvert)/\lVert v\rVert_1.$  This choice would result in a compression scheme satisfying \eqref{genPhierr},  \eqref{phibnd}, and \eqref{Cunbiased}, as required in Theorems \ref{stable1} and \ref{ft1} in Section \ref{sec:conv}.  However, it would not satisfy \eqref{Cpert} and would be a particularly poor choice when $\mathcal{M}$ is a perturbation of the identity.  In fact, this choice would lead to unstable schemes as the size of the perturbation decreases.  
An alternative, much more accurate  choice  that will lead  below (in Lemma~\ref{dmcerr}) to a compression scheme satisfying \eqref{Cpert} is to select the $N_j$ independently with
\begin{equation}\label{dmcNj}
\mathbf{P}\left(N_j =  \left\lceil \frac{m \lvert v_j\rvert}{\lVert v\rVert_1}\right\rceil \right) = 
1 -\mathbf{P}\left(N_j =  \left\lfloor \frac{m \lvert v_j\rvert}{\lVert v\rVert_1}\right\rfloor \right) = \frac{m \lvert v_j\rvert}{\lVert v\rVert_1}- \left\lfloor \frac{m \lvert v_j\rvert}{\lVert v\rVert_1}\right\rfloor.
\end{equation}
Note that this rule randomly rounds $m\lvert v_j\rvert/\lVert v\rVert_1$ to a nearby integer and satisfies \eqref{Nconsist}.
The compression rule \eqref{Phi3} with \eqref{dmcNj} has already appeared above in \eqref{Phi2}. When $v$ has exactly $n$ non-zero entries, the corresponding cost to assemble $\Phi^m_t$  by this rule is  $\mathcal{O}(n)$ operations.  


However, we have emphasized repeatedly in this article that the cost savings at each iteration of  an FRI scheme is entirely do to sparsity introduced by our compressions.  And the results of the last section reveal that compression schemes with large variance will typically give rise to FRI schemes with large error.  In this regard the compression rule in \eqref{Phi3} is clearly suboptimal.  In particular, for any entry $j$ for which ${m \lvert v_j\rvert}/{\lVert v\rVert_1}>1,$ the $j$th component of $\Phi^m_t(v)$ is non-zero with probability 1 so that the error $\left(\Phi^m_t(v)\right)_j - v_j$ is not compensated by an increase in sparsity.  To improve the scheme we can introduce a rule for exactly preserving sufficiently large entries of $v.$  To that end, let $\sigma$ be a permutation  of $\{1,2,\dots,n\}$ so that the  elements of $v_\sigma$ have decreasing magnitude (i.e., $v_\sigma$ is a rearrangement of the entries of $v$ so that, for each $j,$
$\lvert v_{\sigma_j}\rvert \geq  \lvert v_{\sigma_{j+1}}\rvert$) and let
\begin{equation}\label{tau}
\tau_v^m = \min\left\{0\leq \ell \leq m:\, \sum\nolimits_{j=\ell+1}^{n} \lvert v_{\sigma_j} \rvert \geq (m-\ell)\lvert v_{\sigma_{\ell+1}} \rvert \right\}.
\end{equation}
All of the compression schemes we consider will preserve entries $v_{\sigma_1},v_{\sigma_2},\dots, v_{\sigma_{\tau_v^m}}$ exactly.  In fact, they will have the basic structure in Algorithm~\ref{cr}. 
\begin{algorithm}[ht]\label{cr}
 \KwData{$v\in \mathbb{C}^n$ with all nonzero entries, $m\in \mathbb{N}$}
 \KwResult{$V=\Phi^m(v)\in \mathbb{C}^n$ with at most $m$ nonzero entries}
 $\tau_v^m = 0$\;
 $V=0$\;
 $r = \lVert v\rVert_1 / m$\;
 $\sigma_1 = \arg\max_i\{\lvert v_i\rvert\}$\; 
 \While{$\lvert v_{\sigma_{\tau_v^m+1}}\rvert \geq r$ }{
 $ \tau_v^m = \tau_v^m+1$\;
  $V_{\sigma_{\tau_v^m}} = v_{\sigma_{\tau_v^m}}$\; 
  $v_{\sigma_{\tau_v^m}} = 0$\;
  $r = \lVert v\rVert_1 / (m-{\tau_v^m})$\;
  $\sigma_{{\tau_v^m}+1} = \arg\max_i\{\lvert v_i\rvert\}$;
  }
 For each $j$ let $N_j$ be a non-negative random integer with   $\mathbf{E}\left[ N_j\,|\,v\right] = (m-\tau_v^m)\lvert v_j\rvert/ \lVert v\rVert_1$\;
Finally,  for $j\in\{1,2,\dots,n\}\setminus\{\sigma_1,\sigma_2,\dots,\sigma_{\tau_v^m}\},$ set 
  \[
  V_j = N_{j} \frac{v_{j}\lVert v\rVert_1}{\lvert v_{j}\rvert(m-\tau_v^m)}
  \]
  (Note that $v$ here may have fewer non-zero entries than it did upon input)\;
 \caption{A simple compression rule.}
\end{algorithm}

To justify preservation of the  $\tau_v^m$ largest entries in our compression schemes,  we need the following lemma.
\begin{lemma}\label{taubnd}  $\tau_v^m$ satisfies
 \[
  \sum\nolimits_{j=\tau_v^m+1}^n \lvert v_{\sigma_j}\rvert \leq \frac{ m - \tau_v^m}{m} \lVert v\rVert_1.
  \]
  \end{lemma}
\noindent Note that for any compression scheme satisfying an error bound of the form \eqref{genPhierr} for a general vector $v\in \mathbb{C}^n,$ the error resulting from application of the compression scheme after exact preservation of the largest $\tau_v^m$ entries is bounded by
\[
\frac{\gamma}{\sqrt{m-\tau_v^m}} \sum\nolimits_{j=\tau_v^m+1}^n \lvert v_{\sigma_j}\rvert
\]
which, by Lemma \ref{taubnd} is itself bounded by 
\[
\gamma \frac{\sqrt{m-\tau_v^m}}{m} \lVert v\rVert_1
\]
and is always an improvement over \eqref{genPhierr}. 

Lemma \ref{dmcerr} summarizes the properties of the compression scheme resulting from preserving the largest $\tau_v^m$ entries of an input vector $v\in \mathbb{C}^n$ exactly and applying \eqref{Phi3} with \eqref{dmcNj} with $m$ replaced by $m- \tau_v^m$ to the remaining entries.  In particular, Lemma \ref{dmcerr} implies that the compression scheme satisfies conditions \eqref{genPhierr}, \eqref{phibnd}, and  \eqref{Cpert}.
\begin{lemma}\label{dmcerr}  Let $v,w\in \mathbb{C}^n$ and assume that $v$ has at most $m$ non-zero entries.
For $\Phi_t^m$ defined by Algorithm~\ref{cr} with \eqref{dmcNj} 
\begin{equation}\label{dmccomperr2}
\opnorm{ \Phi_t(v+w) - v-w}\leq \sqrt{2} \frac{ \lVert w\rVert_1^\frac{1}{2} \lVert v+w\rVert_1^\frac{1}{2}}{\sqrt{m}}.
\end{equation}
Concerning the size of the resampled vector we have the bound
\[
\mathbf{E}\left[ \lVert \Phi_t^m(v+w)\rVert_1^2\right] \leq \lVert v+w \rVert_1^2 + 2\frac{ \lVert v+w\rVert_1 \lVert w\rVert_1}{m}.
\]
Finally, if $\tau_{v+w}^m>0$ then $\mathbf{P}\left[ \Phi_t(v+w) = 0\right] = 0.$  
If $\tau_{v+w}^m= 0$ then 
\[
\mathbf{P}\left[ \Phi_t(v) = 0\right] \leq \left( \min\left\{\frac{ \lVert w\rVert_1}{\lVert v+w\rVert_1},\frac{1}{e} \right\}\right)^m.
\]
\end{lemma}

In practice this compression scheme would need to be modified to avoid the possibility that $\Phi^m_t(v) = 0.$   
As Lemma~\ref{dmcerr} demonstrates, the probability of this event  is  extremely small.  The issue can be avoided by simply sampling $\Phi_t^m(v)$ until $\Phi_t^m(v)\neq 0,$ i.e., sample $\Phi_t^m(v)$ conditioned on the event $\left\{\Phi_t^m(v)\neq 0\right\},$ and multiplying each entry of the resulting vector by $\mathbf{P}\left[\Phi_t^m(v)\neq 0\right]$
which can be computed exactly. 
A more significant issue is that, while Lemma~\ref{dmcerr} does guarantee that the compression scheme just described satisfies \eqref{Cpert}, the scheme does not guarantee that the number of non-zero entries in $\Phi^m_t(v)$ does not exceed $m$ as required by Lemma~\ref{pertlemma} in the last section.  The results of that section can be modified accordingly or the compression scheme can be modified so that $\Phi^m_t(v)$ has no more than $m$ non-zero entries (by randomly selecting additional entries to set to zero).  Instead of pursuing these modifications here we move on to describe the compression scheme used to generate the results reported in the next section.


Like the compression scheme considered in Lemma~\ref{dmcerr}, the compression scheme used to generate the results in Section \ref{sec:ex} begins with an application of Algorithm~\ref{cr}.  To fully specify the scheme we need to specify the rule used to generate the $N_j$ random variables for $j\in \{\sigma_k:\, k>\tau_v^m\}$.    For $k=1,2,\dots, m-\tau_v^m,$ 
define the random variables 
\begin{equation}\label{sysU}
U^{(k)} = \frac{1}{m-\tau_v^m}\left( k-1 + U\right).
\end{equation}
where  $U$ is a single uniformly chosen random variable on the interval (0,1).
We then set
\begin{equation}\label{sysNj}
N_{\sigma_j} = \Bigl\lvert\Bigl\{k:\, U^{(k)} \sum\nolimits_{j=\tau_v^m+1}^{n} \lvert v_{\sigma_j}\rvert\in I_j\Bigr\}\Bigr\rvert
\end{equation}
where we have defined the intervals $I_{\tau_v^m+1} = \bigl[0,\,\lvert v_{\sigma_{\tau_v^m+1}}\rvert\bigr)$ and, for $i = \tau_v^m+2,\dots,n,$
\begin{equation}\label{I}
I_i = \Bigl[\sum\nolimits_{j=\tau_v^m+1}^{i-1} \lvert v_{\sigma_j}\rvert,\, \sum\nolimits_{j=\tau_v^m+1}^{i} \lvert v_{\sigma_j}\rvert\Bigr).
\end{equation}
As for the rule in \eqref{dmcNj}, the variables $N_j$ generated according to \eqref{sysNj} satisfy 
\[
\mathbf{E}\left[ N_j \mid v\right] = \frac{(m-\tau_v^m)\lvert v_j\rvert}{\sum\nolimits_{i=\tau_v^m+1}^{n} \lvert v_{\sigma_i}\rvert}
\]
 so that the  compression mapping $\Phi^m_t$ resulting from use of \eqref{sysNj} with Algorithm \ref{cr} satisfies 
\eqref{Cunbiased}.   From the definition of $\tau_v^m,$ we know that for $j>\tau_v^m,$
$(m-\tau_v^m)\lvert v_j\rvert \leq \sum\nolimits_{i=\tau_v^m+1}^{n} \lvert v_{\sigma_i}\rvert$  which implies by \eqref{sysNj} that $N_j\in \{0,1\}$.  

Unlike \eqref{dmcNj}, \eqref{sysNj} results in $N_j$ that are correlated and satisfy $\sum\nolimits_{j=\tau_v^m+1}^n N_j = m-\tau_v^m$ exactly (not just in expectation).  The corresponding compression scheme exactly preserves the $\ell_1$-norm of $v$ and results in a vector $\Phi^m_t(v)$ with at most $m$ non-zero entries.  Note that this compression scheme, like the one considered in Lemma~\ref{dmcerr} only requires knowledge of set $\{\sigma_1,\sigma_2,\dots,\sigma_{\tau_v^m}\}$ and does not require sorting of the entire input vector $v.$
Perhaps the most obvious advantage of this scheme over the one that generates the $N_j$ according to \eqref{dmcNj} is that the compression scheme using \eqref{sysNj} only requires a single random variate per iteration (compared to up to $n$ for \eqref{dmcNj}).  Depending on the cost of evaluating $\mathcal{M}(V_t^m),$ this advantage could be substantial.  

Notice that, if we replace the $U^{(k)}$ in \eqref{sysU} by independent random variables uniformly chosen in $(0,1),$ then the $N_j$ would be distributed multinomially, which we have already mentioned is a poor choice.  Relative to multinomial $N_j,$ the increased correlation between the $U^{(k)}$ defined in \eqref{sysU} results in substantially decreased variance for the $N_j$ but also increased covariance.    An unfortunate consequence of this increased covariance is that the analogue of the error bound \eqref{dmccomperr2} from Lemma~\ref{dmcerr} does not hold for $N_j$ generated according to \eqref{sysNj}.  In fact, the rule in \eqref{sysNj} is very closely related to the systematic resampling scheme used frequently in the context of sequential Monte Carlo (see e.g.\ \cite{deFreitasDoucetGordon:2005:book}) which is well known to fail to converge for certain sequences of input vectors.\footnote{If applied to the same vector $v$ as $m$ increases, the scheme does converge}
 Nonetheless, in unreported numerical comparisons  we found that the rule \eqref{sysNj}  resulted  in FRI schemes with significantly lower error than for \eqref{dmcNj}.

\section{Numerical tests}\label{sec:ex}

In this section we describe the application of the framework above to particular matrices arising in (i) the computation of the per-spin partition function of the 2D Ising model, (ii) the spectral gap of a diffusion process governing the evolution of a system of up to five, $2$-dimensional particles (i.e., up to ten spatial dimensions), and (iii) a  free energy landscape for that process.  The corresponding numerical linear algebra problems are, respectively, (i) computing the dominant eigenvalue/eigenvector of matrices up to size $10^{15} \times 10^{15},$  (ii) computing the second largest eigenvalue/eigenvector of matrices up to size $10^{20} \times 10^{20},$  and (iii) solving a linear system involving exponentiation of matrices up to size $10^{20} \times 10^{20}.$  Aside from sparsity, these matrices have no known readily exploitable structure for computations.

All but the first test problem involve matrices with entries of any sign.  As we learned in Section \ref{sec:conv}, we can often expect much better error scaling with dimension when applying FRI to problems involving matrices with all non-negative entries.  The numerical results in this section suggest that dramatic speedups are possible even for more general matrices.

The reader may wonder why we consider random compressions instead of simple thresholding, i.e., a compression rule in which, if $\sigma_j$ is the index of the $j$th largest (in magnitude) of $v,$ $v_{\sigma_j}$ is simply set to zero for all $j>m$ (the resulting vector can be normalized to preserve $\ell^1$-norm or not).  In the rest of this paper we will refer to methods using such a compression rule as truncation-by-size (TbS) schemes.  TbS schemes have been considered by many authors (see e.g.\ \cite{Fuchs:1989:TrMa,SchaefferCaflisch:2013:SparsePDE,OzolinsLai:2013:Sparse2,OzolinsLai:2014:Sparse3}) and are a natural approach.  Note however that the error (if the compression is not normalized), 
\[
\left\lvert \sum\nolimits_{j>m} \bar{f}_{\sigma_j} v_{\sigma_j}\right\rvert,
\]
 for the thresholding compression can be as large as $\Vert f\rVert_\infty \lVert v\rVert_1 \left(1- m/n\right)$  which only vanishes if $m$ is increased faster than $n.$  In contrast, the random compressions above can have vanishing error even when $n$ is infinite.  This observation is key to understanding the substantial reduction in error we find for our fast randomized scheme over TbS  in numerical results presented in this section.  In our first test example the TbS scheme  converges to a value far  from a high quality estimate of the true value (a relative error of $98\%$ compared to $8\%$ for FRI).  In the subsequent examples  TbS iteration appears to converge (in the iteration index $t$) to substantially different values for each fixed choice of $m$  whereas FRI shows much more consistent behavior in $m.$  Moreover, in practice we observe no cost savings per iteration for TbS over FRI.

Finally we comment that, in order for the FRI approach to yield significant performance improvements, one must use an efficient implementation of matrix by sparse vector multiplication.  In the examples below we list the $i,j$ pairs for which the product $K_{ij}v_j$ is nonzero, then sort the products according to the $i$-index and finally, add the products with common $i$.  This is a simple and sub-optimal solution.   More details can be found in the example code  available in \cite{FRIforIsingCode}.

\subsection{A transfer matrix eigenproblem}
In this example we find the dominant eigenvalue $\lambda_*$ of the transfer matrix $K$ of the $2$-dimensional 
$\ell$-spin Ising model.  This eigenvalue is the per-spin partition function of the infinite spin Ising model, i.e.,
\[
\lambda_*(T,B) = \lim_{\ell\rightarrow \infty} \left(\sum\nolimits_{\sigma} e^{\frac{1}{T}\sum\nolimits_{|(i,j)-(i',j')|=1} \sigma_{ij} \sigma_{i' j'} + B \sigma_{ij}}\right)^{1/\ell}
\]
where $\sigma_{ij}\in \{-1,1\}$ and the sum in the exponent is over pairs of indices on a square $2$-dimensional, periodic lattice with $\ell$ sites.  The outer sum is over all $2^\ell$ possible values of $\sigma$, and for larger $\ell$,  one cannot possibly compute it directly.  The matrix $K$ is $2^\ell\times 2^\ell$.  For example, in the case $\ell=3$, 
\begin{equation}\label{Kpart}
K = 
\begin{bmatrix}
a& & & & a^{-1} & & &\\
b& & & & b^{-1} & & &\\
 & a & & & & a^{-1} & &\\ 
& b & & & & b^{-1} & &\\ 
& & b & & & & b^{-1} & \\
& & c & & & & c^{-1} & \\
& & & b & & & & b^{-1} \\
& & & c & & & & c^{-1} 
\end{bmatrix}
\end{equation}
where 
\[
a = e^{(2-B)/T},\quad b= e^{-B/T},\quad c = e^{-(2+B)/T}.
\]
We therefore also cannot hope to apply the power method (or its relatives) directly to $K$ when $\ell$ is large.  In our experiments we set $T=2.2$, $B=0.01$, and  $\ell=50$ so that $n=2^\ell >10^{15}$,   We apply both the FRI and TbS, $\mathcal{O}(1)$ schemes to computing $\lambda_*$ as well as to computing the 
sum of all components of the corresponding eigenvector, $v_*$ (normalized to have sum equal to 1), with index greater than or equal to $2^{49}$, i.e.,
\[
f_* = \sum\nolimits_{j\geq n/2} (v_*)_j.
\]

Knowledge of the partition function $\lambda_*$ as a function of temperature $T$ and field strength $B$ allows one to determine useful quantities such as the average magnetization (sum of spin values) and to diagnose phase transitions \cite{Fuchs:1989:TrMa}. 
Our choice to estimate $\lambda_*$ and $f_*$ is motivated in part by the fact that these quantities can be approximated accurately by the corresponding values for smaller Ising systems.  We will compare our results to those for the 24-spin Ising model which we can solve by standard power iteration.  For an effective specialized method for this problem see \cite{ParlettHeng:1992:Ising}.  A simple, educational implementation of Fast Randomized Iteration applied to this problem can be found here \cite{FRIforIsingCode}.

\begin{figure}[!htb]
\centerline{
\includegraphics[height=22em]{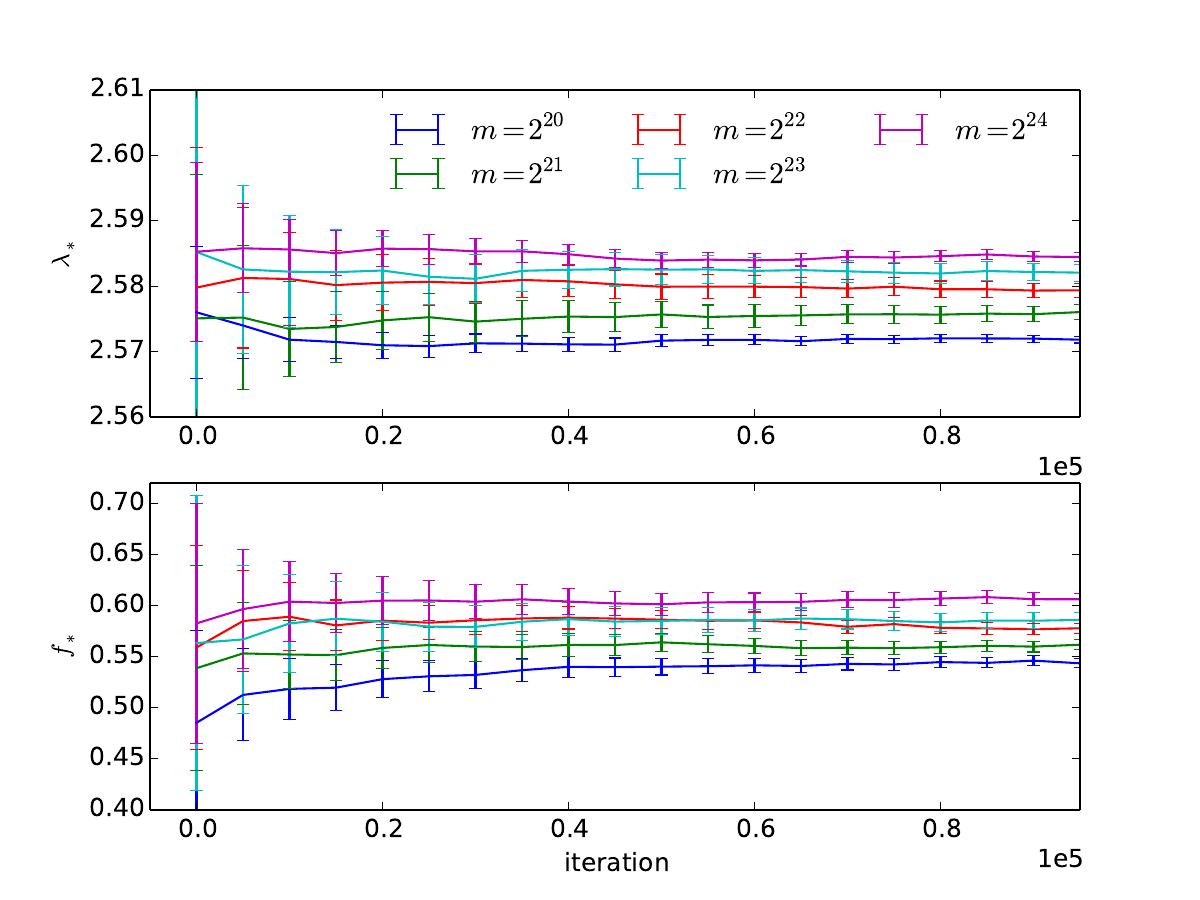}}
\caption{{\bf Top.} Trajectory averages of the approximation,  $\Lambda_t^m$, of the partition function for the 50-spin Ising model with $B=0.01$ and $T=2.2$, with $95\%$ confidence intervals\protect\footnotemark[8]
 as computed by the FRI method with $m=2^k$ for $k=20,\, 21,\, 22,\, 23$, and $24$.  The best (highest $m$) estimate for $\lambda_*$ is 2.584
a difference of roughly $0.5\%$ from the value for the 24-spin Ising model.  {\bf Bottom.} Corresponding trajectory averages of the approximation, $F_t^m$, of the total weight of all components of $v_*$ with index  greater than or equal to $2^{49}$ with $95\%$ confidence intervals for the FRI method.  The best (highest $m$) estimate for $f_*$ is 0.606, a difference of roughly $8\%$ from the value for the 24-spin model.}\label{IsingFRIaves}
\end{figure}

\begin{figure}[!htb]
\centerline{
\includegraphics[height=22em]{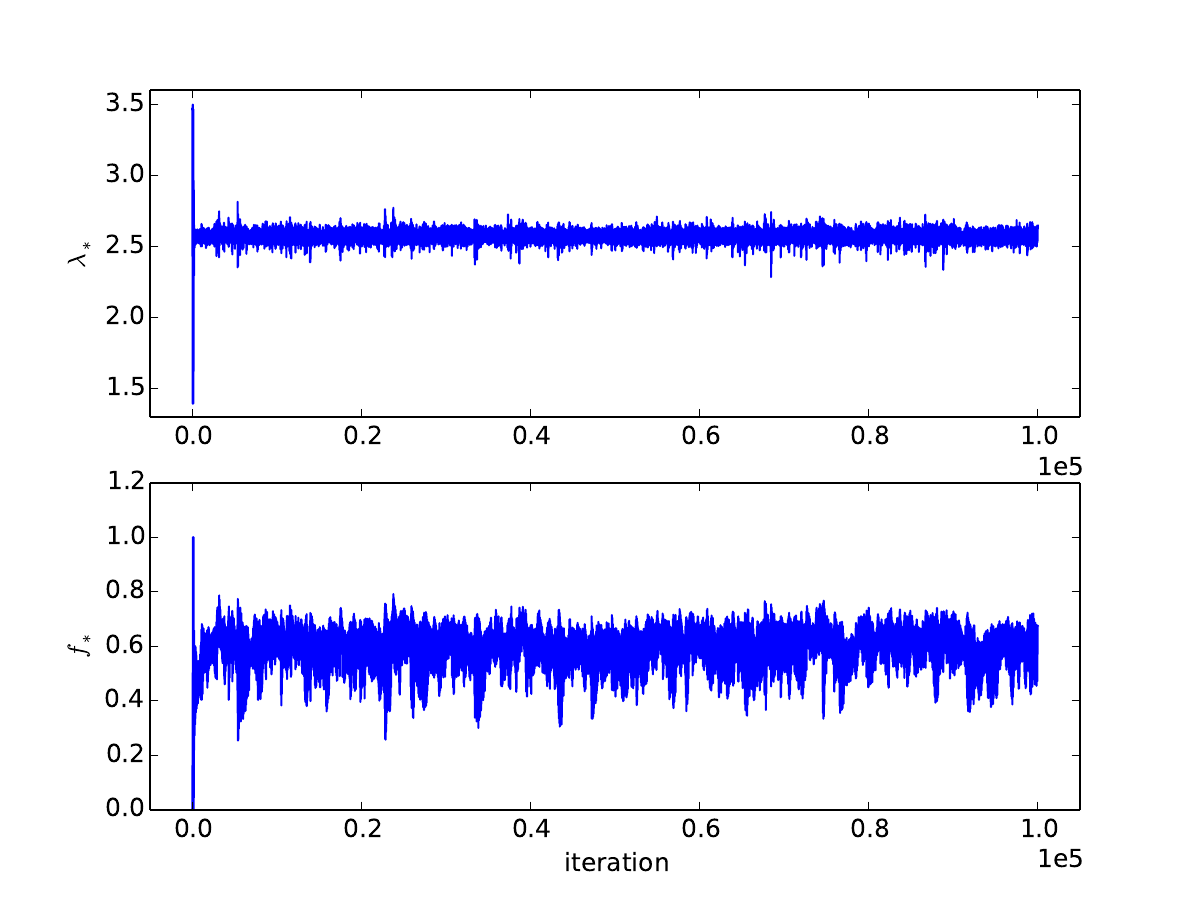}}
\caption{{\bf Top.} Trajectory  of the approximation, $\Lambda_t^m$, of the partition function for the 50-spin Ising model,  for the FRI method with $m=2^{24}$. The approximate integrated autocorrelation time for $\Lambda_t^m$ is 20.5 iterations.  {\bf Bottom.} Corresponding trajectory of  $F_t^m$ as computed by  the FRI method.  The approximate integrated autocorrelation time for $F_t^m$ is 274 iterations.}
\label{IsingFRItraj}
\end{figure}

\begin{figure}[!htb]
\centerline{
\includegraphics[height=22em]{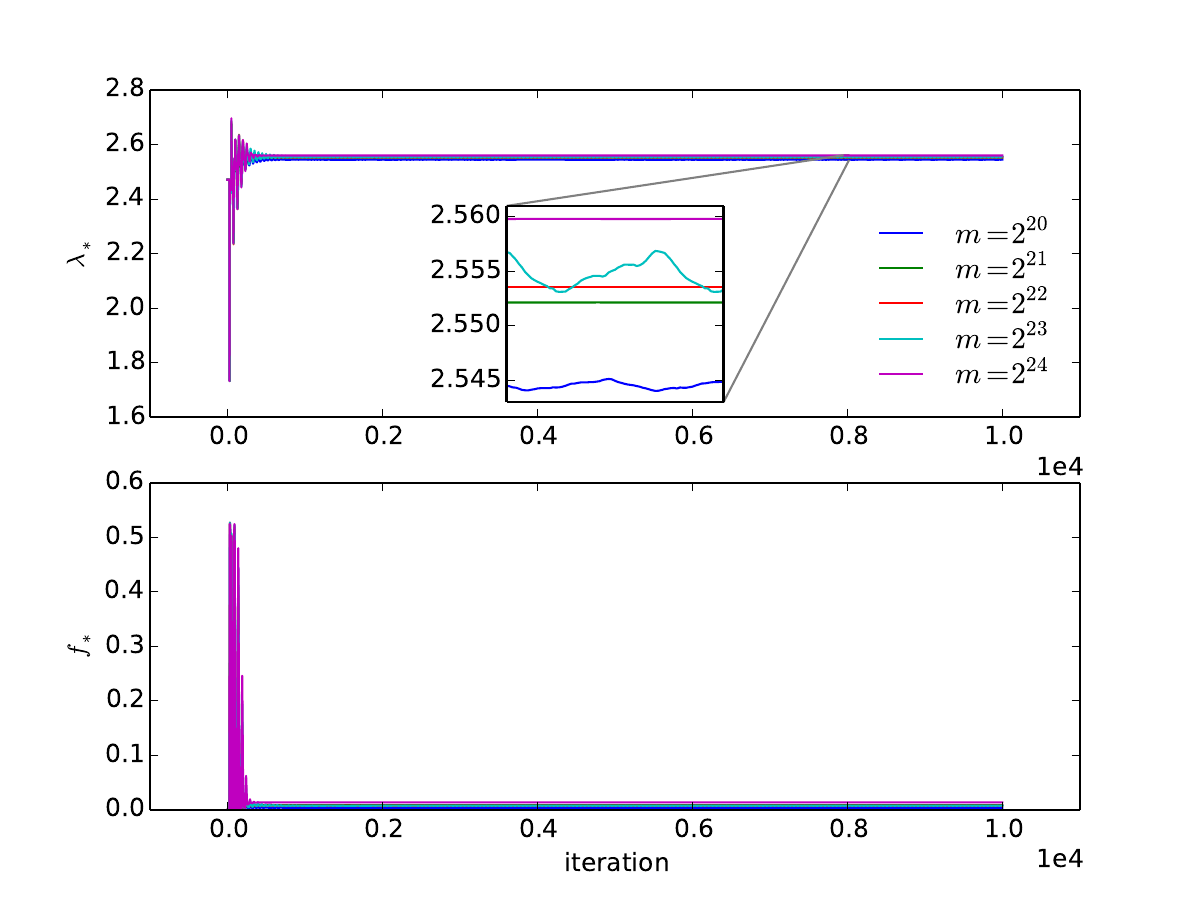}}
\caption{{\bf Top.} Trajectory  of the approximation, $\Lambda_t^m$, of the partition function for the 50-spin Ising model,  for the TbS method with $m=2^{24}$.  The best (highest $m$) approximation is $\lambda_*\approx 2.545$, a difference of about $2\%$ from the value for the 24-spin Ising model.  {\bf Bottom.} Corresponding trajectory of  $F_t^m$ as computed by the TbS method. The best (highest $m$) approximation is $f_*\approx 0.014$, a difference of almost $98\%$ from the value for the 24-spin model.}
\label{IsingTrunc}
\end{figure}

\begin{figure}[!htb]
\centerline{
\includegraphics[height=22em]{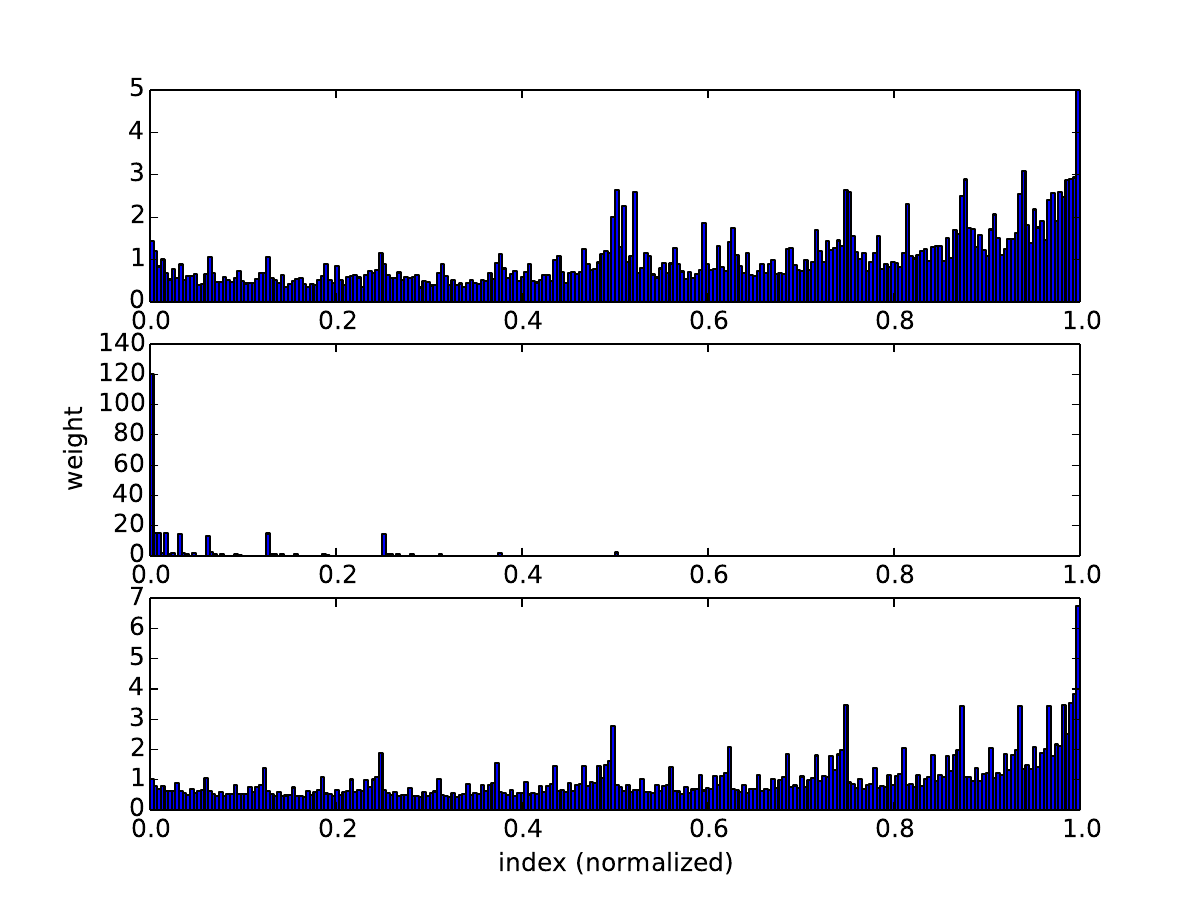}}
\caption{{\bf Top.} Sums of the values of the approximation, $V_t^m$, of the dominant eigenvector of the 50-spin Ising transfer matrix with $B=0.01$ and $T=2.2$, at $t= 10^5$ over $256$ intervals of equal size out of the $2^{50}$ total indices for 
the FRI method with $m=2^{24}$. {\bf Middle.} Corresponding sums for the
TbS method. {\bf Bottom.} Exact eigenvector for the 20-spin Ising model for qualitative comparison.}
\label{IsingHists}
\end{figure}

In Figure~\ref{IsingFRIaves}
we report the trajectory averages of the approximations $\Lambda_t^m$ and $F_t^m$ generated by the FRI scheme (iteration \eqref{random2} using Algorithm~\ref{cr}) with $m=2^{20},\, 2^{21},\, 2^{22},\, 2^{23}$, and $2^{24}$ and $10^5$ total iterations.  The best (highest $m$) approximation of $\lambda_*$ is $2.584$ and the best approximation of $f_*$ is 0.606.  The results for the 24-spin Ising problem are $\lambda_*\approx  2.596$ and $f_*\approx 0.658$ a difference of roughly $0.5\%$ and $8\%$ from the respective approximations generated by the FRI method.  
In Figure~\ref{IsingFRItraj} we plot the corresponding trajectories of $\Lambda_t^m$ and $F_t^m$.  These plots strongly suggest that the iteration equilibrates  rapidly (relative to the total number of iterations).  Indeed, we estimate the integrated autocorrelation times\footnote{According to the central limit theorem for Markov processes (assuming it holds), for large $t$ the variance of the trajectory average of $\Lambda_t^m$ should be $\sigma^2 \tau / t$ where $\sigma^2$ is the infinite $t$ limit of the variance of $\Lambda_t$ and $\tau$ is the integrated autocorrelation time of $\Lambda_t^m.$  Roughly, it measures the number of iterations between independent $\Lambda_t^m.$ } of $\Lambda_t^m$ and $F_t^m$ to be $20.5$ and $274$ respectively.  This in turn suggests that one could achieve a dramatic speedup by running many parallel and independent copies (replicas) of the simulation and averaging the resulting estimates of $\lambda_*$ and $f_*$ though we have not taken advantage of this here.

\footnotetext[8]{We use the term ``confidence intervals'' loosely.  We plot confidence intervals for the values $\lim_{t\rightarrow \infty}\mathbf{E}\left[ \Lambda^m_t\right]$ and $\lim_{t\rightarrow \infty }\mathbf{E}\left[ F_t^m\right]$ (i.e., for finite $m$) and not for  $\lambda_*$ and $f_*.$  In other words, our confidence intervals do not account for bias resulting from a finite choice of $m.$
}

Figure~\ref{IsingTrunc} reports the analogous trajectories (see \eqref{trajave} below) of $\Lambda_t^m$ and $F_t^m$ as generated by iteration \eqref{random2} with the TbS scheme (iteration \eqref{random2} using truncation-by-size) and the same values of $m$.  The best (highest $m$) TbS approximation of $\lambda_*$ is 2.545 and the best TbS approximation of $f_*$ is 0.014, a difference of   almost  $2\%$ and $98\%$ respectively.  In Figure~\ref{IsingHists} we plot the sums of the values of the approximation, $V_t^m$, of the dominant eigenvector of the Ising transfer matrix at $t= 10^5$ over $256$ intervals of equal size out of the $2^{50}$ total indices.  The top plot represents $V_t^m$ as generated by the FRI method and the middle plot represents $V_t^m$ as generated by the TbS approach.  The TbS iteration has converged to a vector with nearly all of weight concentrated on very low indices.  The bottom plot in Figure~\ref{IsingHists} represents the dominant eigenvector for the 24-spin Ising transfer matrix.  The qualitative agreement with the realization of $V_t^m$ as generated by the FRI method is much stronger than agreement with the result of the TbS method.

\begin{remark}
In this problem we  compute the dominant eigenvalue $\lambda_*$ and a projection $f_*$ of the dominant eigenvector $v_*$ of the matrix $K$ defined in equation \eqref{Kpart}  using the FRI in conjunction with the power method.
Using the trajectory averages
\begin{equation}\label{trajave}
\bar\Lambda_t^m = \frac{1}{t} \sum_{s=1}^t \Lambda_t^m
\qquad
\text{and}
\qquad
\bar F_t^m = \frac{1}{t} \sum_{s=1}^t F_t^m
\end{equation}
to estimate $\lambda_*$ and $f_*$ would seem strange had the iterates $\Lambda_t^m$ and $F_t^m$ been generated by the deterministic power method (we have not reported trajectory averages for the deterministic TbS approach).  However, for finite $m$ we do not expect $\Lambda_t^m$ or $F_t^m$ to converge to $\lambda_*$ and $f_*$ as $t$ increases.  Rather we expect that the distribution of $\Lambda_t^m$ and $F_t^m$ will converge to some distribution roughly centered around $\lambda_*$ and $f_*$ respectively.  Though in our convergence results we have not addressed the ergodicity of the Markov process $V_t^m,$ one would expect that reasonable functions of $V_t^m$ such as $\Lambda_t^m$ and $F_t^m$ satisfy a law of large numbers so that, for very large $t,$ the trajectory averages $\bar \Lambda_t^m$ and $\bar F_t^m$ differ from $\lambda_*$ and $f_*$ only by a systematic error (i.e., they converge to the limit of the expectations of $\Lambda_t$ and $F_t^m$ respectively).
\end{remark}

\subsection{A PDE eigenproblem}
For given functions $b(x)$ and $\sigma(x)$ with values 
$\mathbb{R}^n$ and $\mathbb{R}^n\times \mathbb{R}^r$, the backwards Kolmogorov operator
\begin{equation}\label{BK}
L f =  b^{\textsc{t}} D f + \frac{1}{2} \operatorname{trace}\left(\sigma \sigma^{\textsc{t}} D^2 f \right) 
\end{equation}
is the generator of the diffusion process 
\begin{equation}\label{sde}
dX(t) = b(X(t))\,dt + \sigma(X(t))\, dW(t)
\end{equation}
where $W(t)$ is an $r$-dimensional Brownian motion, $Df$ is the vector of first order derivatives of $f$, and $D^2f$ is the matrix of its second order derivatives.  The operator $L$ governs the evolution of moments of $X(t)$ in the sense that
\[
\frac{d}{dt} \mathbf{E}_x\bigl[ f(X(t))\bigr] \Bigr\vert_{t=0} = Lf(x)
\]
(the subscript on the expectation indicates that $X_0=x$).  Note that constant functions are in the kernel of $L$.  The non-trivial eigenfunctions of $L$ all correspond to negative eigenvalues.  The magnitude of the greatest of these negative eigenvalues is the spectral gap and characterizes the rate of convergence of expectations such as $\mathbf{E}_x\bigl[f(X(t))\bigr]$ to their equilibrium (large $t$) values.

In this subsection we consider estimation of the largest negative eigenvalue of $L$ for $b=-DV$ with
\[
V(x^{(1)},\dots,x^{(\ell)}) =  \frac{1}{2}\sum_{j=1}^\ell \cos(2 \pi x^{(i)}_1) \cos(2 \pi x^{(j)}_2) 
+ 2\sum_{j=1}^\ell  \sum_{k=j+1}^\ell  \cos(\pi (x^{(j)}_1-x^{(k)}_1)) \cos(\pi (x^{(j)}_2-x^{(k)}_2))
\]
for $x^{(j)}=(x^{(j)}_1,x^{(j)}_2) \in [-1,1)\times [-1,1)$.  The diffusion coefficient, $\sigma(x)$, is fixed as $\sqrt{2}$.
The function $V$ is the potential energy for a periodic system of $\ell$, 2D-particles, each subject to both an external force as well as a nonlinear spring coupling the particles together.  \eqref{sde} is a model of the dynamics of that system of particles (in a high friction limit).

The equation $Lg_* = \lambda_* g_*$ is first projected onto a Fourier spectral basis, i.e., we assume that
\[
g_*(\vec{x}) = \sum\nolimits_{\vec{\alpha}\in \mathbb{Z}^{2\ell}_N\\
} v({\vec{\alpha}})\, e^{i\pi \vec{\alpha}^\ct \vec{x}}
\]
where $\vec{\alpha} = (\alpha^{(1)},\dots,\alpha^{(\ell)})$ with $\alpha^{(j)} = (\alpha^{(j)}_1,\alpha^{(j)}_1)$,  and the symbol $\mathbb{Z}^{2\ell}_N$ is used to indicate that both $\alpha^{(j)}_1$ and $\alpha^{(j)}_1$ are integers with magnitude less than $N$.  

Suppose that $\hat L$ is the corresponding spectral projection of $L$ (which, in this case, is real).  The matrix $\hat L$ can be decomposed into a sum of diagonal (corresponding to the second order term in $L$) and non-diagonal term (corresponding to first order term in $L$), i.e.,
\[
\hat L = A + D.
\]
\jqw{In this problem the eigenvalues are real and we are trying to find the largest non-zero eigenvalue instead of the eigenvalue with largest magnitude.  We must first transform $\hat L$ so that the largest eigenvalues of $\hat L$ corresponds to the magnitude dominant eigenvalues of the transformed matrix.
As we mentioned in Section~\ref{sec:pert} this can be accomplished using the  matrix obtained from a discrete-in-time approximation of the ODE }
\[
\frac{d}{dt} y = \hat L y
\]
i.e., by exponentiating the matrix $t\hat L$ for very large $t.$
For example, for sufficiently small $\varepsilon>0$, the eigenvalue, $\mu$, of largest magnitude, of the matrix
\begin{equation}\label{cn}
K = e^{\frac{1}{2} \varepsilon D} (I+ \varepsilon \hat L)e^{\frac{1}{2} \varepsilon D}
\end{equation}
is, to within an error of order $\varepsilon^2$, $1+\varepsilon \lambda_*$ where $\lambda_*$ is the eigenvalue of $\hat L$ of largest real part (in our case the eigenvalues are real and non-positive). 
We will apply our iteration schemes to $K$.
By fixing $v_t({\vec{0}}) = 0$ we can guarantee that the approximate solutions all have vanishing integral over $[-1,1)^{2\ell}$, ensuring that the iteration converges to an approximation of the desired eigenvector/value pair 
(instead of to $v(\vec{0}) = 1$, $v(\vec{\alpha})=0$ if $\vec{\alpha}
\neq \vec{0}$).

\begin{remark}
In this problem, rather than estimating the dominant eigenvector of $K,$ our goal is estimate the second largest (in magnitude) eigenvalue of $K.$  Given that we know the largest eigenvalue of $K$ is $1$ with an eigenvector that has value one in the component corresponding to $\vec{\alpha} = \vec{0}$ and zeros in all other components,  we can therefore exactly orthogonalize the iterates $V_t^m$ with respect to the dominant eigenvalue at each iteration (by fixing $V_t^m(\vec{0}) = 0$). We may view this as using FRI in conjunction with a simple case of orthogonal iteration.
\end{remark}

We compare the FRI and TbS approaches  with $N=51$ for the four- and five-particle systems ($\ell = 4,5$).  The corresponding total count of real numbers needed to represent the solution in the five-particle case is $101^{10}\approx 10^{20}$ so only the $\mathcal{O}(1)$ scheme is reasonable.  For $h$ we choose a value of $10^{-3}$.  Our potential $V$ is chosen so that the resulting matrix $\hat L$ (and therefore also $K$) is sparse and its entries are computed by hand.  For a more complicated $V$, the entries of $\hat L$ might have to be computed numerically on-the-fly  or might not be computable at all.  Our ability to efficiently compute the entries of $\hat L$ will be strongly effected by the choice of basis.  For example, if we use a finite difference approximation of $L$ then the entries of $\hat L$ can be computed easily.  On the other hand, if the solution is reasonably smooth, the finite difference approximation will converge much more slowly than an approximation (like the spectral approximation) that more directly incorporates properties of the solution (regularity in this case).

Figure~\ref{PtclsGapAves} plots the trajectory averages over $10^5$ iterations for $\Lambda_t^m$ in the $\ell=4$ case generated by  the FRI method (iteration \eqref{random2} using Algorithm~\ref{cr}) along with corresponding trajectories of $\Lambda_t^m$ as generated by the TbS approach (iteration \eqref{random2} using truncation-by-size).  We present results for both methods  with $m=1,\, 2,\, 3$, and $4\times 10^4$.  Observe that the results from the FRI method appear to have converged on one another while the results generated by the TbS approach show no signs of convergence.  The best (highest $m$) estimate of the eigenvalue generated by FRI is $-2.31$ and the best estimate generated by TbS is $-2.49$.   Figure~\ref{PtclsGap40Traj} plots the trajectory of $\Lambda_t^m$ in the five particle ($\ell=5$) case as generated by the FRI method with $m=10^6$
along with its trajectory average (neglecting the first 500 iterations) of about $-1.3$.  Note that $\Lambda_t^m$ appears to reach its statistical equilibrium rapidly relative to the $2\times 10^3$ total iterations.   Again, the rapid equilibration suggests that statistical error could be removed by averaging over many shorter trajectories evolved in parallel.  


\begin{figure}[!htb]
\centerline{
\includegraphics[height=22em]{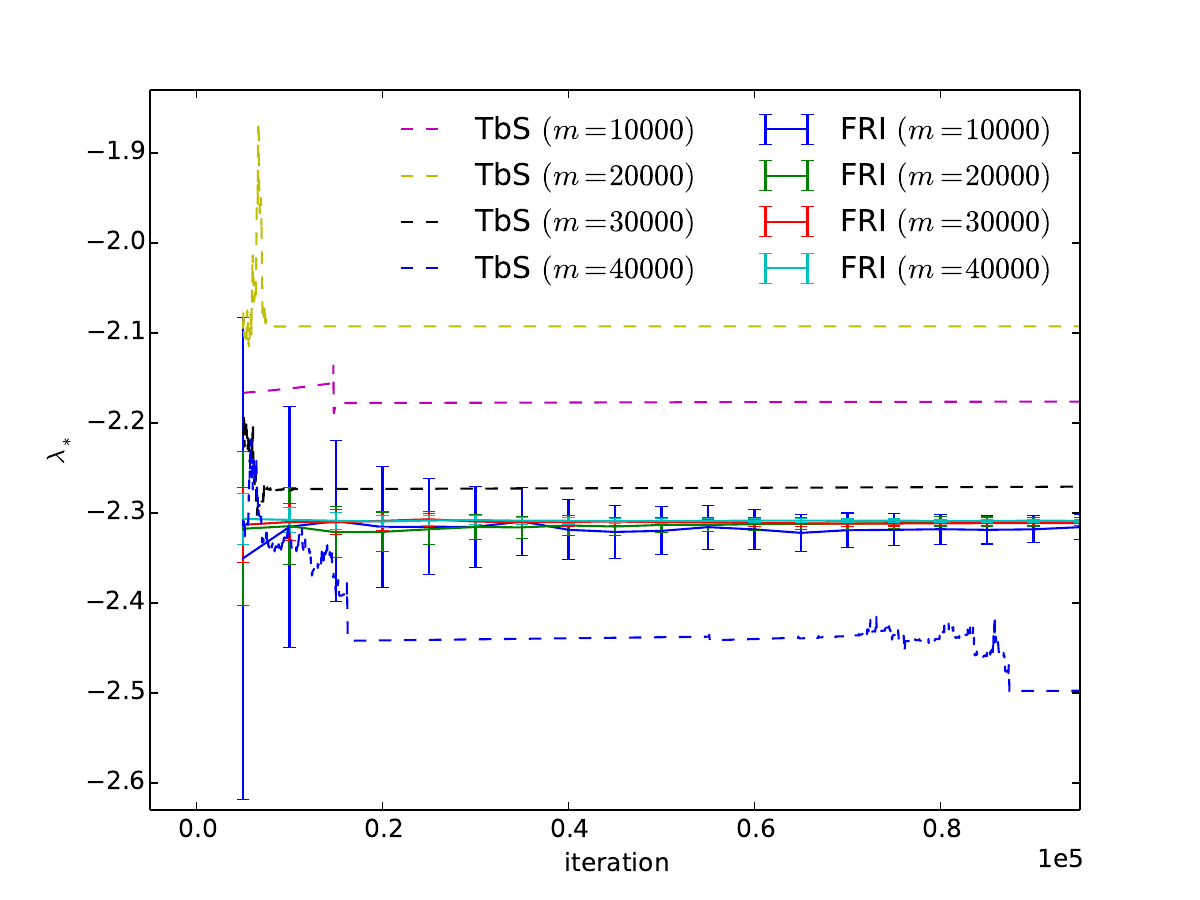}}
\caption{Trajectory averages of the approximation, $\Lambda_t^m$, of the largest negative eigenvalue of a backwards Kolmogorov operator for a four 2D-particle ($8$-dimensional) system, with $95\%$ confidence intervals for the FRI method with $m=1,\,2,\,3$, and $4\times 10^4$.   The operator is discretized using a Fourier basis with $101$ modes per dimension for a total of more than $10^{16}$ basis elements (half that after taking advantage of the fact that the desired eigenvector is real).  The step-size parameter $h$ is set to $10^{-3}$. Also on this graph, trajectories of $\Lambda_t^m$  for the TbS method for the same values of $m$.}
\label{PtclsGapAves}
\end{figure}


\begin{figure}[!htb]
\centerline{
\includegraphics[height=22em]{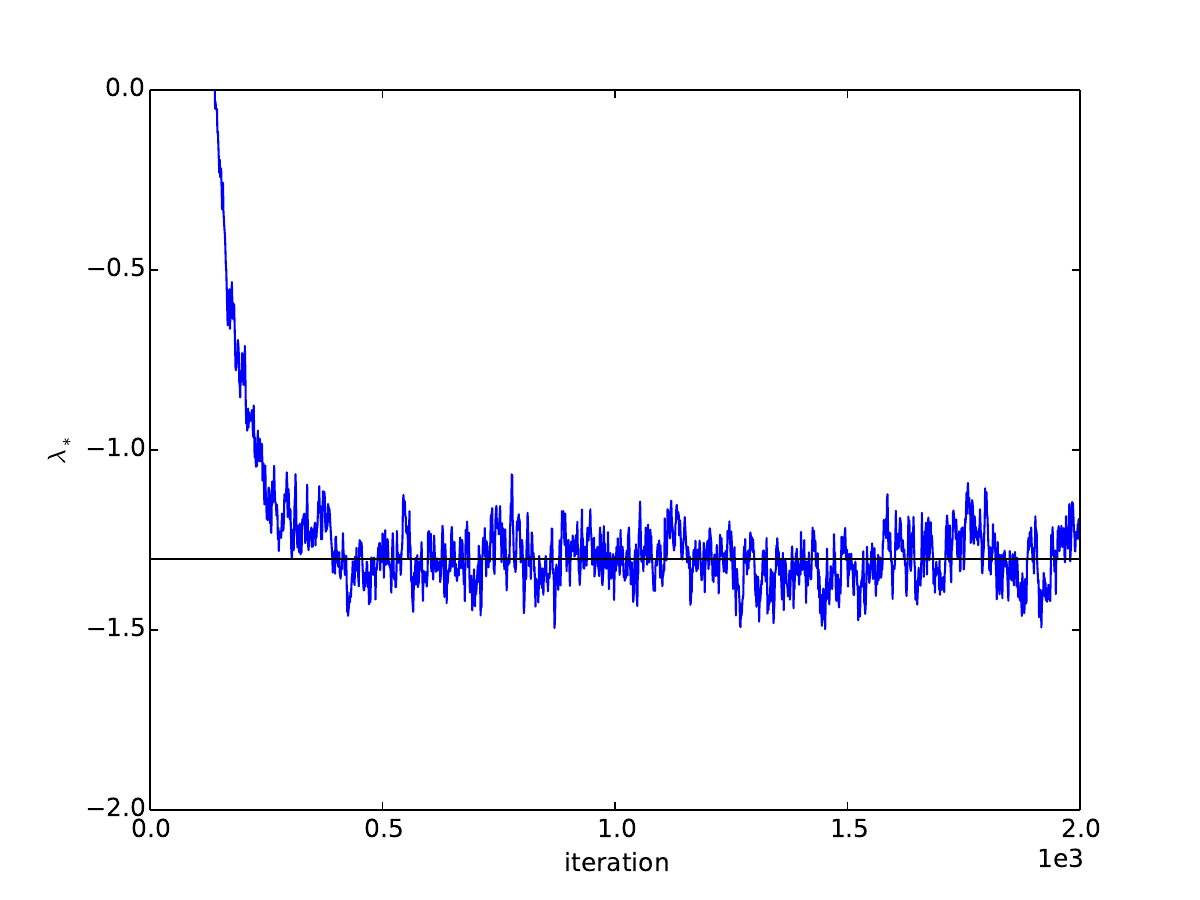}}
\caption{Trajectory  of the approximation, $\Lambda_t^m$ (solid blue line),
of the largest negative eigenvalue of a backwards Kolmogorov operator for the five-particle system as computed by the FRI method with $m=10^6$ over $2\times 10^3$ iterations.  The total dimension of the discretized system is more than $10^{20}$.  The average value of $\Lambda_t^m$ (ignoring the first 500 iterations) is $-1.3$ and is shown by a solid black line.}
\label{PtclsGap40Traj}
\end{figure}

\subsection{A PDE steady state problem}
The adjoint $L^*$ of the operator defined in \eqref{BK} with respect to the standard inner product is called the Fokker--Planck operator.  The operator determines the evolution of the density of the process $X(t)$ defined in \eqref{sde} in the sense that if $\mu$ is that density then
\[
\partial_t \mu = L^*\mu.
\]
An element in the kernel of $L^*$ (a steady state solution of the Fokker--Planck equation) is a density left invariant by $X(t)$.  

For the choice of $b$ and $\sigma$ given in the previous subsection, the steady state solutions are easily seen to be constant multiples of the function 
\[
\mu_*(\vec{x}) = \frac{e^{-V(\vec{x})}}{\int e^{-V(\vec{x})}}.
\]
In most applications, the goal is to compute averages of observables with respect to $\mu_*$.  For example, one might hope to find (up to an additive constant) the effective potential (or free-energy) experienced by particle 1,
\[
F_1(x^{(1)}) = - \log \int \mu_*(\vec{x}) \, dx^{(2)}\cdots dx^{(\ell)}.
\]
For that purpose, explicit knowledge of $\mu_*$ is of little value since one cannot hope to compute integrals with respect to a function of so many variables (up to $101^{8}$ in our tests).  One instead hopes to find a more digestible expression for $\mu_*$.   Notice that if a Fourier expansion
\[
\mu_*(\vec{x}) = \sum\nolimits_{\vec{\alpha}\in \mathbb{Z}^{2\ell}_N} v({\vec{j}})\, e^{i\pi \vec{\alpha}^\ct \vec{x}}
\]
was available then we could compute
\[
\mathcal{F}_1(x^{(1)}) = -\log \sum\nolimits_{\alpha^{(1)}\in \mathbb{Z}^{2}_N} v(\alpha^{(1)},0,\dots,0)\, e^{i\pi \alpha^{(1)T} x^{(1)}}.
\]

As in the previous section\footnote[9]{\jqw{Note that the matrix obtained by $L^2$-projection of the adjoint of a differential operator with real coefficients is the conjugate transpose of the matrix obtained by $L^2$-projection of the differential operator.}}  we discretize the Fokker--Planck operator in a Fourier basis resulting in a finite dimensional linear root finding problem 
\[
\left(K-I\right)v_* = 0
\]
where $K$ is now defined just as in \eqref{cn} but with $A = \hat L^\ct + D.$
  We choose to normalize the solution so that $v(\vec{0}) = 1$ which then results in a linear system 
\[
\left(\bar K - I\right) \bar v_* = r
\]
where $r(\vec{\alpha}) = - K_{\vec{\alpha} \vec{0}}$, $\bar K$ has the row and column corresponding to the index $\vec{0}$ removed, and $\bar v_*$ has the component corresponding to index $\vec{0}$ removed.

\begin{remark}
Note that the linear system $(\bar K-I) \bar v_* = r$ is solved for $v_*$ here  using FRI in conjunction with Jacobi iteration.  With the normalization $v_t(\vec{0})=1,$ this is equivalent to using the power iteration to find the eigenvector corresponding to the largest eigenvalue of $K$ (which is $1$). 
Recalling that here $K \approx I + \varepsilon (\hat L^\ct + D),$ observe that we are (when $\varepsilon$ is small) approximately computing $\lim_{t\rightarrow \infty} \exp(At)v_0$ with  $A = \hat L^\ct + D$, which, since the largest eigenvalue of $A$ is $0$, is the desired eigenvector.
 We repeat that though we know that the dominant eigenvalue of $K$  and we have a formula for $\mu_*,$  the dominant eigenvector of $L^*,$  our goal is to compute projections of $\mu_*$ that cannot be computed by deterministic means. 
 \end{remark}
   
In Figure~\ref{PtclsFE} we present approximations of the function $F_1$ generated by the $\mathcal{O}(1)$ scheme
\begin{align*}
V_{t+1}^m &=\Phi^m_t\left( K V_t^m\right),\\
F_{t+1}^m(\alpha^{(1)}) & = \left(KV_t^m\right)(\alpha^{(1)},0,\dots,0),\\
\overline{F}_{t+1}^m(\alpha^{(1)}) & = (1-\varepsilon_t)\overline{F}_t(\alpha^{(1)}) + \varepsilon_t F_{t+1}^m(\alpha^{(1)}),
\end{align*}
for all $\alpha^{(1)}\in \mathbb{Z}^2_N$ with $\varepsilon_t=(t+1)^{-1}$ and where the independent mappings $\Phi^m_t$ are generated according to Algorithm~\ref{cr}.  The single particle free-energy\footnote[10]{Note that we only approximate $\mathcal{F}_1$ up to the additive constant $-\log\int e^{-V(\vec{x})}.$  In fact, the free energy is typically only defined up to that constant because it is not uniquely specified (one can add a constant to $V$ without changing $\mu_*$).}    as generated by the FRI approach is plotted for $\ell = 2,\,3,\,4$, and $5$  with $m=10,\,200,\,10^4,$ and $10^6$ respectively.  In the two, three, and four particle simulations we use $10^5$ iterations.  Again we choose $N = 51$ and  $h= 10^{-3}$.  The high cost per iteration in the five particle case restricts our simulations to $2\times 10^3$ iterations. 
In the four particle case, for which we have validated the FRI solution by simulations with higher values of $m$ ($m=4\times 10^4$), the free energy profile produced by the TbS approach differs from the FRI result by as much as $100\%$.  We take slight advantage of the particle exchange symmetry and, at each iteration, replace $\left(KV_{t}^m\right)(\alpha^{(1)},0,\dots,0)$ in the above equation for $F_{t+1}^m$ by the average of all $\ell$ components of the form $\left(KV_{t}^m\right)(0,\dots,\alpha^{(k)},0,\dots,0)$.  Note that in the expansion of $\mu_*$, we know that $v(\vec{\alpha})$ is unchanged when we swap the indices $\alpha^{(j)}$ and $\alpha^{(k)}$ corresponding to any two particles.  This fact could be leveraged to greatly reduce the number of basis functions required to accurately represent the solution.  We have not exploited this possibility.

Though it is not accurate, the TbS scheme is substantially more stable on this problem.  We assume that the relative stability of the TbS scheme is a manifestation of the fact that TbS is not actually representing the high wave number modes that are responsible for stability constraints.  Nonetheless, simulating higher dimensional systems would require modifications in our approach.  In particular it might be necessary to identify a small number of components of the solution that should always be resolved (never set to zero).  For example, for this problem one might choose to resolve some number of basis functions for each particle that are independent of the positions of the other particles.

\begin{figure}[!htb]
\centerline{
\includegraphics[height=22em]{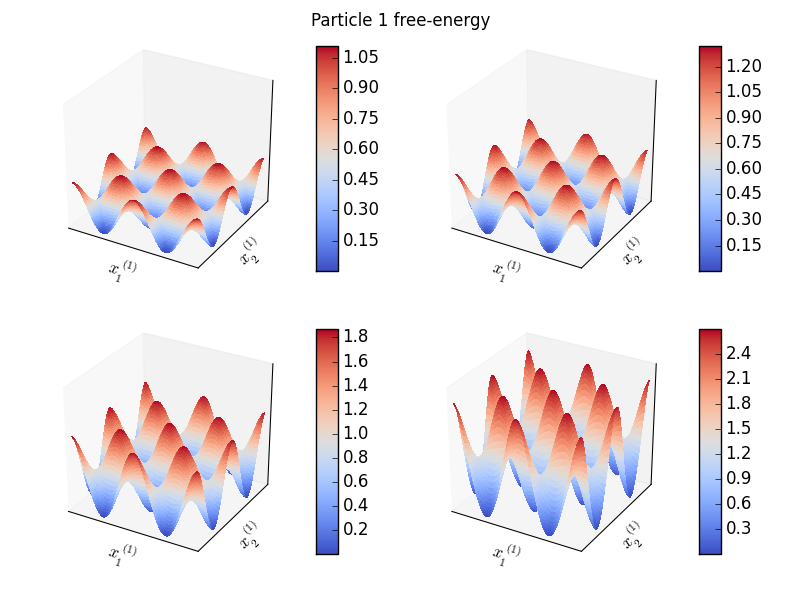}}
\caption{Free energy landscape experienced by a single particle for the (clockwise from top left) two, three, four, and five 2D-particle systems.  The surfaces were generated using the FRI method with $m=10,\,200,\,10^4,$ and $10^6$ respectively and $h=10^{-3}$.  The two-, three-, and four-particle simulations were run for $10^5$ iterations.   The five particle simulation was substantially more expensive per iteration and was run for only $2\times 10^3$ iterations.   The number of Fourier modes used to represent the solution in all simulations is 101 per dimension for a total of more than $10^8,\, 10^{12},\, 10^{16}$, and $10^{20}$ basis functions (half that after taking advantage of the fact that the solution is real).   As expected, the free energy basins deepen as the number of particles grows.  In the four particle case (for which we have high confidence in the estimate produced by FRI) the error from the TbS approach (the results of which we do not plot) is roughly $100\%$ at peaks of the free energy landscape.}
\label{PtclsFE}
\end{figure}

\section{Discussion}
We have introduced a family of fast, randomized iteration schemes for   eigenproblems, linear systems, and matrix exponentiation.  
Traditional iterative methods for
numerical linear algebra were created in part to deal with instances where
the coefficient matrix $A$ (of size $\mathcal{O}(n^2)$) is too big to store but where the operation $x \mapsto Ax$ can nonetheless be carried out. The  iterative
methods in this paper are intended for instances in which the ultimate goal is to compute  $f\cdot x$ for some predetermined vector $f$ but the cost of assembling the product $Ax$ ($\mathcal{O}(n^2)$) is too high and even for cases in which   the solution vector $x$
(of size $\mathcal{O}(n)$) is too big to store.  
We provide basic theoretical results justifying the approach and illustrating in particular that the cost of the schemes can be independent of dimension for some problems.  Generally we expect sublinear scaling with dimension of both cost and storage requirements as observed in our numerical experiments.  The identification of general conditions guaranteeing sublinear scaling for FRI schemes is not addressed in this article but seems a very interesting direction for future research.


A completely deterministic approach to iterative problems related to the methods proposed in this article is the simple thresholding by size (TbS) in which, at each iteration, the smallest entries in the approximation are set to zero.  An adaptive version of the TbS approach has recently been advocated for a wide range of applications (see \cite{SchaefferCaflisch:2013:SparsePDE,OzolinsLai:2013:Sparse2,OzolinsLai:2014:Sparse3}).
 Like TbS our randomized schemes also rely on the enforcement of sparsity and also tend to set small entries in the approximate solution to zero.  While the TbS approach can be effective on some problems with sparse solutions, their error in general will be strongly dependent on system size and we find that it performs very poorly on our test problems relative to FRI.

The core concept behind the FRI schemes introduced in this article is the notion that, by randomly setting entries in vectors to zero, while maintaining a statistical consistency property, we can dramatically reduce the cost and storage of standard iterative schemes.  One can view our FRI schemes as an attempt to blur the line separating Markov chain Monte Carlo (MCMC), which is effective in extremely high dimensional settings but is limited to a relatively narrow class of problems and often does not allow the user to take full advantage of known properties of the solution (e.g.\ smoothness or symmetry properties as in \cite{BoothThom:2009:CIQMC}), and traditional deterministic schemes, which are effective on a very general set of relatively low dimensional problems.  As for MCMC approaches, if one settles for computing low dimensional projections of the full solution then not every element of the state space need be visited and  effective FRI schemes with per iteration cost and storage requirements independent  of system size can be derived (as for MCMC the validity of this statement depends on the particular sequence of problems considered).   Also as for MCMC we expect that, when deterministic alternatives are available, they will outperform our randomized schemes.  For matrices of the size considered in all of our numerical tests, deterministic alternatives are \emph{not} available.


Experience with diffusion Monte Carlo in the context of quantum Monte Carlo simulations suggests that our randomized schemes will be most useful if applied after considerable effort has been expended on finding changes of variables that either make the desired solution as sparse as possible (reducing both bias and variance) or reduce bias by some other means.  In many cases this will mean applying our randomized schemes only after one has obtained an estimate of the solution by some deterministic method applied to a reduced dimensional version of the target problem.

\section*{Acknowledgments}
JQW would like to thank Eric Cances, Tony Lelievre,  and Matthias Rousset, for their hospitality and  helpful discussions during a visit to ENPC that coincided with the early stages of this work.   Both authors would like to thank Mihai Anitescu, Alexandre Chorin, Petros Drineas, Risi Kondor, Jianfeng Lu, Omiros Papaspiliopoulos, Panos Stinis, and the anonymous referees, who all made comments that  strongly affected this article's structure and content. 
LHL's work is generously supported by DARPA D15AP00109, AFOSR FA9550-13-1-0133, NSF IIS-1546413, DMS-1209136, and DMS-1057064. JQW's effort  was supported by the Advance Scientific Computing Research program within the DOE Office of Science through award DE-SC0014205 as well as through a contract from Argonne, a U.S.\ Department of Energy Office of Science laboratory.

\bibliographystyle{plain}
\bibliography{references}

\newpage

\pagenumbering{arabic}

\appendix

\section*{List of proofs}

\begin{proof}[Proof of Lemma \textup{\ref{opnormrep2}}]
First, note that 
\[
\mathbf{E}\left[ \lVert G X \rVert_1^2 \right]
 = \mathbf{E}\Bigl[ \Bigl( \sum\nolimits_{i=1}^n \lVert g_i\rVert_\infty\, \lvert \hat g_i^\ct X \rvert \Bigr)^2\Bigr],
 \]
 where $g_i$ is the $i$th column of $G$ and $\hat g_i = g_i/\lVert g_i\rVert_\infty$.  Using the condition $\lVert G\rVert_{\infty,*}\leq 1$
and Jensen's inequality, we find that
\[
\mathbf{E}\left[ \lVert G X \rVert_1^2 \right]
\leq  \sum\nolimits_{i=1}^n \lVert g_i\rVert_\infty \,
\mathbf{E}\left[ \lvert \hat g_i^\ct X\rvert^2 \right]
\leq \opnorm{X}^2.
\]
The other inequality follows by noting that for any $f\in \mathbb{C}^n$ with $\lVert f\rVert_\infty \leq 1,$ the matrix $G$ with first row equal to $f^\ct$ and all other rows zero satisfies the constraint $\lVert G\rVert_{\infty,*}\leq 1$. That \eqref{eq:inftydual} is the dual norm of the $\infty$-norm follows from straightforward verification, or see \cite[Proposition~7.2]{FL}.
\end{proof}

 \begin{proof}[Proof of Lemma \textup{\ref{opnormlowerbound}}]
 The result holds in $n=1$ dimensions.  Suppose that the result holds in $n-1$ dimensions.  We will show that it must also therefore hold in $n$ dimensions and conclude, by induction, that the result holds in any dimension.

Let $\tilde A$ be the $(n-1)\times (n-1)$ principle submatrix of an $n\times n$ matrix $A.$  For any vector $f\in \mathbb{C}^{n}$ we can write
\begin{align*}
f^\ct  A f &= \sum\nolimits_{i=1}^n \lvert f_i \rvert^2 A_{ii} + 2 \Re \left[\sum\nolimits_{i=1}^n \sum\nolimits_{j=1}^{i-1}  \bar{f_i} A_{ij} f_j\right]\\
& = \tilde{f}^\ct  \tilde A \tilde{f} + \lvert f_n\rvert^2 A_{nn} 
+ 2 \Re \left[ \bar{f_n} \sum\nolimits_{j=1}^{n-1}  A_{nj} \tilde{f}_j\right],
\end{align*}
where $\tilde f\in \mathbb{C}^{n-1}$ has entries equal to the first $n-1$ entries of $f.$

By the induction hypothesis, we can choose the first $n-1$ entries of $f$ (i.e., $\tilde f$)  so that the right-hand side of the last display is not less than 
\[
\sum\nolimits_{i=1}^{n-1} A_{ii} + \lvert f_n\rvert^2 A_{nn} 
+ 2  \Re \left[\bar{f_n} \sum\nolimits_{j=1}^{n-1}  A_{nj} \tilde{f}_j\right].
\]
If, for this choice  of $\tilde f$, $\sum\nolimits_{j=1}^{n-1}  A_{nj} \tilde{f}_j$
is nonzero, then choose $f_n$ as
\[
f_n = \frac{ \sum\nolimits_{j=1}^{n-1}  A_{nj} \tilde{f}_j}{\left\lvert \sum\nolimits_{j=1}^{n-1}  A_{nj} \tilde{f}_j\right\rvert }.
\]
Otherwise set $f_n =1.$  With the resulting choice of $f_n,$
\[
 \lvert f_n\rvert^2 A_{nn} 
+ 2 \Re \left[ \bar{f_n} \sum\nolimits_{j=1}^{n-1}  A_{nj} \tilde{f}_j \right]
 \geq  A_{nn}.
\]
We have therefore shown that
\[
 \sup_{\lVert f\rVert_\infty\leq 1}
 f^\ct  A f \geq \sum\nolimits_{i=1}^n A_{ii}. \qedhere
 \]
 \end{proof}

 \begin{proof}[Proof of Theorem \textup{\ref{stable1}}]
Let $V^m_t$ be generated by \eqref{Rit}. Let $Y_t^m = \Phi^m_t(V_t^m)$ and notice that
\begin{align*}
\mathcal{U}(V_t^m)  &= \mathcal{U}(\mathcal{M}(Y_{t-1}^m))\\
&\leq R  + \alpha\, \mathcal{U}( V_{t-1}^m) + \alpha  \left( \mathcal{U}(Y_{t-1}^m) - \mathcal{U}(V_{t-1}^m)\right).
\end{align*}
Using the fact that $\mathcal{U}$ is twice differentiable with bounded second derivative, this last expression 
is bounded above by
\[
\mathcal{U}(V_t^m)  \leq  R+\alpha\, \mathcal{U}(V_{t-1}^m) + \alpha \nabla \mathcal{U}(V_{t-1}^m) \left( Y_{t-1}^m-V_{t-1}^m\right) + 
\frac{\alpha \sigma}{2} \lVert G\left( Y_{t-1}^m-V_{t-1}^m\right) \rVert_1^2.
\]
Taking the expectation and using \eqref{Cunbiased} yields 
\[
\mathbf{E}\left[\mathcal{U}(V_t^m) \right] \leq  R+\alpha\, \mathbf{E}\left[\mathcal{U}(V_{t-1}^m)\right]  + 
\frac{\alpha \sigma}{2} \mathbf{E}\left[ \lVert G\left( Y_{t-1}^m-V_{t-1}^m\right) \rVert_1^2\right].
\]
An application of Lemma~\ref{opnormrep2}  reveals that
\[
\mathbf{E}\left[\lVert G\left( Y_{t-1}^m-V_{t-1}^m\right) \rVert_1^2\right]
\leq \opnorm{Y_{t-1}^m-V_{t-1}^m}^2.
\]
As a consequence,  noting \eqref{genPhierr}, we arrive at the upper bound
\begin{align*}
\mathbf{E}\left[\mathcal{U}(V_t^m)\right] &\leq R +
\alpha\, \mathbf{E}\left[\mathcal{U}(V_{t-1}^m)\right]  +  
\frac{\alpha \gamma^2 \sigma}{2 m}  \mathbf{E}\left[ \lVert V_{t-1}^m\rVert_1^2\right]\\
& \leq
R + \alpha \left( 1 + \frac{ \beta \gamma^2 \sigma}{2 m}\right) \mathbf{E}\left[ \mathcal{U}(V_{t-1}^m)\right],
\end{align*}
from which we can conclude that
\[
\mathbf{E}\left[ \lVert V_t^m\rVert_1^2\right] \leq
\beta \mathbf{E}\left[\mathcal{U}(V_t^m)\right] \leq  \beta R\Biggl[ \frac{ 1- \alpha^t \bigl( 1 + \frac{ \beta \gamma^2 \sigma}{2 m}\bigr)^t}
{ 1- \alpha \bigl( 1 + \frac{ \beta \gamma^2 \sigma}{2 m}\bigr)}\Biggr]
+ \beta \alpha^t \biggl( 1 + \frac{ \beta \gamma^2 \sigma}{2 m}\biggr)^t \mathcal{U}(V_0^m). \qedhere
\]

\end{proof}

 \begin{proof}[Proof of Theorem \textup{\ref{ft1}}]
 We begin with a standard expansion of the scheme's error.
 \begin{align*}
\opnorm{V_t^m - v_t} &= \opnorm*{ V_t^m -
 \mathcal{M}_0^t(v_0)}\\
 &= \opnorm*{ \sum\nolimits_{r=0}^{t-1} \mathcal{M}_{r+1}^t (V_{r+1}^m) -
 \mathcal{M}_{r}^t(V_{r}^m)}.
 \end{align*}
 Now notice that if we define $Y_r^m = \Phi^m_r(V_r^m),$ then $V_{r+1}^m = \mathcal{M} (Y_r^m)$ and the last equation becomes
 \[
 \opnorm{V_t^m - v_t} =  \opnorm*{ \sum\nolimits_{r=0}^{t-1} \mathcal{M}_{r}^t (Y_{r}) -
 \mathcal{M}_{r}^t(V_{r}^m)}.
 \]
 The right-hand side of the last equation is bounded above by
 \[
 \opnorm*{ \sum\nolimits_{r=0}^{t-1} \mathcal{M}_{r}^t (Y_{r}) -
 \mathbf{E}[ \mathcal{M}_r^t (Y_r)\mid V_r^m]}
 + \sum\nolimits_{r=0}^{t-1} \opnorm{ \mathbf{E}[ \mathcal{M}_r^t (Y_r)\mid V_r^m]-
 \mathcal{M}_{r}^t(V_{r}^m)}.
 \]

 Considering the first term in the last display, note that, for any fixed $f\in \mathbb{C}^n,$
{\footnotesize
 \begin{multline*}
  \mathbf{E}\Bigl[ \bigl\lvert f^\ct \sum\nolimits_{r=0}^{t-1} \bigl(
 \mathcal{M}_{r}^t (Y_{r}) -
 \mathbf{E}[ \mathcal{M}_r^t (Y_r) \mid V_r^m]\bigr) \bigr\rvert^2\Bigr]
= \sum\nolimits_{r=0}^{t-1} \mathbf{E}\Bigl[ \bigl\lvert  f^\ct \bigl(
 \mathcal{M}_{r}^t (Y_{r}) -
 \mathbf{E}[ \mathcal{M}_r^t (Y_r) \mid V_r^m]\bigr) \bigr\rvert^2\Bigr] \\
+ 2 \sum\nolimits_{s=0}^{t-1}\sum\nolimits_{r= s+1}^{t-1}
\Re\Bigl\{ \mathbf{E}\Bigl[ \bigl(  f^\ct (
 \mathcal{M}_{r}^t (Y_{r}) -
 \mathbf{E} [ \mathcal{M}_r^t (Y_r) \mid V_r^m ]) \bigr) \times
\bigl( \overline{ f^\ct (
 \mathcal{M}_{s}^t (Y_{s}) -
 \mathbf{E}[ \mathcal{M}_s^t (Y_s) \mid V_s^m])} \bigr)\Bigr]\Bigr\}.
\end{multline*}}%
Letting $\mathcal{F}_r$ denote the $\sigma$-algebra generated by $\{V_s^m\}_{s=0}^r$ and $\{Y_r^m\}_{s=0}^{r-1},$  for $s<r$ we can write
{\footnotesize
\begin{multline*}
\mathbf{E}\Bigl[ \bigl(  f^\ct (
 \mathcal{M}_{r}^t (Y_{r}) -
 \mathbf{E}[ \mathcal{M}_r^t (Y_r)\mid V_r^m]) \bigr) \times
\bigl( \overline{ f^\ct (
 \mathcal{M}_{s}^t (Y_{s}) -
 \mathbf{E}[ \mathcal{M}_s^t (Y_s)\mid V_s^m])} \bigr)\Bigr]
\\ = \mathbf{E}\Bigl[ \mathbf{E}\bigl[ f^\ct (
 \mathcal{M}_{r}^t (Y_{r}) -
 \mathbf{E} [ \mathcal{M}_r^t (Y_r) \mid V_r^m])\bigm| \mathcal{F}_r\bigr] \times
\bigl(\overline{f^\ct  (
 \mathcal{M}_{s}^t (Y_{s}) -
 \mathbf{E}[ \mathcal{M}_s^t (Y_s) \mid V_s^m])}\bigr)\Bigr].
\end{multline*}}%
Because, conditioned on $V_r^m,$ $Y_r^m$ is independent of $\mathcal{F}_r$, the expression above vanishes exactly.

Supremizing over the choice of $f$, we have shown that
\[
 \opnorm{V_t^m - v_t}
 \leq 
 \Bigl(  \sum\nolimits_{r=0}^{t-1}\opnorm{ \mathcal{M}_{r}^t (Y_{r}) -
 \mathbf{E} [ \mathcal{M}_r^t (Y_r) \mid V_r^m ]}^2\Bigr)^{1/2}
+ \sum\nolimits_{r=0}^{t-1} \opnorm{ \mathbf{E}[\mathcal{M}_r^t (Y_r)\mid V_r^m]-
 \mathcal{M}_{r}^t(V_{r}^m)}.
 \]
 Expanding the term inside of the square root, we find that
 \begin{align*}
  \opnorm{V_t^m - v_t}&\leq
   \left(  \sum\nolimits_{r=0}^{t-1}\left(\opnorm{ \mathcal{M}_{r}^t (Y_{r}) -
  \mathcal{M}_r^t (V_r^m)}
 +\opnorm{\mathbf{E}[ \mathcal{M}_r^t (Y_r) \mid V_r^m ]-   \mathcal{M}_r^t (V_r^m) 
 } \right)^2\right)^{1/2}\\
&\hspace{2cm}+ \sum\nolimits_{r=0}^{t-1} \opnorm{ \mathbf{E}[ \mathcal{M}_r^t (Y_r)\mid V_r^m]-
 \mathcal{M}_{r}^t(V_{r}^m)}\\
 & \leq \Bigl(  \sum\nolimits_{r=0}^{t-1}\opnorm{ \mathcal{M}_{r}^t (Y_{r}) -
  \mathcal{M}_r^t (V_r^m)}
 ^2\Bigr)^{1/2} +
 \Bigl(  \sum\nolimits_{r=0}^{t-1}\opnorm{\mathbf{E}[ \mathcal{M}_r^t (Y_r)\mid V_r^m]-   \mathcal{M}_r^t (V_r^m) 
 }^2\Bigr)^{1/2} \\
&\hspace{2cm}+ \sum\nolimits_{r=0}^{t-1} \opnorm{ \mathbf{E}[ \mathcal{M}_r^t (Y_r)\mid V_r^m]-
 \mathcal{M}_{r}^t(V_{r}^m)},
 \end{align*}
 where, in the second inequality, we have used the triangle inequality for the $\ell^2$-norm in $\mathbb{R}^t.$  
Noting that $\mathbf{E}\left[ A(V_r^m)(Y_r-V_r^m)\mid V_r^m\right] = 0$  yields
 \[
\mathbf{E}[ \mathcal{M}_r^t (Y_r)\mid V_r^m]-
 \mathcal{M}_{r}^t(V_{r}^m) = 
 \mathbf{E}\bigl[ \bigl(\mathcal{M}_r^t-A_r\bigr) (Y_r)\bigm| V_r^m \bigr]
 - \bigl(\mathcal{M}_{r}^t-A_r\bigr)(V_{r}^m).
 \]
 As a consequence, 
 applying our assumptions \eqref{cntr1} and \eqref{cntr2}, we obtain the upper bound
 \[
  \opnorm{V_t^m - v_t} \leq (L_1+L_2) \left(  \sum\nolimits_{r=0}^{t-1}\alpha^{2(t-r)}\opnorm{ \Phi^m_r(V_r^m) -
 V_r^m}
 ^2\right)^{1/2} + L_2\sum\nolimits_{r=0}^{t-1}\alpha^{t-r} \opnorm{ \Phi^m_r(V_r^m)-
 V_{r}^m}^2.
 \]
 Bounding the error from the random compressions, we arrive at the error bound
  \[
  \opnorm{V_t^m - v_t} \leq  \frac{\gamma (L_1+L_2) }{\sqrt{m}}\left(  \sum\nolimits_{r=0}^{t-1}\alpha^{2(t-r)}\mathbf{E}\left[ \lVert V_r^m\rVert_1^2\right]
\right)^{1/2} + \frac{\gamma^2L_2}{m} \sum\nolimits_{r=0}^{t-1}\alpha^{t-r} 
\mathbf{E}\left[ \lVert V_r^m\rVert_1^2\right].   \qedhere
 \]
 
\end{proof}

\begin{proof}[Proof of Corollary \textup{\ref{nonneg}}]
We have already seen that when $\mathcal{M}(v) = Kv$ we can take $\alpha=\lVert K \rVert_1$ in the statement of Theorem \ref{ft1} to verify conditions \eqref{cntr1} and \eqref{cntr2}.  We have also commented above that when $K$ is nonnegative, the quantities $\mathbf{E}\left[ \lVert V_r^m\rVert_1^2\right]$ can be bounded independently of $n.$

When $\mathcal{M}(v) = Kv/ \lVert Kv \rVert_1,$ bounding the size of the  iterates is not an issue, but it becomes slightly more difficult to verify \eqref{cntr1} and \eqref{cntr2}.  That $K$ is aperiodic and irreducible implies that the dominant left and right eigenvectors, $v_L$ and $v_R,$  of $K$ are unique and have all positive entries.  Because power iteration is invariant to scalar multiples of $K$ we can assume that the dominant eigenvalue of $K$ is~1.  We will assume that $v_L$ is normalized so that $\lVert v_L\rVert_\infty=1$ and that $v_R$ is normalized so that $v_L^\tp v_R = 1.$  Let $D$ be the diagonal matrix with $D_{ii} = (v_L)_i$ (i.e., $D \mathbbm{1} = v_L$).  Our matrix $K$ can  be written $K = D^{-1} S D$ where $S$ is an aperiodic, irreducible, column-stochastic matrix.  Let
\[
\widetilde K = K - v_R v_L^\tp = D^{-1} S P D,
\]
where we have defined the projection $P = I - Dv_R \mathbbm{1}^\tp.$  Note that $\lVert P\rVert_1 \leq 2$ and that 
$PSP = SP$ so that for any positive integer $r,$ $\tilde K^r = D^{-1} S^r P D.$ 
Letting 
\[
C = \frac{1}{\min_j \{ (v_L)_j\}}\geq 1
\]
we find that, for any positive integer $r,$
\begin{equation*}
 \lVert \widetilde K^r  \rVert_1
\leq \lVert D^{-1} \rVert_1 \lVert D \rVert_1 \lVert S^r P \rVert_1
\leq 2\, C \sup_{\substack{\lVert v\rVert_1 = 1\\ \mathbbm{1}^\tp v=0}} \lVert S^r v\rVert_1
\leq 2\, C \,
\alpha^r
\end{equation*}
where 
\[
\alpha = \sup_{\substack{\lVert v\rVert_1 = 1\\ \mathbbm{1}^\tp v = 0}} \lVert S v \rVert_1
\] 
Aperiodicity and irreducibility of $S$ implies that $\alpha<1$.
We also have that
\[
\sup_{\substack{ v_L^\text{\tiny T} v = 1 }}\lVert K^r v\rVert_1 \leq C  \qquad \text{and}\qquad \inf_{\substack{ v_L^\text{\tiny T} v = 1 \\ v_j\geq 0\,\forall j}} \lVert K^r v\rVert_1 \geq 1.
\]

Now let $u$ and $v$ be any two non-negative vectors normalized so that $ v_L^\text{\tiny T} u =  v_L^\text{\tiny T} v = 1$ and, for $\theta\in [0,1],$ define $w_\theta = (1-\theta) u + \theta v.$   Note that $w_\theta$ also has non-negative entries and that $v_L^\text{\tiny T} w_\theta = 1.$  For any fixed $f\in \mathbb{R}^n$ with $\lVert f\rVert_\infty \leq 1,$ define the function
\[
{\varphi_r}(u, v; \theta) =  \frac{f^\text{\tiny T} K^r w_\theta}{\lVert K^r w_\theta\rVert_1}  -  \frac{f^\text{\tiny T} K^r u}{\lVert K^r u\rVert_1}.
\]
Our goal is to establish bounds on 
\[
{\varphi_r}(u, v; 1) =  \frac{f^\text{\tiny T} K^r v}{\lVert K^r v\rVert_1} -  \frac{f^\text{\tiny T} K^r u}{\lVert K^r u\rVert_1}.
\]

To that end note that
\[
\frac{d}{d\theta} {\varphi_r}(u, v; \theta) = \frac{f^\text{\tiny T} K^r (v-u)}{\lVert K^r w_\theta\rVert_1}
- \frac{(f^\text{\tiny T} K^r w_\theta)(\mathbbm{1}^\text{\tiny T} K^r (v-u))}{\lVert K^r w_\theta\rVert_1^2}
\]
and
\[
\frac{d^2}{d\theta^2} {\varphi_r}(u, v; \theta) = -2 \frac{(f^\text{\tiny T} K^r (v-u))(\mathbbm{1}^\text{\tiny T} K^r (v-u))}{\lVert K^r w_\theta\rVert_1^2}
+ 2\frac{(f^\text{\tiny T} K^r w_\theta)(\mathbbm{1}^\text{\tiny T} K^r (v-u))^2}{\lVert K^r w_\theta\rVert_1^3}.
\]
Observing that $K^r (v-u) = \tilde K^r (v-u),$ and applying our bounds we find that
\begin{align}
\lvert {\varphi_r}(u, v; 1)\rvert &\leq \max_\theta \left\lvert \frac{d}{d\theta} {\varphi_r}(u, v; \theta)\right\rvert \notag\\
&\leq \lvert f^\text{\tiny T} \tilde K^r (v-u)\rvert + C \lvert \mathbbm{1}^\text{\tiny T} \tilde K^r (v-u)\rvert \notag\\
&\leq 4\, C^2\, \alpha^r\, \lVert G(v-u)\rVert_1\label{stab1check}
\end{align}
where $G\in \mathbb{R}^{n\times n}$ is the matrix with first row equal to 
$ {f^\tp \tilde K^r}/{\lVert 2 f^\tp \tilde K^r \rVert_\infty}$,
second row equal to $ {\mathbbm{1}^\tp \tilde K^r}/{\lVert 2 \mathbbm{1}^\tp \tilde K^r \rVert_\infty}$, 
and all other entries equal to 0.

Defining the matrix valued function
\[
A_r(u) = \frac{1}{\lVert K^r u\rVert_1} \left[ I - \frac{K^r u \mathbbm{1}^\text{\tiny T}}{\lVert K^r u\rVert_1}\right]K^r
\]
we observe that 
\[
\frac{d}{d\theta} {\varphi_r}(u, v; 0) = f^\text{\tiny T} A_r(u) (v-u)
\]
 so that 
\begin{align}
\lvert {\varphi_r}(u, v; 1)  - f^\text{\tiny T} A_r(u) (v-u) \rvert
& \leq \frac{1}{2} \max_\theta \left\lvert \frac{d^2}{d\theta^2} {\varphi_r}(u, v; \theta)\right\rvert \notag\\
& \leq  \lvert f^\text{\tiny T} \tilde K^r (v-u)\rvert  \lvert \mathbbm{1}^\text{\tiny T} \tilde K^r (v-u)\rvert + C \lvert \mathbbm{1}^\text{\tiny T} \tilde K^r (v-u)\rvert^2 \notag\\
& \leq 16\, C^3\, \alpha^{2r}\, \lVert G(v-u)\rVert_1^2
\label{stab2check}
\end{align}

Expressions \eqref{stab1check} and \eqref{stab2check} verify the stability conditions in the statement of Theorem \ref{ft1} with $L_1$ and $L_2$ dependent only on $C$ yielding the first term on the right-hand side of \eqref{nonnegtoterr}.  The second term follows similarly when one observes that \eqref{cntr1} implies 
\[
\sup_{v,\tilde v\in \mathcal{X}}\frac{\lVert \mathcal{M}^r_s(v) -   \mathcal{M}^r_s(\tilde v)\rVert_1}{\lVert v - \tilde v\rVert_1}
\leq L_1 \alpha^{r-s}.
\] 
\end{proof}

\begin{proof}[Proof of Lemma \textup{\ref{pertlemma}}]
If $Y_t^m = \Phi^m_t( V_t^m),$ then
\begin{align*}
\mathbf{E}\left[ \lvert f^\ct  \Phi^m_t(V_t^m) - f^\ct V_t^m \rvert^2 \mid Y_{t-1}^m\right] &= \mathbf{E}\left[ \lvert f^\ct \Phi^m_t\left(Y_{t-1}^m + \varepsilon b(Y_{t-1}^m)\right) - f^\ct \left(Y_{t-1}^m + \varepsilon b(Y_{t-1}^m) \right)\rvert^2 \mid Y_{t-1}^m\right]\\
& \leq \gamma_p \frac{\varepsilon}{m} \lVert b(Y_{t-1}^m)\rVert_1  \lVert V_t^m\rVert_1 
\end{align*}
for some constant $C.$
Our assumed bound on the growth of $b$ along with \eqref{phibnd} implies that
\[
\mathbf{E}\left[ \lVert b(Y_{t-1}^m)\rVert_1^2\right] \leq C' \left(1 + \mathbf{E}\left[ \lVert V_{t-1}^m\rVert_1^2\right]\right)
\]
for some constant $C'.$  From these bounds it follows that for some constant $\tilde \gamma,$
   \[
\opnorm{ \Phi^m_t(V_t^m) - V_t^m}^2 \leq    \tilde \gamma^2  \frac{ \varepsilon }{ m } \sqrt{\mathbf{E}\left[ \lVert V_t^m \rVert_1^2\right]} \sqrt{ 1 + \mathbf{E}\left[ \lVert V_{t-1}^m\rVert_1^2\right]}.  \qedhere
    \]
\end{proof}

\begin{proof}[Proof of Theorem \textup{\ref{ft2}}]
By exactly the same arguments used in the proof of Theorem~\ref{ft1} we arrive at the bound
 \begin{multline*}
  \opnorm{V_t^m - v_t} \leq (L_1+L_2) \left(  \sum\nolimits_{r=0}^{t-1}e^{-2\beta (t-r)\varepsilon}\opnorm{ \Phi^m_r(V_r^m) -
 V_r^m}
 ^2\right)^{1/2} \\
+ L_2\sum\nolimits_{r=0}^{t-1}e^{-\beta (t-r)\varepsilon}  \opnorm{ \Phi^m_r(V_r^m)-
 V_{r}^m}^2.
 \end{multline*}
 Bounding the error from the random compressions, we arrive at the error bound
{\small
  \begin{multline*}
  \opnorm{V_t^m - v_t} \leq  \frac{ \tilde \gamma (L_1+L_2) }{\sqrt{m}}\left(  e^{-2\beta t \varepsilon} \mathbf{E}\left[\lVert V_0^m\rVert_1^2\right] +  \varepsilon \sum\nolimits_{r=1}^{t-1}e^{-2\beta (t-r)\varepsilon}
     \sqrt{\mathbf{E}\left[ \lVert V_r^m \rVert_1^2\right]}\sqrt{  1 + \mathbf{E}\left[ \lVert V_{r-1}^m\rVert_1^2\right]} 
\right)^\frac{1}{2}\\ + \frac{ \tilde \gamma^2  L_2}{m}\sum\nolimits_{r=1}^{t-1}e^{-\beta (t-r)\varepsilon} 
   \sqrt{\mathbf{E}\left[ \lVert V_r^m \rVert_1^2\right]}\sqrt{  1 +  \mathbf{E}\left[ \lVert V_{r-1}^m\rVert_1^2\right]} . \qedhere
 \end{multline*}}
\end{proof}

\begin{proof}[Proof of Theorem \textup{\ref{ft3}}]
By an argument very similar to that in the proof of Theorem~\ref{ft1}, we arrive at the bound
 \begin{align*}
  \opnorm{V_t^m - v_t}& \leq \left(  \sum\nolimits_{r=0}^{t-1}\opnorm{ \mathcal{M}_{r+1}^t (V_r^m + \varepsilon b(Y_r^m)) -
  \mathcal{M}_{r+1}^t (V_r^m + \varepsilon b(V_r^m))}
 ^2\right)^{1/2}\\
 &\hspace{1cm} +
 \left(  \sum\nolimits_{r=0}^{t-1}\opnorm{\mathbf{E}\left[ \mathcal{M}_{r+1}^t (V_r^m + \varepsilon b(Y_r^m)) \mid V_r^m\right]-   \mathcal{M}_r^t (V_r^m) 
 }^2\right)^{1/2} \\
&\hspace{2cm}+ \sum\nolimits_{r=0}^{t-1} \opnorm{ \mathbf{E}\left[ \mathcal{M}_{r+1}^t (V_r^m + \varepsilon b(Y_r^m)) \mid V_r^m\right]-
 \mathcal{M}_{r+1}^t(V_r^m + \varepsilon b(V_r^m))},
 \end{align*}
 which, also as in that proof, is bounded above by
 \begin{multline*}
  \opnorm{V_t^m - v_t} \leq (L_1+L_2) \left( \varepsilon^2  \sum\nolimits_{r=0}^{t-1}\alpha^{2(t-r-1)}\opnorm{ b(Y_r^m) - b(V_r^m)}
 ^2\right)^{1/2}\\
 + L_2 \varepsilon^2 \sum\nolimits_{r=0}^{t-1}\alpha^{t-r} \opnorm{ b(Y_r^m) - b(V_r^m)}^2.
 \end{multline*}
 From \eqref{cntr1eps} and Lemma~\ref{opnormrep2} we find that 
 \[
 \opnorm{ b(Y_r^m) - b(V_r^m)} \leq L_1 \opnorm{ Y_r^m - V_r^m}.
 \]
 The rest of the argument proceeds exactly as in the proof of Theorem~\ref{ft1}.
\end{proof}

 \begin{proof}[Proof of Lemma \textup{\ref{taubnd}}]
  Observe that if $\tau_v^m>0,$ then condition
  \[
  \sum\nolimits_{j=\ell+1}^n  \lvert v_{\sigma_j}\rvert  \leq \frac{ m - \ell}{m} \lVert v\rVert_1
  \]
  holds for $\ell=0.$  Assume that 
    \[
  \sum\nolimits_{j=\ell}^n \lvert v_{\sigma_j}\rvert \leq \frac{ m - \ell+1}{m} \lVert v\rVert_1
  \]
  for some $\ell\leq \tau_v^m.$
 From the definition of $\tau_v^m$ and the fact that $\ell\leq \tau_v^m,$ we must also have that
  \[
 \frac{1}{m-\ell} \sum\nolimits_{j=\ell+1}^n \lvert v_{\sigma_j}\rvert <  \lvert v_{\sigma_{\ell+1}}\rvert.
  \]
Combining the last two inequalities yields
  \[
\sum\nolimits_{j=\ell+1}^n \lvert v_{\sigma_j}\rvert
\leq \frac{m-\ell}{m}   \lVert v\rVert_1.  \qedhere
\]
\end{proof}

\begin{proof}[Proof of Lemma \textup{\ref{dmcerr}}]
First we assume that, for all $j,$ $\lvert v_j+w_j\rvert \leq \lVert v+w\rVert/m.$  We will remove this assumption later.  With this assumption in place, $N_j\in\{0,1\}$ and the $\mathbf{while}$ loop in Algorithm~\ref{cr} is inactive so that
\begin{align*}
f^\ct \Phi_t(v+w) &= \sum\nolimits_{j=1}^n \bar f_j \frac{v_j+w_j}{\lvert v_j + w_j\rvert}  \frac{\lVert v+w\rVert}{m} N_j,\\
\mathbf{E}\bigl[\lvert f^\ct  \Phi_t(v+w) - f^\ct (v+w)\rvert^2\bigr]&= \frac{\lVert v+w\rVert_1^2}{m^2} \mathbf{E}\biggl[\biggl\lvert \sum\nolimits_{j=1}^n \bar f_j \frac{v_j+w_j}{\lvert v_j + w_j\rvert} \biggl( N_j- 
\frac{m\lvert v_j-w_j\rvert}{\lVert v+w\rVert_1}\biggr)\biggr\rvert^2\biggr].
\end{align*}
The random variables in the sum are independent, so the last expression becomes
\begin{align*}
\mathbf{E}\bigl[\lvert f^\ct \Phi_t(v+w) - f^\ct (v+w)\rvert^2 \bigr]
&=  \frac{\lVert v+w\rVert^2_1}{m^2} \sum\nolimits_{j=1}^n \lvert f_j\rvert^2\mathbf{E}\biggl[  \biggl\lvert N_j- 
\frac{m \lvert v_j-w_j\rvert}{\lVert v+w\rVert_1}\biggr\rvert^2\biggr]\\
&\leq \frac{\lVert v+w\rVert^2_1}{m^2} \sum\nolimits_{j=1}^n\mathbf{var}\left[ N_j\right].
\end{align*}
Since $N_j\in \{0,1\},$ the expression for the variance of $N_j$ becomes
\[
\mathbf{var}\left[N_j\right] = \mathbf{E}\left[ N_j\right] \left( 1 - \mathbf{E}\left[ N_j\right]\right)
= \frac{m \lvert v_j+w_j\rvert}{\lVert v+w\rVert_1} \left( 1-  \frac{m \lvert v_j+w_j\vert}{\lVert v+w\rVert_1} \right),
\]
so that
\[
\mathbf{E}\left[\lvert f^\ct \Phi_t(v+w) - f^\ct (v+w)\rvert^2 \right] \leq \frac{\lVert v+w\rVert^2_1}{m^2} \biggl[ m - \biggl(\frac{m}{\lVert v+w\rVert_1}\biggr)^2 \lVert v+w\rVert_2^2\biggr].
\]

Because this scheme does not depend on the ordering of the entries of $v+w$ we can assume that the entries have been ordered so that $v_j = 0$ for $j>m.$  In this case we can write
\[
\lVert v+w\rVert_2^2 = \sum\nolimits_{j=1}^m \lvert v_j+w_j\rvert^2 + \sum\nolimits_{j=m+1}^n \lvert w_j \rvert^2
 \geq \frac{1}{m} \left( \sum\nolimits_{j=1}^m \lvert v_j+w_j\rvert\right)^2,
\]
which then implies that
\begin{align*}
\mathbf{E}\left[\lvert f^\ct \Phi_t(v+w) - f^\ct (v+w)\rvert^2 \right] 
&\leq \frac{\lVert v+w\rVert^2_1}{m} \left( 1 - \frac{1}{\lVert v+w\rVert_1^2}\left( \lVert v+w\rVert_1 -  \sum\nolimits_{j=m+1}^n \lvert w_j\rvert \right)^2 \right)\\
&\leq \frac{2\lVert w\rVert_1 \lVert v+w\rVert_1}{m}. 
\end{align*}
We now remove the assumption that $\lvert v_j+w_j\rvert \leq \lVert v+w\rVert/m.$  Let $\sigma$ be a permutation of the indices of $v+w$ resulting in a vector $v_\sigma + w_\sigma$ with entries of nonincreasing magnitude. Since Algorithm~\ref{cr} preserves the largest $\tau_{v+w}^m$ entries of $v+w$ and the remaining entries, $v_{\sigma_j}+w_{\sigma_j}$ for $j>\tau_{v+w}^m$, satisfy 
\[
\lvert v_{\sigma_j}+w_{\sigma_j}\rvert \leq \frac{1}{m-\tau_{v+w}^m} \sum\nolimits_{k=\tau_{v+w}^m}^n \lvert v_{\sigma_k}+w_{\sigma_k}\rvert,
\]
we can apply the sampling error bound just proved to find that
\begin{equation*}
\opnorm{ \Phi_t(v+w) - v-w} \leq \sqrt{2}
\frac{\bigl( \sum\nolimits_{j=\tau_{v+w}^m+1}^n \lvert w_j\rvert \bigr)^\frac{1}{2}\bigl(\sum\nolimits_{j=\tau_{v+w}^m+1}^n \lvert v_j+w_j\rvert \bigr)^\frac{1}{2}}{\sqrt{m-\tau_{v+w}^m}}.
\end{equation*}
An application of Lemma~\ref{taubnd} then yields \eqref{dmccomperr2}.

%

In bounding the size of $\Phi_t^m(v+w)$ we will again assume that $\tau_{v+w}^m=0$ and 
that the entries have been ordered so that $v_j = 0$ for $j>m.$  The size of the resampled vector can be bounded by first noting that, since the $N_j$ are independent and are  in $\{0,1\},$ 
\begin{align*}
\mathbf{E}\Bigl[ \Bigl( \sum\nolimits_{j=1}^n N_j\Bigr)^2\Bigr] &= \sum\nolimits_{j=1}^n \frac{m\lvert v_j+w_j\rvert}{\lVert v+w\rVert_1}
+ 2 \sum\nolimits_{i=1}^n \sum\nolimits_{j=i+1}^n \frac{m\lvert v_i+w_i\rvert}{\lVert v+w\rVert_1}
\frac{m\lvert v_j+w_j\rvert}{\lVert v+w\rVert_1}\\
& = \sum\nolimits_{j=1}^n \left(\frac{m\lvert v_j+w_j\rvert}{\lVert v+w\rVert_1}\right)^2 
+ 2 \sum\nolimits_{i=1}^n \sum\nolimits_{j=i+1}^n \frac{m\lvert v_i+w_i\rvert}{\lVert v+w\rVert_1}
\frac{m\lvert v_j+w_j\rvert}{\lVert v+w\rVert_1}\\
&\hspace{2cm}+ \sum\nolimits_{j=1}^n \frac{m\lvert v_j+w_j\rvert}{\lVert v+w\rVert_1} - \left(\frac{m\lvert v_j+w_j\rvert}{\lVert v+w\rVert_1}\right)^2 \\
& = m^2 + \sum\nolimits_{j=1}^n \frac{m\lvert v_j+w_j\rvert}{\lVert v+w\rVert_1} - \left(\frac{m\lvert v_j+w_j\rvert}{\lVert v+w\rVert_1}\right)^2.
\end{align*}
Breaking up the last sum in this expression, we find that
\begin{align*}
\sum\nolimits_{j=1}^m \frac{m \lvert v_j+w_j\rvert}{\lVert v+w\rVert_1} -  \left(\frac{m\lvert v_j+w_j\rvert}{\lVert v+w\rVert_1}\right)^2 & \leq  m\sum\nolimits_{j=1}^m \frac{ \lvert v_j+w_j\rvert}{\lVert v+w\rVert_1} -   m \left(\sum\nolimits_{j=1}^m \frac{\lvert v_j + w_j\rvert}{\lVert v+w\rVert_1}\right)^2 \\
& \leq  m\left( 1 - \sum\nolimits_{j=1} \frac{\lvert v_j + w_j\rvert}{\lVert v+w\rVert_1}\right)
 \leq \frac{m\lVert w\rVert_1}{\lVert v+w\rVert_1}
\end{align*}
and that
\[
\sum\nolimits_{j=m+1}^n \frac{m \lvert w_j\rvert}{\lVert v+w\rVert_1} -  \left(\frac{m\lvert w_j\rvert}{\lVert v+w\rVert_1}\right)^2  \leq \frac{m\lVert w\rVert_1}{\lVert v+w\rVert_1},
\]
so that
\[
\mathbf{E}\Bigl[ \Bigl( \sum\nolimits_{j=1}^n N_j\Bigr)^2\Bigr] \leq m^2 + 2\frac{m\lVert w\rVert_1}{\lVert v+w\rVert_1}.
\]
It follows then that (at least when $\tau_{v+w}^m=0$)\[
\mathbf{E}\left[ \lVert \Phi_t^m(v+w)\rVert_1^2\right] \leq \lVert v+w\rVert_1 +  2\frac{\lVert v+w\rVert_1\lVert w\rVert_1}{m}.
\]
Writing the corresponding formula for $\tau_{v+w}^m>0$ and applying Lemma~\ref{taubnd} gives the bound in the statement of the lemma.


Finally we consider the probability of the event $\left\{\Phi_t^m(v+w)=0\right\}.$  If $\tau_{v+w}^m= 0,$ then $N_j\in\{0,1\},$ so that
 $\mathbf{P}\left[ N_j = 0\right] = 1 - m \lvert v_j+w_j\rvert/\lVert v+w\rVert_1$, and, since the $N_j$ are independent,
 \[
 \mathbf{P}\left[ N_j=0\text{ for all }j\right] = 
 \prod_{j=1}^n \left( 1 - \frac{m \lvert v_j+w_j\rvert}{\lVert v+w\rVert_1}\right)
 \leq \prod_{j\leq n, \; v_j\neq 0} \left( 1 - \frac{m \lvert v_j+w_j\rvert}{\lVert v+w\rVert_1}\right).
 \]
 The first product in the last display is easily seen to be bounded above by $e^{-m}.$
 The second product is maximized subject to the constraint
 \[
 \sum\nolimits_{j\leq n, \; v_j\neq 0}  \left( 1 - \frac{m \lvert v_j+w_j\rvert}{\lVert v+w\rVert_1}\right)
 \leq \frac{ m \lVert w\rVert_1}{\lVert v+w\rVert_1}
 \]
 when the terms in the product are all equal, in which case we get
 \[
 \mathbf{P}\left[ N_j=0\text{ for all }j\right] 
 \leq \left( \frac{\lVert w\rVert_1}{\lVert v+w\rVert_1}\right)^m.  \qedhere
 \] 
 \end{proof}

\end{document}